\documentclass[a4paper,11pt]{amsart}

\usepackage{amsmath,amsthm,amssymb}
\usepackage[mathscr]{eucal}
\usepackage{cite}
\usepackage[bookmarks,bookmarksnumbered]{hyperref}

\numberwithin{equation}{section}

\theoremstyle{plain}
\newtheorem{theorem}{Theorem}[section]

\newtheorem{lemma}[theorem]{Lemma}

\newtheorem{corollary}[theorem]{Corollary}
\theoremstyle{definition}
\newtheorem{definition}[theorem]{Definition}

\newtheorem{notations}[theorem]{Notations}
\newtheorem{remark}[theorem]{Remark}

\newcommand{\C}{\mathbb{C}}
\newcommand{\actson}{\curvearrowright}
\newcommand{\ot}{\otimes}
\newcommand{\N}{\mathbb{N}}
\newcommand{\T}{\mathbb{T}}

\newcommand{\cF}{\mathcal{F}}
\newcommand{\recht}{\rightarrow}
\newcommand{\eps}{\varepsilon}
\newcommand{\ovt}{\mathbin{\overline{\otimes}}}
\newcommand{\om}{\omega}
\newcommand{\cP}{\mathcal{P}}
\newcommand{\cK}{\mathcal{K}}
\newcommand{\cM}{\mathcal{M}}
\newcommand{\F}{\mathbb{F}}

\begin{document}

\title [Cartan subalgebras of amalgamated free product II$_1$ factors]
{Cartan subalgebras of amalgamated free product \\ II$_1$ factors\\ Sous-alg\`{e}bres de Cartan de produit amalgam\'{e} de facteurs de type II$_1$}

\author [Adrian Ioana]
{By Adrian Ioana\\  \\ With an appendix by Adrian Ioana and Stefaan Vaes\\ \\}

\thanks{Mathematics Department, University of California, San Diego, CA 90095-1555 (United States). aioana@ucsd.edu. Supported by  NSF  Grant DMS \#1161047, NSF Career Grant DMS \#1253402, and a Sloan Foundation Fellowship.}

\dedicatory{\large{\bf {\it Dedicated to Sorin Popa}}}


\begin{abstract} 

We study  Cartan subalgebras in the context of amalgamated free product II$_1$ factors and obtain several uniqueness and non-existence results.  We prove that if $\Gamma$ belongs to a large class of amalgamated free product groups (which contains the free product of any two infinite groups) then any II$_1$ factor $L^{\infty}(X)\rtimes\Gamma$ arising from a free ergodic probability measure preserving action of  $\Gamma$ has a unique Cartan subalgebra, up to unitary conjugacy. We also prove that if $\mathcal R=\mathcal R_1*\mathcal R_2$  is the free product of any two non-hyperfinite countable ergodic probability measure preserving equivalence relations, then the II$_1$ factor $L(\mathcal R)$ has a unique Cartan subalgebra, up to unitary conjugacy. Finally, we show that the free product $M=M_1*M_2$ of any two II$_1$ factors  does not have a Cartan subalgebra.
More generally, we prove that if $A\subset M$ is a diffuse amenable von Neumann subalgebra and $P\subset M$ denotes the algebra generated by its normalizer,  then  either $P$ is amenable, or a corner of $P$ can be unitarily conjugate into $M_1$ or $M_2$.
\vskip 0.1in
\noindent
{R\'{E}SUM\'{E}.}
Nous \'{e}tudions les sous-alg\`{e}bres de Cartan dans le contexte du produit amalgam\'{e} de facteurs de type II$_1$  et nous obtenons plusieurs r\'{e}sultats d'unicit\'{e} et de non-existence. Nous d\'{e}montrons que, si $\Gamma$ appartient \`{a} une grande classe de produits amalgam\'{e}s de groupes (qui contient le produit libre de deux groupes infinis), alors tout facteur de type II$_1$  associ\'{e} \`{a} une action libre ergodique de $\Gamma$ a une sous-alg\`{e}bre de Cartan unique, \`{a} conjugaison unitaire. Nous d\'{e}montrons aussi que, si $\mathcal R=\mathcal R_1*\mathcal R_2$  est le produit libre de toute relation d'\'{e}quivalence ergodique non-hyperfinie d\'{e}nombrable, alors le facteur de type II$_1$  $L(\mathcal R)$ a une sous-alg\`{e}bre de Cartan unique, \`{a} conjugaison unitaire. Enfin, nous d\'{e}montrons que le produit libre $M = M_1 * M_2$  de tout facteur de type II$_1$ n'a pas de sous-alg\`{e}bre de Cartan. Plus g\'{e}n\'{e}ralement, nous d\'{e}montrons que, si $A\subset M$  est une sous-alg\`{e}bre de von Neumann amenable et non-atomique et si $P\subset M$ d\'{e}signe l'alg\`{e}bre engendr\'{e}e par son normalisateur, alors soit $P$  est amenable, soit un coin de $P$ peut \^{e}tre unitairement conjugu\'{e} dans $M_1$  ou  $M_2$.

 \end{abstract}

\maketitle

\section {Introduction} 
A {\it Cartan subalgebra} of a II$_1$ factor $M$ is a maximal abelian von Neumann subalgebra $A$ whose normalizer generates  $M$.
The study of Cartan subalgebras
plays a central role in the classification of II$_1$ factors arising from probability measure preserving (pmp)  actions. If $\Gamma\curvearrowright (X,\mu)$ is a free ergodic pmp action of a countable group $\Gamma$, then the {\it group measure space} II$_1$ factor $L^{\infty}(X)\rtimes\Gamma$ \cite{MvN36} contains $L^{\infty}(X)$  as a Cartan subalgebra. In order to classify $L^{\infty}(X)\rtimes\Gamma$ in terms of the action $\Gamma\curvearrowright X$, one would ideally aim to show that $L^{\infty}(X)$ is its unique Cartan subalgebra (up to conjugation by an automorphism). Proving that certain classes of group measure space  II$_1$ factors  have a unique Cartan subalgebra  is useful because it  reduces their classification, up to isomorphism, to the classification of the corresponding actions, up to orbit equivalence.  Indeed, following \cite{Si55, FM77}, two free ergodic pmp actions $\Gamma \curvearrowright X$ and $\Lambda\curvearrowright Y$ are {\it orbit equivalent} if and only if there exists an isomorphism $\theta:L^{\infty}(X)\rtimes\Gamma\rightarrow L^{\infty}(Y)\rtimes\Lambda$ such that $\theta(L^{\infty}(X))=L^{\infty}(Y)$.

In the case of II$_1$ factors  coming from actions of amenable groups, both the classification and uniqueness of Cartan problems have been completely settled since the early 1980's.
A celebrated theorem of A. Connes \cite{Co76} asserts that all II$_1$ factors arising from free ergodic pmp actions of infinite amenable groups are isomorphic to the hyperfinite II$_1$ factor, $R$. Additionally, \cite{CFW81} shows that any two Cartan subalgebras of $R$ are conjugate by an automorphism of $R$. 

For a long time, however, the questions of classification and 
uniqueness of Cartan subalgebras for II$_1$ factors associated with actions of non-amenable groups, were considered intractable.
During the last decade, S. Popa's {\it deformation/rigidity} theory has led to spectacular progress in the classification of  group measure space II$_1$ factors (see the surveys \cite{Po07,Va10a,Io12}). This was in part made possible by several results providing classes of group measure space II$_1$ factors  that have a unique Cartan subalgebra, up to unitary conjugacy. The first such classes were obtained by N. Ozawa and S. Popa in their breakthrough work \cite{OP07,OP08}. They showed that II$_1$ factors $L^{\infty}(X)\rtimes\Gamma$ associated with free ergodic {\it profinite} actions of free groups $\Gamma=\mathbb F_n$ and their direct products $\Gamma=\mathbb F_{n_1}\times\mathbb F_{n_2}\times...\times\mathbb F_{n_k}$ have a unique Cartan subalgebra, up to unitary conjugacy. Recently, this result has been extended to profinite actions of hyperbolic groups \cite{CS11} and of direct products of hyperbolic groups \cite{CSU11}. The proofs of these results rely both on the fact that free groups (and, more generally, 
hyperbolic groups, see \cite{Oz07},\cite{Oz10}) are {\it weakly amenable}  and that the actions are profinite.

In a very recent breakthrough, S. Popa and S. Vaes succeeded in removing the profiniteness assumption on the action and obtained wide-ranging unique Cartan subalgebra results. They proved that
if $\Gamma$ is either a weakly amenable group with $\beta_1^{(2)}(\Gamma)>0$ \cite{PV11} or a hyperbolic group \cite{PV12} (or 
a direct product of groups in one of these classes), then II$_1$ factors $L^{\infty}(X)\rtimes\Gamma$ arising from {\it arbitrary} free ergodic pmp actions of $\Gamma$ have a unique Cartan subalgebra, up to unitary conjugacy. 
Following \cite[Definition 1.4]{PV11}, such groups $\Gamma$, whose every  action gives rise to a II$_1$ factor with a unique Cartan subalgebra,  are called $\mathcal C$-{\it rigid} (Cartan rigid).

 In this paper we study Cartan subalgebras of tracial amalgamated free product von Neumann algebras $M=M_1*_{B}M_2$ (see \cite{Po93,VDN92} for the definition). 
 Our methods are best suited to the case when $M=L^{\infty}(X)\rtimes\Gamma$ comes from an action of an amalgamated free product group $\Gamma=\Gamma_1*_{\Lambda}\Gamma_2$. In this context, by imposing that the inclusion $\Lambda<\Gamma$ satisfies a weak malnormality condition \cite{PV09}, we prove that $L^{\infty}(X)$ is the unique Cartan subalgebra of $M$, up to unitary conjugacy, for {\it any} free ergodic pmp action $\Gamma\curvearrowright X$.
   
\begin{theorem}\label{main} Let $\Gamma=\Gamma_1*_{\Lambda}\Gamma_2$ be an amalgamated free product group such that $[\Gamma_1:\Lambda]\geqslant 2$ and  $[\Gamma_2:\Lambda]\geqslant 3$.  Assume that there exist $g_1,g_2,...,g_n\in\Gamma$ such that $\cap_{i=1}^ng_i\Lambda g_i^{-1}$ is finite. 
 Let $\Gamma\curvearrowright (X,\mu)$ be any free ergodic pmp action of $\Gamma$ on a standard probability space $(X,\mu)$. 
 
 Then the II$_1$ factor $M=L^{\infty}(X)\rtimes\Gamma$ has a unique Cartan subalgebra, up to unitary conjugacy.

Moreover, the same holds if $\Gamma$ is replaced with a direct product of finitely many such groups $\Gamma$.
\end{theorem}

This result provides  the first examples of $\mathcal C$-rigid groups $\Gamma$ that are not weakly amenable (take e.g. $\Gamma=SL_3(\mathbb Z)*\Sigma$, where $\Sigma$ is any non-trivial countable group).

Theorem \ref{main} generalizes and strengthens the main result of \cite{PV09}. Indeed, in the above setting, assume further that $\Lambda$ is amenable and that $\Gamma_2$ contains  either a non-amenable subgroup with the relative property (T) or two non-amenable commuting subgroups.
\cite[Theorem 1.1]{PV09} then asserts that $M$ has a unique {\it group measure space} Cartan subalgebra.

Theorem \ref{main} provides strong supporting evidence for a general conjecture which predicts 
that any group $\Gamma$ with  positive first $\ell^2$-Betti number, $\beta_1^{(2)}(\Gamma)>0$, is $\mathcal C$-rigid.  Thus, it implies that the free product $\Gamma=\Gamma_1*\Gamma_2$ of any two countable groups  satisfying $|\Gamma_1|\geqslant 2$ and $|\Gamma_2|\geqslant 3$, is $\mathcal C$-rigid.

Recently, there have been several results offering positive evidence for this conjecture. Firstly, it was shown in \cite{PV09} that if $\Gamma=\Gamma_1*\Gamma_2$, where $\Gamma_1$ is a property (T) group and $\Gamma_2$ is a non-trivial group, then any II$_1$ factor $L^{\infty}(X)\rtimes\Gamma$ associated with a free ergodic pmp action of $\Gamma$ has a unique group measure space Cartan subalgebra, up to unitary conjugacy (see also\cite{FV10,HPV10}). Secondly, the same has been proven in \cite{CP10} under the assumption that $\beta_1^{(2)}(\Gamma)>0$ and $\Gamma$ admits a non-amenable subgroup with the relative property (T). For a common generalization of the last two results, see \cite{Va10b}. Thirdly, we proved  that if $\beta_1^{(2)}(\Gamma)>0$, then $L^{\infty}(X)\rtimes\Gamma$ has a unique group measure space Cartan subalgebra whenever the action $\Gamma\curvearrowright (X,\mu)$ is either rigid \cite{Io11a} or compact \cite{Io11b}.
As already mentioned above, the conjecture has been very recently established in full generality for weakly amenable groups $\Gamma$ with $\beta_1^{(2)}(\Gamma)>0$ in \cite{PV11}.

As a consequence of Theorem \ref{main} we obtain a new family of W$^*$-superrigid actions.
Recall that a free ergodic pmp action $\Gamma\curvearrowright (X,\mu)$ is called W$^*$-{\it superrigid} if whenever $L^{\infty}(X)\rtimes\Gamma\cong L^{\infty}(Y)\rtimes\Lambda$, for some free ergodic pmp action $\Lambda\curvearrowright (Y,\nu)$, the groups $\Gamma$ and $\Lambda$ are isomorphic, and their actions are conjugate. The existence of virtually W$^*$-superrigid actions was proven in \cite{Pe09}. The first concrete families of W$^*$-superrigid actions were found in \cite{PV09} where it was shown for instance that Bernoulli actions of many amalgamated free product groups have this property. 
In \cite{Io10} we proved that Bernoulli actions of  icc property (T) groups are W$^*$-superrigid.
By combining Theorem \ref{main} with the cocycle superrigidity theorem \cite{Po06a}  we derive the following. 

\begin{corollary}\label{super} Let $\Gamma=\Gamma_1*_{\Lambda}\Gamma_{2}$ and $\Gamma'=\Gamma_1'*_{\Lambda'}\Gamma_2'$ be two amalgamated free product groups satisfying the hypothesis of Theorem \ref{main}. Denote $G=\Gamma\times\Gamma'$.

Then any free action of $G$ which is a quotient of the Bernoulli action $G\curvearrowright [0,1]^{G}$ is W$^*$-superrigid.
\end{corollary}

Next, we return to the study of Cartan subalgebras  of general amalgamated free product II$_1$ factors $M=M_1*_{B}M_2$. Assuming that $B$ is amenable and $M$ satisfies some rather mild conditions, we prove that any Cartan subalgebra $A\subset M$ has a corner which embeds into $B$, in the sense  of S. Popa's {\it intertwining-by-bimodules} \cite{Po03}  (see Theorem \ref{corner}). This condition, written in symbols as $A\prec_{M}B$, roughly means that $A$ can be conjugated into $B$ via a unitary element from $M$.

\begin{theorem}\label{cartan} Let $(M_1,\tau_1)$ and $(M_2,\tau_2)$ be two tracial von Neumann algebras with a common amenable von Neumann subalgebra $B$ such that ${\tau_1}_{|B}={\tau_2}_{|B}$. Assume that $M=M_1*_{B}M_2$ is a factor and that  either:

\begin{enumerate} 
\item $M_1$ and $M_2$ have no amenable direct summands, or
\item  $M$ does not have property $\Gamma$ and $pM_1p\not=pBp\not=pM_2p$, for any non-zero projection $p\in B$.
\end{enumerate}
If $A\subset M$ is a Cartan subalgebra, then $A\prec_{M}B$. 

\end{theorem} 
Recall that a {\it tracial von Neumann algebra} $(M,\tau)$ is a von Neumann algebra $M$ endowed with a normal faithful tracial state $\tau$. As usual,  we denote by  $\|x\|_2=\tau(x^*x)^{\frac{1}{2}}$  the induced Hilbert norm on $M$. Recall also that a II$_1$ factor $M$ has {\it property $\Gamma$}  if there exists a sequence $u_n\in M$ of unitary elements such that $\tau(u_n)=0$, for all $n$, and $\|u_nx-xu_n\|_2\rightarrow 0$, for every $x\in M$ \cite{MvN43}.

Theorem \ref{cartan} has two interesting  applications. 

Firstly, it yields a classification result for von Neumann algebras $L(\mathcal R)$ \cite{FM77} arising from 
the {\it free product}  $\mathcal R=\mathcal R_1*\mathcal R_2$  of two  equivalence relations (see \cite{Ga99} for the definition). For instance, it implies that if $\mathcal R_1$, $\mathcal R_2$ are ergodic and non-hyperfinite, then any countable pmp equivalence relation $\mathcal S$ such that $L(\mathcal S)\cong L(\mathcal R)$ is necessarily isomorphic to $\mathcal R$.
 More generally, we have

\begin{corollary}\label{equirel} Let $\mathcal R$ be a countable ergodic pmp equivalence relation on a standard probability space $(X,\mu)$. Assume that $\mathcal R=\mathcal R_1*\mathcal R_2$, for two equivalence relations $\mathcal R_1$ and $\mathcal R_2$ on $(X,\mu)$. Additionally, suppose that either:
\begin{enumerate}
\item ${\mathcal R_1}_{|Y}$ and ${\mathcal R_2}_{|Y}$ are not hyperfinite, for any Borel set $Y\subset X$ with $\mu(Y)>0$, or
\item $\mathcal R$ is strongly ergodic, and ${\mathcal R_1}$ and ${\mathcal R_2}$ have infinite orbits, almost everywhere.
\end{enumerate}

Then $L^{\infty}(X)$ is the unique Cartan subalgebra of $L(\mathcal R)$, up to unitary conjugacy.

Thus, if $L(\mathcal R)\cong L(\mathcal S)$, for any ergodic countable pmp equivalence relation $\mathcal S$, then $\mathcal R\cong\mathcal S$.
 
\end{corollary}

Here, $\mathcal R_{|Y}:=\mathcal R\cap (Y\times Y)$ denotes the restriction of $\mathcal R$ to  $Y$. Recall that an ergodic countable pmp equivalence relation $\mathcal R$ on a probability space $(X,\mu)$ is called {\it strongly ergodic} if there does not exist a sequence of Borel sets $Y_n\subset X$ such that $\mu(Y_n)=\frac{1}{2}$, for all $n$, and $\mu(\theta(Y_n)\Delta Y_n)\rightarrow 0$, for any Borel automorphism $\theta$ of $X$ satisfying $(\theta(x),x)\in\mathcal R$, for almost every $x\in X$.

Secondly, Theorem \ref{cartan} allows us to show that the free product of any two diffuse tracial von Neumann algebras does not have a Cartan subalgebra.
By using the notion of free entropy for von Neumann algebras,   D. Voiculescu proved that the free group factors $L(\mathbb F_n)$  do not have Cartan subalgebras \cite{Vo95}. This result was extended in \cite[Lemma 3.7]{Ju05} to show that the free product $M=M_1*M_2$ of any two diffuse tracial von Neumann algebras $(M_1,\tau_1)$ and $(M_2,\tau_2)$, which are embeddable into $R^{\omega}$,  does not have a Cartan subalgebra. Here we prove this result without requiring that $M_1$ and $M_2$ embed into $R^{\omega}$. More generally, we have

\begin{corollary}\label{free}
Let $(M_1,\tau_1)$, $(M_2,\tau_2)$ be tracial von Neumann algebras satisfying $M_1\not=\mathbb C1\not= M_2$ and $\text{dim}(M_1)+\text{dim}(M_2)\geqslant 5$.

Then their free product $M=M_1*M_2$ does not have a Cartan subalgebra.
\end{corollary}
Corollary \ref{free}  shows that if $M_1\not=\mathbb C1\not=M_2$ and $(\text{dim}(M_1),\text{dim}(M_2))\not=(2,2)$, then $M$ has no Cartan subalgebra. On the other hand, if $\text{dim}(M_1)=\text{dim}(M_2)=2$, then $M$ is of type I (see \cite[Theorem 1.1]{Dy93})  and therefore has a Cartan subalgebra.

So far, our results only apply to Cartan subalgebras of amalgamated free product von Neumann algebras $M=M_1*_{B}M_2$. From now on, we more generally study, in the spirit of \cite{OP07} and \cite{PV11}, normalizers of arbitrary diffuse amenable von Neumann subalgebras $A\subset M$.  
Recall that the {\it normalizer} of $A$ in $M$, denoted $\mathcal N_{M}(A)$, is the group of unitaries $u\in M$ such that $uAu^*=A$. 
Assuming that the normalizer of $A$ satisfies a certain spectral gap condition, we prove the following dichotomy: either a corner of $A$ embeds into $M_i$, for some $i\in\{1,2\}$, or the algebra generated by the normalizer of $A$ is amenable relative to $B$. More precisely, we show

\begin{theorem}\label{general} Let $(M_1,\tau_1)$ and $(M_2,\tau_2)$ be two tracial von Neumann algebras with a common von Neumann subalgebra $B$ such that ${\tau_1}_{|B}={\tau_2}_{|B}$.
Let $M=M_1*_{B}M_2$ and $A\subset pMp$ be a von Neumann subalgebra which is amenable relative to $B$, for some projection $p\in M$. Denote by $P=\mathcal N_{pMp}(A)''$ the von Neumann algebra generated by the normalizer of $A$ in $pMp$. Assume that $P'\cap (pMp)^{\omega}=\mathbb C1$, for a free ultrafilter $\omega$ on $\mathbb N$.

Then one of the following conditions holds true:

\begin{enumerate}
\item $A\prec_{M}B$.
\item $P\prec_{M}M_i$, for some $i\in\{1,2\}$.
\item$P$ is amenable relative to $B$.
\end{enumerate}
\end{theorem}

For the  definition of {\it relative amenability}, see Section \ref{relativeamen}. For now, note that if  $B$ is amenable, then $P$ is amenable relative to $B$ if and only if $P$ is amenable.

We believe that Theorem \ref{general} should hold without assuming that $P'\cap M^{\omega}=\mathbb C1$, but we were unable to prove this for general $B$. Nevertheless, in the case $B=\mathbb C$, a detailed analysis of the relative commutant $P'\cap M^{\omega}$ (see Section 6) enabled us to show that the condition $P'\cap M^{\omega}=\mathbb C1$ is indeed redundant.

\begin{corollary}\label{corgeneral} Let $(M_1,\tau_1)$, $(M_2,\tau_2)$ be two tracial von Neumann algebras. Let $M=M_1*M_2$ and $A\subset M$ be a diffuse amenable von Neumann subalgebra. Denote  $P=\mathcal N_{M}(A)''$.

Then either $P\prec_{M}M_i$, for some $i\in\{1,2\}$, or $P$ is amenable.
\end{corollary}

For a more precise version of this result in the case $M_1$ and $M_2$ are II$_1$ factors, see Corollary \ref{corgeneral2}.

Finally, we present a new class of strongly solid von Neumann algebras.
Recall that a von Neumann algebra $M$ is called {\it strongly solid} if  $\mathcal N_{M}(A)''$ is amenable, whenever $A\subset M$ is a diffuse amenable von Neumann subalgebra \cite{OP07}. N. Ozawa and S. Popa proved in \cite{OP07} that the free group factors $L(\mathbb F_n)$ are strongly solid. More generally, I. Chifan and T. Sinclair recently showed that the von Neumann algebra $L(\Gamma)$ of any icc hyperbolic group $\Gamma$ is strongly solid \cite{CS11}.
 
The class of strongly solid von Neumann algebras is not closed under taking amalgamated free products. For instance, if $\mathbb F_2\curvearrowright (X,\mu)$ is a pmp action on a non-atomic probability space $(X,\mu)$, then the group measure space algebra $L^{\infty}(X)\rtimes\mathbb F_2=(L^{\infty}(X)\rtimes\mathbb Z)*_{L^{\infty}(X)}(L^{\infty}(X)\rtimes\mathbb Z)$ is not strongly solid, although the  algebras involved in its amalgamated free product decomposition are amenable and hence strongly solid. 

However, as an application of Theorem \ref{general}, we prove that the class of strongly solid von Neumann algebras is closed under free products (Corollary \ref{stro1})
More generally, we show that if $M_1$ and $M_2$ are strongly solid von Neumann algebras, then the amalgamated free product $M=M_1*_{B}M_2$ is strongly solid, provided that the inclusions $B\subset M_1$ and $B\subset M_2$ are {\it mixing}, and $B$ is amenable.

\begin{theorem}\label{strongly}
Let $(M_1,\tau_1)$ and  $(M_2,\tau_2)$ be strongly solid von Neumann algebras with a common amenable von Neumann subalgebra $B$ such that ${\tau_1}_{|B}={\tau_2}_{|B}$.   Assume that the inclusions $B\subset M_1$ and $B\subset M_2$ are mixing. Denote $M=M_1*_{B}M_2$.

Then $M$ is strongly solid.
\end{theorem}

For the definition of mixing inclusions of von Neumann algebras, see Section \ref{sstrongly}. For now, let us point out that the inclusion $B\subset M$ is mixing whenever the $B$-$B$ bimodule $L^2(M)\ominus L^2(B)$ is contained in a multiple of the coarse $B$-$B$ bimodule $L^2(B)\otimes L^2(B)$.

Theorem \ref{strongly} implies that if $M_1,M_2,...,M_n$ are amenable von Neumann algebras with a common von Neumann subalgebra $B$ such that the inclusions $B\subset M_1,B\subset M_2,...,B\subset M_n$ are mixing, then $M=M_1*_{B}M_2*_{B}...*_{B}M_n$ is strongly solid (Corollary \ref{stro2}).

\vskip 0.1in
{\bf Comments on the proofs.} The most general type of result that we prove is Theorem \ref{general}. Let us say a few words about its proof.
Assume therefore that $A$ is a von Neumann subalgebra of an amalgamated free product von Neumann algebra $M=M_1*_{B}M_2$ that is amenable relative to $B$. We denote  $P=\mathcal N_{M}(A)''$ and assume that $P'\cap M^{\omega}=\mathbb C1$. 

Our goal is to show that either $A\prec_{M}M_i$, for some $i\in\{1,2\}$, or $P$ is amenable relative to $B$. This is enough to deduce the conclusion of Theorem \ref{general}, because by \cite[Theorem 1.1]{IPP05} the first case implies that either $A\prec_{M}B$ or $P\prec_{M}M_i$, for some $i\in\{1,2\}$.

The strategy of proof is motivated by a beautiful recent dichotomy theorem due to S. Popa and S. Vaes. To state the particular case of \cite[Theorem 1.6]{PV11} that will be useful to us,  let $\mathbb F_2\curvearrowright (N,\tau)$  be a trace preserving action of the free group $\mathbb F_2$ on a tracial von Neumann algebra $(N,\tau)$. Denote $\tilde M=N\rtimes\mathbb F_2$. Given a von Neumann subalgebra $D\subset \tilde M$ that is amenable relative to $N$, it is shown in \cite{PV11} that either $D\prec_{\tilde M}N$ or $\mathcal N_{\tilde M}(D)''$ is amenable relative to $N$.

In order to apply this result in our context, we use the {\it free malleable deformation} introduced in \cite{IPP05}. More precisely, define $\tilde M=M*_B(B\bar{\otimes}L(\mathbb F_2))$. Then $M\subset\tilde M$ and one constructs a 1-parameter group of automorphisms $\{\theta_t\}_{t\in\mathbb R}$ of $\tilde M$ as follows. Let $u_1,u_2\in L(\mathbb F_2)$ be the canonical generating unitaries and $h_2,h_2\in L(\mathbb F_2)$ be hermitian elements such that $u_1=\exp(ih_1)$ and $u_2=\exp(ih_2)$. For $t\in\mathbb R$, define the unitary elements $u_1^t=\exp(ith_1)$ and $u_2^t=\exp(ith_2)$. Then there exists an automorphism $\theta_t$ of $\tilde M$ such that $${\theta_t}_{|M_1}=\text{Ad}(u_1^t)_{|M_1},\;\;\;\; {\theta_t}_{|M_2}=\text{Ad}(u_2^t)_{|M_2}\;\;\;\;\text{and}\;\;\;\;{\theta_t}_{|L(\mathbb F_2)}=\text{id}_{L(\mathbb F_2)}.$$

The starting point of the proof is the key observation that $\tilde M$ can be written as $\tilde M=N\rtimes\mathbb F_2$, where $N$ is the von Neumann subalgebra of $\tilde M$ generated by $\{u_gMu_g^*\}_{g\in\mathbb F_2}$ and $\mathbb F_2$ acts on $N$ via conjugation with $\{u_g\}_{g\in\mathbb F_2}$.  

Now, let $t\in (0,1)$ and notice that $\theta_t(P)\subset\mathcal N_{\tilde M}(\theta_t(A))''$.  Since $A$ is amenable relative to $B$ and $\theta_t(B)=B\subset N$, we deduce that $\theta_t(A)$ is amenable relative to $N$.
By applying the dichotomy of \cite{PV11}, we conclude that either $\theta_t(A)\prec_{\tilde M}N$ or $\theta_t(P)$ is amenable relative to $N$. Since $t\in (0,1)$ is arbitrary, we are therefore in one of the following two cases:
\begin{enumerate}
\item $\theta_t(A)\prec_{\tilde M}N$, for some $t\in (0,1)$.
\vskip 0.1in
\item $\theta_t(P)$ is amenable relative to $N$, for any $t\in (0,1)$.
\end{enumerate}

The core of the paper consists of analyzing what can be said about the von Neumann subalgebras $A$ and $P$ of $M$ which satisfy these conditions. Note that since $\theta_1(M)\subset N$, these conditions are trivially satisfied for any subalgebra $A\subset M$ when $t=1$.

Thus, we prove in Section 3 that if (1) holds then $A\prec_{M}M_i$, for some $i\in\{1,2\}$. The proof of this result has two main ingredients. To explain what they are, assume by contradiction that $A\nprec_{M}M_i$, for  any $i\in\{1,2\}$. Then \cite[Theorem 3.1]{IPP05}  provides a sequence of unitary elements $u_k\in A$ which are asymptotically (i.e., as $k\rightarrow\infty$) supported  on words in $M_1\ominus B$ and $M_2\ominus B$ of length $\geqslant\ell$, for every $\ell\geqslant 1$. In the second part of the proof, we use a  calculation from the theory of random walks on groups to derive that the unitaries $\theta_t(u_k)\in\theta_t(A)$ are asymptotically perpendicular to $aNb$, for any $a,b\in\tilde M$. This contradicts the assumption that (1) holds.

In Sections 4 and 5 we investigate which von Neumann subalgebras $P\subset M$ satisfy (2). 

 Our first result in this direction applies in the particular case when $P=M$. More precisely, we prove that if  (2) holds for $P=M$, then either $M_1$ or $M_2$ must have a amenable direct summand (see Theorem \ref{amena}). In combination with the above, it follows that if $A\subset M$ is a Cartan subalgebra, then either $A\prec_{M}M_i$ or $M_i$ has an amenable direct summand, for some $i\in\{1,2\}$. This readily implies Theorem \ref{cartan} and Corollary \ref{equirel} under the first sets of conditions.

In general, however, we are only able to  treat von Neumann subalgebras $P\subset M$ which in addition to satisfying (2) also verify the spectral gap condition $P'\cap M^{\omega}=\mathbb C1$. Under these assumptions, we prove that either $P\prec_{M}M_i$, for some $i\in\{1,2\}$, or $P$ is amenable relative to $B$ (see Theorem \ref{relamen}). It is clear that  this result completes the proof of Theorem \ref{general}.

Note that if $M=M_1*M_2$ is a plain free product and $P'\cap M^{\omega}$ is diffuse, then we can show that either $P\prec_{M}M_i$, for some $i\in\{1,2\}$, or $P$ has an amenable direct 
summand (see Theorem \ref{afpgamma}). It follows
 that, in the case of plain free products, Theorem \ref{general} holds without the assumption $P'\cap M^{\omega}=\mathbb C1$. 
This explains why Corollary \ref{corgeneral} also  does not require this assumption.
 
\vskip 0.1in
{\bf Organization of the paper}.
Besides the introduction this paper has eight other sections. In Section 2 we recall the tools that are needed in the sequel as well as establish some new results. For instance, we prove that if  $A\subset M=M_1*_{B}M_2$ is a von Neumann subalgebra that is amenable relative to $M_1$, then either $A$ is amenable relative to $B$, or a corner of $\mathcal N_{M}(A)''$ embeds into $M_1$ (see Corollary \ref{211}).
We have described above the contents of Section 3-5. In Section 6, motivated by the hypothesis of Theorem \ref{general}, we study  the relative commutant $P'\cap M^{\omega}$, where $P$ is a von Neumann subalgebra of an amalgamated free product algebra $M=M_1*_{B}M_2$. Finally, Sections 7-9 are devoted to the proofs of the results stated in the introduction.

\vskip 0.1in
{\bf Dedication}. This paper is dedicated to Sorin Popa, with great affection and admiration.

\vskip 0.1in
{\bf Acknowledgements}. I am very grateful to R\'{e}mi Boutonnet, Ionut Chifan, Cyril Houdayer, Yoshimichi Ueda and Stefaan Vaes for many helpful comments on the first version of this paper. 
 In particular, I would like to thank Cyril and Stefaan for pointing out errors in the initial proofs of Lemmas \ref{ultrapower} and \ref{24}, respectively, and Yoshimichi for pointing out that Corollary \ref{free} holds in the present generality. Finally, I would like to thank the two referees whose comments helped improve the exposition.

\vskip 0.1in
{\bf Added in the proof}. Since the first version of this paper has been posted on the arXiv, there have been some related developments. 
 Firstly, R. Boutonnet, C. Houdayer and S. Raum generalized some of our results  to the non-tracial setting \cite{BHR12}. In particular, they extended Corollary \ref{free} to arbitrary von Neumann algebras. More recently, S. Vaes was able to remove the spectral gap assumption $P'\cap M^{\omega}=\mathbb C1$  from Theorem \ref{general}. This allowed him for instance to prove an improved, optimal version of Corollary \ref{equirel}, where one only assumes that almost every class of $\mathcal R_1$ has at least $2$ elements and almost every class of $\mathcal R_2$ has at least $3$ elements  \cite{Va13}.
 
 \vskip 0.1in
 {\bf Correction}. 
 Theorem 2.5 from the initial version of this paper (posted on the arXiv in July 2012) falsely asserted that the notions of spectral gap and $w$-spectral gap were equivalent for arbitrary inclusions of tracial von Neumann algebras (see the Appendix for the definitions). I am very grateful to Cyril Houdayer for pointing out this mistake. 
 The false assertion was only used in the proof of Theorem 5.1 to deduce spectral gap for an inclusion $A\subset pMp$ that was originally assumed to have $w$-spectral gap. However, the original proof of Theorem \ref{relamen} still works if the inclusion $A\subset pMp$ does not not necessarily have spectral gap, but instead satisfies a certain weaker technical property.
In the Appendix, written jointly with Stefaan Vaes, we prove that this technical property, which, a priori, sits in between spectral gap and $w$-spectral gap, is in fact equivalent to $w$-spectral gap. 

\section{preliminaries}
We start by recalling some of the terminology that we use in this paper. 

Throughout  we work with {\it  tracial} von Neumann algebras $(M,\tau)$, i.e. von Neumann algebras $M$  endowed with a faithful, normal, tracial state $\tau$. We assume that $M$ is {\it separable}, unless it is an ultraproduct algebra or we specify otherwise.

 We denote by $\mathcal Z(M)$ the {\it center} of $M$, by $\mathcal U(M)$ the {\it group of unitaries} of $M$ and by $(M)_1$ the {\it unit ball} of $M$. We say that a von Neumann subalgebra $A\subset M$ is {\it regular} in $M$ if $\mathcal N_{M}(A)''=M$.

For a free ultrafilter $\omega$ on $\mathbb N$, the {\it ultraproduct} algebra $M^{\omega}$ is defined as the quotient  $\ell^{\infty}(\mathbb N,M)/\mathcal I$, where $\mathcal I\subset\ell^{\infty}(\mathbb N,M)$ is the closed ideal of $x=(x_n)_n$ such that $\lim_{n\rightarrow\omega}\|x_n\|_2=0$. As it turns out, $M^{\omega}$ is a tracial von Neumann algebra, with its canonical trace given by $\tau_{\omega}((x_n)_n)=\lim_{n\rightarrow\omega}\tau(x_n)$.

If $M$ and $N$ are tracial von Neumann algebras, then an $M$-$N$ {\it bimodule} is a Hilbert space $\mathcal H$ endowed with commuting normal $*$-homomorphisms $\pi:M\rightarrow \mathbb B(\mathcal H)$ and $\rho:N^{op}\rightarrow \mathbb B(\mathcal H)$. For $x\in M,y\in N$ and $\xi\in\mathcal H$ we denote $x\xi y=\pi(x)\rho(y)(\xi)$.

Next, let $M,N,P$ be tracial von Neumann algebras. Let $\mathcal H$ and $\mathcal K$ be $M$-$N$ and $N$-$P$ bimodules.  Let $\mathcal K_0$ be vector subspace of vectors $\eta\in\mathcal K$ that are left bounded, i.e. for which there exists $c>0$ such that $\|x\eta\|\leqslant c\|x\|_2$, for all $x\in N$. The {\it Connes tensor product} $\mathcal H{\otimes}_N\mathcal K$ is defined as the separation/completion of the algebraic tensor product $\mathcal H\otimes\mathcal K_0$ with respect to the scalar product  $\langle \xi\otimes_N\eta,\xi'\otimes_N\eta'\rangle=\langle\xi y,\xi'\rangle$, where $y\in N$ satisfies $\langle x\eta,\eta'\rangle=\tau(xy)$, for all $x\in N$. Note that $\mathcal H{\otimes}_N\mathcal K$ carries a $M$-$P$ bimodule structure given by $x(\xi\otimes_N\eta)y=x\xi\otimes_N\eta y$. 

In the following six subsections we present the tools we will use in the proofs of our main results.

\subsection {Intertwining-by-bimodules} We first recall from  \cite [Theorem 2.1 and Corollary 2.3]{Po03} S. Popa's powerful {\it intertwining-by-bimodules} technique.

\begin {theorem}\cite{Po03}\label{corner} Let $(M,\tau)$ be a tracial von Neumann algebra and $P,Q\subset M$ be two (not necessarily unital) von Neumann subalgebras. 
Then the following are equivalent:

\begin{itemize}

\item There exist  non-zero projections $p\in P, q\in Q$, a $*$-homomorphism $\phi:pPp\rightarrow qQq$  and a non-zero partial isometry $v\in qMp$ such that $\phi(x)v=vx$, for all $x\in pPp$.

\item There is no sequence $u_n\in\mathcal U(P)$ satisfying $\|E_Q(xu_ny)\|_2\rightarrow 0$, for all $x,y\in M$.
\end{itemize}

If one of these conditions holds true,  then we say that {\it a corner of $P$ embeds into $Q$ inside $M$} and write $P\prec_{M}Q$.
\end{theorem}

Note that if $M$ is not separable, then the same statement holds if the sequence $\{u_n\}_n$ is replaced by a net. 

\subsection {Relative amenability}\label{relativeamen}
A tracial von Neumann algebra $(M,\tau)$ is called {\it amenable} if there exists a net $\xi_n\in L^2(M)\bar{\otimes}L^2(M)$ such that $\langle x\xi_n,\xi_n\rangle\rightarrow \tau(x)$ and $\|x\xi_n-\xi_n x\|_2\rightarrow 0$, for every $x\in M$. By A. Connes' theorem \cite{Co76}, $M$ is amenable iff it is approximately finite dimensional, i.e. $M=(\cup_{n\geqslant 1}M_n)''$, for an increasing sequence $(M_n)_n$ of finite dimensional subalgebras of $M$.

Let $Q\subset M$ be a von Neumann subalgebra. {\it Jones' basic construction} $\langle M,e_Q\rangle$ is defined as the von Neumann subalgebra of $\mathbb B(L^2(M))$ generated by $M$ and the orthogonal projection $e_Q$ from $L^2(M)$ onto $L^2(Q)$.
Recall that $\langle M,e_Q\rangle$ has a faithful semi-finite  trace given by $Tr(xe_QyL)=\tau(xy)$ for all $x,y\in M$. We denote by $L^2(\langle M,e_Q\rangle)$ the associated Hilbert space and endow it with the natural $M$-bimodule structure. Note that $L^2(\langle M,e_Q\rangle)\cong L^2(M){\otimes}_QL^2(M)$, as $M$-$M$ bimodules.

Now, let $P\subset pMp$ be a von Neumann subalgebra, for some projection $p\in M$.
Following \cite[Definition 2.2]{OP07} we say that
 $P$ is {\it amenable relative to $Q$ inside $M$} if there exists a net $\xi_n\in L^2(p\langle M,e_Q\rangle p)$ such that $\langle x\xi_n,\xi_n\rangle\rightarrow \tau(x)$, for every $x\in pMp$, and $\|y\xi_n-\xi_n y\|_2\rightarrow 0$, for every $y\in P$. Note that when $Q$ is amenable, this condition is equivalent to $P$ being amenable.
 
 By \cite[Theorem 2.1]{OP07}, relative amenability is equivalent to the existence of a $P$-central state $\phi$ on $p\langle M,e_Q\rangle p$ such that $\phi_{|pMp}=\tau_{|pMp}$. Recall that if $S$ is a subset of a von Neumann algebra $\mathcal M$, then a state $\phi$ on $\mathcal M$ is said to be $S$-{\it central} if $\phi(xT)=\phi(Tx)$, for all $x\in S$ and $T\in\mathcal M$.

\begin{remark}\label{embedamen} Let $P\subset pMp$ and $Q\subset M$ be von Neumann subalgebras.

\begin{enumerate}
\item Suppose that  there exists a non-zero projection $p_0\in P$ such that $p_0Pp_0$ is amenable relative to $Q$ inside $M$. Let $p_1\in\mathcal Z(P)$ be the central support of $p_0$. Then $Pp_1$ is amenable relative to $Q$.
 Indeed, let $\xi_n\in L^2(p_0\langle M,e_Q\rangle p_0)$ be a net such that $\langle x\xi_n,\xi_n\rangle\rightarrow \tau(x)$, for every $x\in p_0Mp_0$, and $\|y\xi_n-\xi_n y\|_2\rightarrow 0$, for every $y\in p_0Pp_0.$ Also, let $\{v_i\}_{i=1}^{\infty}\subset P$ be partial isometries such that $p_1=\sum_{i=1}^{\infty}v_iv_i^*$ and $v_i^*v_i\leqslant p_0$, for all $i$. It is easy to see that the net $\eta_n=\sum_{i=1}^{\infty}v_i\xi_nv_i^* \in L^2(p_1\langle M,e_Q\rangle p_1)$ witnesses the fact that $Pp_1$ is amenable relative to $Q$. 

\item Suppose that there exists a non-zero projection $p_1\in P'\cap pMp$ such that $Pp_1$ is amenable relative to $Q$ inside $M$. Let $p_2\in\mathcal Z(P'\cap pMp)$ be the central support of $p_1$. By reasoning as in part (1) one deduces that $Pp_2$ is amenable relative to $Q$ inside $M$.

\item If $P\prec_{M}Q$, then there is a non-zero projection $p_0\in P$ such that $p_0Pp_0$ is amenable relative to $Q$. Thus by (1) and (2) there is a non-zero projection $p_2\in\mathcal Z(P'\cap pMp)$ such that $Pp_2$ is amenable relative to $Q$ inside $M$.

\end{enumerate}

\end{remark}

The following lemma, established in \cite[Corollary 2.3]{OP07} (see also \cite[Section 2.5]{PV11}), provides a very useful criterion for relative amenability. 

\begin{lemma}\cite{OP07}\label{op}
Let $(M,\tau)$ be a tracial von Neumann algebra and $Q\subset M$ be a von Neumann subalgebra. Let $P\subset pMp$ be a von Neumann subalgebra, for some projection $p\in M$.  Assume that there exists a $Q$-$M$ bimodule $\mathcal K$ and a net $\xi_n\in pL^2(M){\otimes}_Q\mathcal K$ such that \begin{itemize} 
\item $\limsup_{n}\|x\xi_n\|_2\leqslant \|x\|_2$, for all $x\in pMp$,
\item $\limsup_{n}\|\xi_n\|_2>0$, and
\item $\|y\xi_n-\xi_n y\|_2\rightarrow 0$, for all $y\in P$. 
\end{itemize}
Then $Pp'$ is amenable relative to $Q$ inside $M$, for some non-zero projection $p'\in\mathcal Z(P'\cap pMp)$.
\end{lemma}

\begin{proof} Let us first argue that  we may additionally assume that $\liminf_{n}\|\xi_n\|_2>0$. To see this,  suppose that the net $\xi_n$ is indexed by a directed set $I$ and denote
$\delta=\limsup_{n}\|\xi_n\|_2$. Let $J$ be set of triples $j=(X,Y,\varepsilon)$, where $X\subset pMp,Y\subset P$ are finite sets and $\varepsilon>0$. We make $J$ a directed set by putting $(X,Y,\varepsilon)\leqslant (X',Y',\varepsilon')$ if $X\subset X'$, $Y\subset Y'$ and $\varepsilon'\leqslant\varepsilon$.

Fix $j=(X,Y,\varepsilon)\in J$. By the hypothesis  we can find $n\in I$ such that $\|x\xi_m\|_2\leqslant \|x\|_2+\varepsilon$ and $\|y\xi_m-\xi_m y\|_2\leqslant\varepsilon,$ for all $x\in X$, $y\in Y$ and every $m\geqslant n$. Since $\sup_{m\geqslant n}\|\xi_m\|_2\geqslant\limsup_{n}\|\xi_n\|_2$, we can find $m\geqslant n$ such that $\|\xi_m\|_2>\frac{\delta}{2}$.
Define  $\eta_j=\xi_m$. Then  the net $(\eta_j)_{j\in J}$ clearly satisfies $\limsup_{j}\|x\eta_j\|_2\leqslant \|x\|_2$, for all $x\in pMp$, $\liminf_{j}\|\eta_j\|_2>0$, and $\|y\eta_j-\eta_j y\|_2\rightarrow 0,$ for all $y\in P$.

Now,  choose a state, denoted $\lim_j$, on $\ell^{\infty}(J)$ extending the usual limit. Note that $\pi:\langle M,e_Q\rangle\rightarrow \mathbb B(L^2(M)\bar{\otimes}_Q\mathcal K)$ given by $\pi(T)(\xi\otimes_Q\eta)=T(\xi)\otimes_Q\eta$ is a normal $*$-homomorphism. Define $\psi:\langle M,e_Q\rangle\rightarrow\mathbb C$ by letting $$\psi(T)=\lim_j\|\eta_j\|_2^{-2}\langle \pi(T)\eta_j,\eta_j\rangle.$$

Then $\psi$ is a state on $\langle M,e_Q\rangle$ such that $\psi(p)=1$, $\psi$ is $P$-central and $\psi_{|pMp}$ is normal. By choosing, as in the proof of \cite[Corollary 2.3]{OP07}, the minimal projection $p'\in\mathcal Z(P'\cap pMp)$ such that $\psi(p')=1$ and applying \cite[Theorem 2.1]{OP07}, the conclusion follows.
\end{proof}

\begin{lemma}\label{24}
Let $(M,\tau)$ be a tracial von Neumann algebra and $Q\subset M$ be a von Neumann subalgebra. Let $P\subset pMp$ be a von Neumann subalgebra, for some projection $p\in M$.  Let $\omega$ be a free ultrafilter on $\mathbb N$. 

Suppose that $P\prec_{M^{\omega}}Q^{\omega}$.
More generally, assume that there exists a non-zero projection $p_0\in P'\cap (pMp)^{\omega}$ such that $Pp_0$ is amenable relative to $Q^{\omega}$ inside $M^{\omega}$. 

Then $Pp'$ is amenable relative to $Q$ inside $M$, for some non-zero projection $p'\in\mathcal Z(P'\cap pMp)$.
\end{lemma}

\begin{proof} 
Let $X\subset pMp$, $Y\subset P$ be finite subsets and $\varepsilon>0$. Since $Pp_0$ is amenable relative to $Q^{\omega}$, we can find a vector $\xi\in L^2(p_0\langle M^{\omega},e_{Q^{\omega}}\rangle p_0)$ such that 
\begin{equation}\label{ecu1} \|x\xi\|_2\leqslant\|x\|_2\;\;\;\text{for all}\;\;\; x\in X,\;\;\;\;\; \|\xi\|_2>{\frac{\|p_0\|_2}{2}},\;\;\; \text{and}\end{equation} 
\begin{equation}\label{ecu2} \|y\xi-\xi y\|_2<\varepsilon\;\;\;\text{for all}\;\;\; y\in Y.\end{equation}

By approximating $\xi$ in $\|.\|_2$, we may assume that $\xi$ is in  the linear span of $\{ae_{Q^{\omega}}b|a,b\in M^{\omega}\}$.
Write $\xi=\sum_{i=1}^k a_ie_{Q^{\omega}}b_i,$ where $a_i,b_i\in M^{\omega}$. For every $i\in \{1,...,k\}$, represent $a_i=(a_{i,n})_n$ and $b_i=(b_{i,n})_n$, where $a_{i,n},b_{i,n}\in M$. For every $n$, define $\xi_n=\sum_{i=1}a_{i,n}e_Qb_{i,n}\in\langle M,e_Q\rangle.$

Then  for all $z\in M$, we have that $\|z\xi\|_2=\lim_{n\rightarrow\omega}\|z\xi_n\|_2$ and $\|\xi z\|_2=\lim_{n\rightarrow\omega}\|\xi_nz\|_2$.
Using \ref{ecu1} and \ref{ecu2} it follows that we can find $n$ such that $\eta=\xi_n\in \langle M,e_Q\rangle$ satisfies $\|x\eta\|_2<\|x\|_2$, for all $x\in X$, $\|\eta\|_2>\frac{\|p_0\|_2}{2}$, and $\|y\xi-\xi y\|_2<\varepsilon$, for all $y\in Y$. Continuing as in the proof of Lemma \ref{op} gives the conclusion.
\end{proof}

\subsection{Property $\Gamma$} A II$_1$ factor $M$ has {\it property $\Gamma$} of Murray and von Neumann \cite{MvN43} if  there exists a sequence of unitaries $u_n\in M$ with $\tau(u_n)=0$ such that $\|xu_n-u_nx\|_2\rightarrow 0$, for all $x\in M$. If $\omega$ is a free ultrafilter on $\mathbb N$, then property $\Gamma$ is  equivalent to  $M'\cap M^{\omega}\not=\mathbb C1$. 
By a well-known result of A. Connes \cite[Theorem 2.1]{Co76} property $\Gamma$ is also equivalent to the existence of a net of unit vectors $\xi_n\in L^2(M)\ominus\mathbb C1$ such that $\|x\xi_n-\xi_n x\|_2\rightarrow 0$, for all $x\in M$.

The following theorem is a joint result with S. Vaes (see the Appendix).

 It shows in particular that if an inclusion $P\subset M$  satisfies $P'\cap M^{\omega}=\mathbb C1$, then it also satisfies an, a priori, stronger spectral gap property. We will use this fact later on to prove Theorem \ref{relamen}.

\begin{theorem}\label{spgap}
Let $(M,\tau)$ be a von Neumann algebra with a faithful normal tracial state. Let $P \subset M$ be a von Neumann subalgebra. The following two conditions are equivalent.
\begin{enumerate}
\item The inclusion $P \subset M$ does not have $w$-spectral gap: there exists a net $u_i \in (M)_1$ in the unit ball of $M$ satisfying $\lim_i \|x u_i - u_i x\|_2 = 0$ for all $x \in P$ and satisfying \linebreak $\liminf_i \|u_i - E_{P' \cap M}(u_i)\|_2 > 0$.
\item There exist a Hilbert space $H$ and a net of vectors $\xi_i \in L^2(M) \ot H$ satisfying the following properties:
\begin{itemize}
\item $\lim_i \|(x \ot 1) \xi_i - \xi_i(x \ot 1)\|_2 = 0$ for all $x \in P$,
\item $\liminf_i \|\xi_i - p_{L^2(P' \cap M) \ot H}(\xi_i)\|_2 > 0$,
\item $\limsup_i \|(a \ot 1)\xi_i\|_2 \leq \|a\|_2$ and $\limsup_i \|\xi_i(a \ot 1)\|_2 \leq \|a\|_2$ for all $a \in M$.
\end{itemize}
\end{enumerate}
\end{theorem}

\begin{remark}
In the initial version of this paper, it was falsely claimed that an inclusion $P\subset M$ satisfies $P'\cap M^{\omega}=\mathbb C1$ if and only if it has spectral gap, i.e. every net $\xi_i \in L^2(M)\ominus\mathbb C1$ of unit vectors that satisfy $\lim_i \|x \xi_i - \xi_i x\|_2 = 0$, for all $x \in P$, must verify $\lim_i \|\xi_i\|_2 = 0$. For a discussion of the difference between these two spectral gap properties, see the Appendix.
\end{remark}

Next, we prove that the maximal central projection $e$ of $P'\cap M^{\omega}$ such that $(P'\cap M^{\omega})e$ is diffuse, belongs to $M$. More precisely, we have:

\begin{lemma}\label{gammadec} Let $(M,\tau)$ be a tracial von Neumann algebra and $P\subset pMp$ a von Neumann subalgebra, for a projection $p\in M$.
Let $\omega$ be a free ultrafilter on $\mathbb N$ and denote $P_{\omega}=P'\cap (pMp)^{\omega}$.

Then we can find a projection $e\in\mathcal Z(P'\cap pMp)\cap\mathcal Z(P_{\omega})$ such that

\begin{enumerate}
\item $P_{\omega}e$ is completely atomic and $P_{\omega}e=(P'\cap pMp)e$.
\item $P_{\omega}(p-e)$ is diffuse.
\end{enumerate}
\end{lemma}
\begin{proof}  Let $e\in\mathcal Z(P_{\omega})$ be the maximal projection such that $P_{\omega}e$ is completely atomic.

 Let us prove that $e\in\mathcal Z(P'\cap pMp)$.
To this end, write $e=(e_{n})_n$, where $e_n\in pMp$ is a projection, and let $a$ be the weak limit of $e_n$, as $n\rightarrow\omega$. We have the following:

{\bf Claim}. Let $f_1,f_2,...,f_m\in M^{\omega}$.
Then we can find a subsequence $\{k_n\}_{n\geqslant 1}$ of $\mathbb N$ such that the projection  $f=(e_{k_n})_n\in (pMp)^{\omega}$ satisfies $f\in P_{\omega}$ and  $$\tau_{\omega}(ef)=\tau(a^2),\;\;\tau_{\omega}(efa)=\tau(a^3)\;\;\text{and}\;\;\tau_{\omega}(ef_jf)=\tau_{\omega}(ef_ja),\;\;\text{for all}\;\;j\in\{1,2,...,m\}.$$

 {\it Proof of the claim}. Let $\{x_i\}_{i\geqslant 1}$ be a  $\|.\|_2$ dense sequence of $(P)_1$ and write $f_j=(f_{j,n})_n$, for $j\in\{1,2,...,m\}$. Recall that $\|x_ie_n-e_nx_i\|_2\rightarrow 0$, for all $i$, and that $e_n\rightarrow a$, weakly, as $n\rightarrow\omega$.  Therefore, for every $n\geqslant 1$ we can find $k_n\geqslant 1$ such that $$\|x_ie_{k_n}-e_{k_n}x_i\|_2\leqslant\frac{1}{n},\;\;\text{for all $i\in\{1,2,..,n\}$},\;\; |\tau(e_ne_{k_n})-\tau(e_na)|\leqslant\frac{1}{n},$$ $$|\tau(e_ne_{k_n}a)-\tau(e_na^2)|\leqslant\frac{1}{n}\;\;\text{and}\;\;|\tau(e_nf_{j,n}e_{k_n})-\tau(e_nf_{j,n}a)|\leqslant\frac{1}{n},\;\;\text{for all}\;\;j\in\{1,2,...,m\}.$$
This inequalities clearly imply that $f=(e_{k_n})_n$ satisfies the claim.
\hfill$\square$

Now, using the claim we can inductively construct a sequence of projections $\{f_m\}_{m\geqslant 1}\in P_{\omega}$ such that $\tau_{\omega}(ef_m)=\tau(a^2)$, $\tau_{\omega}(ef_ma)=\tau(a^3)$ and $\tau_{\omega}(ef_jf_m)=\tau_{\omega}(ef_ja),$ for all $j\in\{1,2,...m-1\}$ and $m\geqslant 1$. But then it follows that $\tau(ef_jf_m)=\tau(a^3)$, for all $1\leqslant j<m$.

Next, for $m\geqslant 1$, let $p_m=ef_m$. Since $e$ belongs to the center of $P_{\omega}$, we deduce that $\{p_m\}_{m\geqslant 1}\in P_{\omega}e$ are projections such that $\tau_{\omega}(p_m)=\tau(a^2)$ and $\tau_{\omega}(p_jp_m)=\tau(a^3)$, for all $1\leqslant j<m$.

Finally, since $P_{\omega}e$ is completely atomic, its unit ball is compact in $\|.\|_2$. Thus we can find a subsequence $\{p_{m_l}\}_{l\geqslant 1}$ of $\{p_m\}_{m\geqslant 1}$ which is convergent in $\|.\|_2$. In particular, we have that $|\tau_{\omega}(p_{m_l}p_{m_k})-\tau_{\omega}(p_{m_l})|\leqslant \|p_{m_l}-p_{m_k}\|_{2,\omega}\rightarrow 0$, as $l,k\rightarrow\infty$. This implies that $\tau(a^2)=\tau(a^3)$. Since $0\leqslant a\leqslant 1$, $a$ must be a projection. Thus we have that $\|e_n-a\|_2^2=\tau(e_n)+\tau(a)-2\tau(e_na)\rightarrow 0$, as $n\rightarrow\omega$. Hence $e=(e_n)_n=a\in pMp$ and so $e\in P'\cap pMp$. Since $P_{\omega}'\cap pMp\subset (P'\cap pMp)'\cap pMp$,  it follows that $e\in\mathcal Z(P'\cap pMp)$.

Let $P_0=Pe$. Since $e\in M$, we have that $P_0$ is a subalgebra of $eMe$ and $P_0'\cap (eMe)^{\omega}=P_{\omega}e$ is completely atomic. The proof of \cite[Lemma 2.6]{Co76} then gives that $P_0'\cap (eMe)^{\omega}\subset eMe$. Thus $P_{\omega}e\subset eMe$ and hence $P_{\omega}e=(P'\cap pMp)e$. This proves that $e$ satisfies the first assertion. The second assertion is immediate by the maximality of $e$.
\end{proof}

\subsection{Normalizers in crossed products  by free groups}
Very recently, S. Popa and S. Vaes have established the following remarkable dichotomy.

\begin{theorem}\cite{PV11}\label{pv} Let $\mathbb F_n\curvearrowright (N,\tau)$ be a trace preserving action of a  free group on a tracial von Neumann algebra $(N,\tau)$. Denote $M=N\rtimes\mathbb F_n$ and let $A\subset pMp$ be a von Neumann subalgebra that is amenable relative to $N$, for some projection $p\in M$.

Then either $A\prec_{M}N$ or $\mathcal N_{pMp}(A)''$ is amenable relative to $N$ inside $M$.

\end{theorem}

More generally, it is proven in \cite[Theorem 1.6]{PV11} that the same holds when $\mathbb F_n$ is replaced by a weakly amenable group $\Gamma$ that admits a proper cocycle into an orthogonal representation that is weakly contained in the regular representation.

\subsection{Deformations of AFP algebras}\label{ipp1}

Let $(M_1,\tau_1)$ and $(M_2,\tau_2)$ be two tracial von Neumann algebras with a common von Neumann subalgebra $B$ such that ${\tau_1}_{|B}={\tau_2}_{|B}$. Denote by $M=M_1*_BM_2$ the amalgamated free product algebra (abbreviated, {\bf AFP algebra}) and by $\tau$ its trace extending $\tau_1$ and $\tau_2$.  To present the canonical decomposition of $L^2(M)$, let us fix some notations:

\begin{notations}\label{H_n} Let $n\geqslant 1$
 
\begin{itemize}
\item We denote by $S_n=\{(1,2,1,...),(2,1,2,...)\}$ the set  consisting of the two alternating sequences of $1$'s and $2$'s of length $n$.

\item For $\mathcal I=(i_1,i_2,...,i_n)\in S_n$, we denote $\mathcal H_{\mathcal I}=L^2(M_{i_1}\ominus B)\otimes_{B}...\otimes_BL^2(M_{i_n}\ominus B)$.

\item We also let $\mathcal H_n=\bigoplus_{\mathcal I\in S_n}\mathcal H_{\mathcal I}$ and $\mathcal H_0=L^2(B)$.  
\end{itemize}
\end{notations}

With these notations, we have  $L^2(M)=\oplus_{n=0}^{\infty}\mathcal H_n$. This decomposition easily implies the following lemma that will be useful in the sequel:

\begin{lemma}\label{BM} Let $(M_1,\tau_1)$, $(M_2,\tau_2)$, ($M_3,\tau_3)$ be tracial von Neumann algebras with a common von Neumann subalgebra $B$ such that ${\tau_1}_{|B}={\tau_2}_{|B}={\tau_3}_{|B}$. Then

\begin{enumerate}
\item We can find a $B$-$M_1$ bimodule $\mathcal H$ and a $M_1$-$B$ bimodule $\mathcal K$ such that, as $M_1$-$M_1$ bimodules,  we have  $L^2(M_1*_{B}M_2)\ominus L^2(M_1)\cong L^2(M_1){\otimes}_B\mathcal H\cong\mathcal K{\otimes}_B L^2(M_1)$.
\item We can find a $B$-$B$ bimodule $\mathcal L$ such that $L^2(M_1*_{B}M_2*_{B}M_3)\cong L^2(M_1){\otimes}_B\mathcal L{\otimes}_B L^2(M_2)$, as $M_1$-$M_2$ bimodules.
\end{enumerate}

\end{lemma}

Let us  recall from \cite[Section 2.2]{IPP05} the construction of the {\it free malleable deformation}  of $M=M_1*_{B}M_2$.
Define $\tilde M=M*_{B}(B\bar{\otimes}L(\mathbb F_2))$. 
 Denote  $u_1=u_{a_1}$, $u_2=u_{a_2}$, where $a_1$, $a_2$ are generators of $\mathbb F_2$.  Note that we can decompose $\tilde M=\tilde M_1*_{B}\tilde M_2$, where $\tilde M_1={M_1}*_B(B\bar{\otimes}L(\mathbb Z))$ and $\tilde M_2={M_2}*_B(B\bar{\otimes}L(\mathbb Z))$, and the two copies of $\mathbb Z$ are the cyclic groups generated by $a_1$ and $a_2$, respectively.

Consider the unique function $f:\mathbb T\rightarrow (-\pi,\pi]$ satisfying $f(\exp(it))=t$, for all $t\in (-\pi,\pi]$.
Then $\alpha_1=f(u_1)$ and $\alpha_2=f(u_2)$ are hermitian operators such that $u_1=\exp(i\alpha_1)$ and $u_2=\exp(i\alpha_2)$.
 For $t\in\mathbb R$, define the unitary elements $u_{1}^t=\exp(it\alpha_1)$ and $u_{2}^t=\exp(it\alpha_2)$.

Since  the restrictions of the automorphisms Ad$(u_1^t)$ and Ad$(u_2^t)$  of $\tilde M_1$ and $\tilde M_2$  to $B$ are equal (to id$_B$), the formulae $$\theta_t(x)=u_1^tx{u_1^t}^*,\;\;\text{for $x\in\tilde M_1,$}\;\;\text{and}\; \;\theta_t(y)=u_2^ty{u_2^t}^*,\;\;\text{for  $y\in \tilde M_2$}, $$ define a 1-parameter group $\{\theta_t\}_{t\in\mathbb R}$ automorphisms of $\tilde M$.

The following is the main technical result of \cite{IPP05}.

\begin{theorem}\cite{IPP05}\label{ipp}
Let $A\subset pMp$ be a von Neumann subalgebra, for a projection $p\in M$. Assume that there exist $c>0$ and $t>0$ such that $\tau(\theta_t(u)u^*)\geqslant c$, for all $u\in\mathcal U(A)$. 

Then either $A\prec_{M}B$, or
$\mathcal N_{pMp}(A)''\prec_{M}M_i$, for some $i\in\{1,2\}$.

\end{theorem}

Theorem \ref{ipp} is formulated in a different way and proved under an additional assumption in \cite[Theorem 3.1]{IPP05}. For the formulation given here, see \cite[Section 5]{Ho07} and  \cite[Theorem 5.4]{PV09}.

Note that since $\tau(u_1^t)=\tau(u_2^t)=\frac{\sin(\pi t)}{\pi t}$, we have that $E_M(\theta_t(x))=(\frac{\sin(\pi t)}{\pi t})^{2n}x$, for all $x\in\mathcal H_n$. 
Thus, if we write $x\in M$ as $x=\sum_{n\geqslant 0}x_n$, where $x_n\in\mathcal H_n$, then we have \begin{equation}\label{theta_t}\tau(\theta_t(x)x^*)=\tau(E_M(\theta_t(x))x^*)=\sum_{n\geqslant 0}(\frac{\sin(\pi t)}{\pi t})^{2n}\|x_n\|_2^2.\end{equation}

We derive next a consequence of Theorem \ref{ipp} that we will need in the proof of Theorem \ref{afpgamma}.
\begin{corollary}\label{211}
Let $A\subset pMp$ be a von Neumann subalgebra, for some projection $p\in M$. 

If $A$ is amenable relative to $M_1$, then either $A$ is amenable relative to $B$ or $\mathcal N_{pMp}(A)''\prec_{M}M_1$.

\end{corollary}

\begin{proof}  Assume that $A$ is amenable relative to $M_1$. In the first part of the proof we show that either $Ap'$ is amenable relative to $B$, for a non-zero projection $p'\in\mathcal Z(A'\cap pMp)$, or $\mathcal N_{pMp}(A)''\prec_{M}M_1$. 
To do this, we follow closely the strategy of proof of \cite[Theorem 4.9]{OP07}. 

Since $A$ is amenable relative to $M_1$ we can find a net $\{\xi_n\}_{n\in I}\in L^2(p\langle M,e_{M_1}\rangle p)$ such that 
\begin{equation}\label{centr}
\|x\xi_n-\xi_n x\|_2\rightarrow 0,\;\;\text{for all}\;\; x\in A,\;\;\text{and}\end{equation}
\begin{equation}\label{trac}
\;\;\langle y\xi_n,\xi_n\rangle\rightarrow \tau(y),\;\;\text{for all}\;\; y\in pMp.
\end{equation}
Moreover, the proof of \cite[Theorem 2.1]{OP07} shows that $\xi_n$ can be chosen such that $\xi_n=\zeta_n^{\frac{1}{2}}$, for some $\zeta_n\in L^1(\langle M,e_{M_1}\rangle)_{+}$. Thus, $\langle \xi_n y,\xi_n\rangle=Tr(\zeta_n y)=\langle y\xi_n,\xi_n\rangle\rightarrow\tau(y)$, for all $y\in pMp$.

Next, for $t\in\mathbb R$, we consider the automorphism $\alpha_t$ of $\tilde M$ given by $\alpha_t(x)=x$, for all $x\in \tilde M_1$, and $\alpha_t(y)=u_2^ty{u_2^t}^*$, for all $ y\in\tilde M_2$.  Since $\alpha_t$ is an automorphism of $\tilde M$ that leaves $M_1$ invariant we can extend it to a trace preserving automorphism of $\langle\tilde M,e_{M_1}\rangle$ by letting $\alpha_t(e_{M_1})=e_{M_1}$. 

We also let $\mathcal H$ be the $\|.\|_2$ closure of the span of $Me_{M_1}\tilde M=\{xe_{M_1}y|x\in M,y\in\tilde M\}$ and denote by $e$ the orthogonal projection  from $L^2(\langle\tilde M,e_{M_1}\rangle)$ onto $\mathcal H$. 

{\bf Claim.} Let $x\in A, y\in\tilde M$ and $t\in\mathbb R$. Then we have

\begin{enumerate}
\item $\lim_n\|y\alpha_t(\xi_n)\|_2^2=\tau(y^*y\alpha_t(p))\leqslant \|y\|_2^2$ and $\lim_n\|\alpha_t(\xi_n)y\|_2^2=\tau(yy^*\alpha_t(p))\leqslant \|y\|_2^2$.
\item$\limsup_{n}\|ye(\alpha_t(\xi_n))\|_2\leqslant \|y\|_2$.
\item $\limsup_n \|x\alpha_t(\xi_n)-\alpha_t(\xi_n)x\|_2\leqslant 2\|\alpha_t(x)-x\|_2$.
\end{enumerate}

{\it Proof of the claim.} (1) Since $\xi_n\in p\mathcal H$, by using \ref{trac} we get that $$\|y\alpha_t(\xi_n)\|_2^2=\langle \alpha_t^{-1}(y^*y)\xi_n,\xi_n\rangle=\langle E_M(\alpha_t^{-1}(y^*y))\xi_n,\xi_n\rangle=$$ $$\langle pE_M(\alpha_t^{-1}(y^*y))p\xi_n,\xi_n\rangle\longrightarrow\tau(pE_M(\alpha_t^{-1}(y^*y))p)=\tau(y^*y\alpha_t(p)).$$

The second inequality follows similarly using the fact that $\langle \xi_n y,\xi_n\rangle\rightarrow\tau(y)$, for all $y\in pMp$.

(2) Since $(\tilde M\ominus M)\mathcal H\perp\mathcal H$  and  $\mathcal H$ is a left $M$-module,  we derive that $$\|ye(\alpha_t(\xi_n))\|_2^2=\langle y^*ye(\alpha_t(\xi_n)),e(\alpha_t(\xi_n))=\langle E_M(y^*y)e(\alpha_t(\xi_n),e(\alpha_t(\xi_n))\rangle=$$ $$\|e(E_M(y^*y)^{\frac{1}{2}}\alpha_t(\xi_n))\|_2^2\leqslant \|E_M(y^*y)^{\frac{1}{2}}\alpha_t(\xi_n)\|_2^2.$$
On the other hand, by (1) we have that  $\|E_M(y^*y)^{\frac{1}{2}}\alpha_t(\xi_n)\|_2\leqslant \|E_M(y^*y)^{\frac{1}{2}}\|_2=\|y\|_2$.

(3) Since $\|x\alpha_t(\xi_n)-\alpha_t(\xi_n)x\|_2\leqslant \|(x-\alpha_t(x))\alpha_t(\xi_n)\|_2+\|\alpha_t(\xi_n)(x-\alpha_t(x))\|_2+\|x\xi_n-\xi_n x\|_2$, the inequality folows by combining (1) and \ref{centr}.
\hfill$\square$

Let   $J=(0,\infty)\times I$. Given $(t,n)\in J$, we denote $\eta_{t,n}=\alpha_t(\xi_n)-e(\alpha_t(\xi_n))$ and $\delta_{t,n}=\|\eta_{t,n}\|_2$. For the rest of the proof we treat two separate cases.

{\bf Case 1.} We can find $t>0$ such that $\limsup_{n}\delta_{t,n}<\frac{\|p\|_2}{2}$.

{\bf Case 2.} For all $t>0$ we have that $\limsup_{n}\delta_{t,n}\geqslant\frac{\|p\|_2}{2}$.

In {\bf Case 1}, fix $x\in\mathcal U(A)$. Since $\mathcal H$ is a left $M$-module and $(\tilde M\ominus M)\mathcal H\perp\mathcal H$  we get that \begin{equation}\label{una}\|E_M(\alpha_t(x))\alpha_t(\xi_n)\|_2\geqslant \|e(E_M(\alpha_t(x))\alpha_t(\xi_n))\|_2=\|e(\alpha_t(x)e(\alpha_t(\xi_n)))\|_2\geqslant\end{equation} $$\|e(\alpha_t(x)\alpha_t(\xi_n))\|_2-\delta_{t,n}\geqslant \|e(\alpha_t(\xi_n)\alpha_t(x))\|_2-\|x\xi_n-\xi_n x\|_2-\delta_{t,n}$$
On the other hand, since $\mathcal H$ is a right $\tilde M$-module we deduce that \begin{equation}\label{doua}\|e(\alpha_t(\xi_n)\alpha_t(x))\|_2=\|e(\alpha_t(\xi_n))\alpha_t(x)\|_2\geqslant \|\alpha_t(\xi_n)\alpha_t(x)\|_2-\delta_{t,n}=\|\xi_n x\|_2-\delta_{t,n}\end{equation}

By combining part (1) of the Claim with equations \ref{una}, \ref{doua}, \ref{centr} and \ref{trac} we derive that \begin{equation}\label{alpha} \|E_M(\alpha_t(x))\|_2\geqslant\lim_n\|E_M(\alpha_t(x))\alpha_t(\xi_n)\|_2\geqslant\end{equation} $$\liminf_n(\|\xi_nx\|_2-\|x\xi_n-\xi_n x\|_2-2\delta_{t,n})=$$ $$\|x\|_2-2\limsup_n\delta_{t,n}=\|p\|_2-2\limsup_n\delta_{t,n}>0,\;\;\text{for all}\;\; x\in\mathcal U(A).$$

Now, recall from notations \ref{H_n} that $L^2(M)=\mathcal H_0\bigoplus_{m\geqslant 1}(\oplus_{\mathcal I\in S_m}\mathcal H_{\mathcal I})$. Thus, we can write $x=x_0+\sum_{\substack{m\geqslant 1\\ \mathcal I\in S_m}}x_{\mathcal I}$, where $x_{\mathcal I}\in\mathcal H_{\mathcal I}$.
It is easy to see that if $c_{\mathcal I}$ denotes the number of times 2 appears in $\mathcal I$, then $E_M(\alpha_t(x_{\mathcal I}))=(\frac{\sin(\pi t)}{\pi t})^{2c_{\mathcal I}}x_{\mathcal I}$.  Therefore,  $\|E_M(\alpha_t(x))\|_2^2=\|x_0\|_2^2+\sum_{\substack{m\geqslant 1\\ \mathcal I\in S_m}}(\frac{\sin(\pi t)}{\pi t})^{4c_{\mathcal I}}\|x_{\mathcal I}\|_2^2$. On the other hand, by \ref{theta_t} we have $\tau(\theta_t(x)x^*)=\|x_0\|_2^2+\sum_{\substack{m\geqslant 1\\ \mathcal I\in S_m}}(\frac{\sin(\pi t)}{\pi t})^{2m}\|x_{\mathcal I}\|_2^2.$ Since every $\mathcal I\in S_m$ is an alternating sequence of 1's and 2's, we have that $2c_{\mathcal I}\geqslant m-1.$ 

By combining the last three facts, we conclude that $\tau(\theta_t(x)x^*)\geqslant (\frac{\sin(\pi t)}{\pi t})^{2}\|E_M(\alpha_t(x))\|_2^2$, for every $x\in M$. Together with \ref{alpha} this implies that $\inf_{x\in\;\mathcal U(A)}\tau(\theta_t(x)x^*)>0$. 

Thus, by Theorem \ref{ipp} we get that either  $A\prec_{M}M_1$ or $A\prec_{M}M_2$.  If $A\prec_{M}M_1$, then  \cite[Theorem 1.1]{IPP05} gives that either  $A\prec_{M}B$ or $\mathcal N_{M}(A)''\prec_{M}M_1$. Since by Remark \ref{embedamen}, having $A\prec_{M}B$ implies that there exists a non-zero projection $p'\in\mathcal Z(A'\cap pMp)$ such that $Ap'$ is amenable relative to $B$, the conclusion follows in this case. 

Therefore, in order to finish the proof of {\bf Case 1} we only need to analyze the case when $A\prec_{M}M_2$. By Remark \ref{embedamen} we can find a non-zero projection $p'\in\mathcal Z(A'\cap pMp)$ such that $Ap'$ is amenable relative to $M_2$. By the hypothesis we have that $A$ and thus $Ap'$ is amenable relative to $M_1$. 

We claim that $Ap'$ is amenable relative to $B$.
To this end, denote $\mathcal K=L^2(\langle M,e_{M_1}\rangle){\otimes}_{M}L^2(\langle M,e_{M_2}\rangle)$.  Lemma \ref{BM} provides a $B$-$B$ bimodule $\mathcal L$ such that $L^2(M)\cong L^2(M_1){\otimes}_B\mathcal L{\otimes}_BL^2(M_2)$, as $M_1$-$M_2$ bimodules. Thus, we have the following isomorphisms of $M$-$M$ bimodules $$\mathcal K\cong(L^2(M){\otimes}_{M_1}L^2(M)){\otimes}_{M}(L^2(M){\otimes}_{M_2}L^2(M))\cong L^2(M){\otimes}_{M_1}L^2(M){\otimes}_{M_2}L^2(M)\cong $$ $$  L^2(M){\otimes}_{M_1}(L^2(M_1){\otimes}_B\mathcal L{\otimes}_BL^2(M_2)){\otimes}_{M_2}L^2(M)\cong$$ $$ L^2(M){\otimes}_{B}\mathcal L{\otimes}_{B}L^2(M).$$

Since $Ap'$ is amenable relative to both $M_1$ and $M_2$,  the first part of the proof of \cite[Proposition 2.7]{PV11} implies that the  $p'Mp'$-$Ap'$  bimodule $L^2(p'Mp')$ is weakly contained in the $p'Mp'$-$Ap'$ bimodule $p'\mathcal Kp'$. Thus the $p'Mp'$-$Ap'$ bimodule $p'L^2(M){\otimes}_{B}\mathcal L{\otimes}_{B}L^2(M)p'$ weakly contains the $p'Mp'$-$Ap'$ bimodule $L^2(p'Mp')$. By Lemma \ref{op} it follows that $Ap'$ is amenable relative to $B$. This completes the proof of {\bf Case 1}.

In {\bf Case 2}, we claim that there exists a net $(\eta_k)$ in $\mathcal H^{\perp}$ such that $\|x\eta_k-\eta_k x\|_2\rightarrow 0$, for all $x\in A$, $\limsup_{k}\|y\eta_k\|_2\leqslant 2\|y\|_2$, for all $y\in pMp$, and $\limsup_k\|p\eta_k\|_2>0$.
  
Towards this, let $k=(X,Y,\varepsilon)$ be a triple such that
$X\subset A$, $Y\subset pMp$ are finite sets and $\varepsilon>0$. Then we can find $t>0$ such that \begin{equation}\label{alpha_t}\|\alpha_t(x)-x\|_2<\frac{\varepsilon}{2},\;\;\text{for all}\;\; x\in X,\;\;\text{and}\;\; \|\alpha_t(p)-p\|_2<\frac{\|p\|_2}{10}.\end{equation}

Let $x\in X$ and $y\in Y$.
Firstly, since $\eta_{t,n}=(1-e)(\alpha_t(\xi_n))$ and $x\in M$ we get that $\|x\eta_{t,n}-\eta_{t,n}x\|_2\leqslant \|x\alpha_{t}(\xi_n)-\alpha_t(\xi_n)x\|_2$. This inequality together with part (3) of the  Claim and \ref{alpha_t} implies that $\limsup_n\|x\eta_{t,n}-\eta_{t,n}x\|_2\leqslant 2\|\alpha_t(x)-x\|_2<\varepsilon$.
 
 Secondly, by combining parts (1) and (2) of the Claim we get that $\limsup_n\|y\eta_{t,n}\|_2\leqslant  2\|y\|_2$.
 
Thirdly, part (1) of the Claim gives that $\limsup_n\|p\eta_{t,n}\|_2\geqslant \limsup_n(\|p\alpha_t(\xi_n)\|_2-\|e(\alpha_t(\xi_n))\|_2)=\|p\alpha_t(p)\|_2-\liminf_n\|e(\alpha_t(\xi_n))\|_2.$ Also, since $\|\xi_n\|_2\rightarrow \|p\|_2$ we have that $\liminf_n\|e(\alpha_t(\xi_n))\|_2=\sqrt{\|p\|_2^2-\limsup_{n}\|\eta_{t,n}\|_2^2}\leqslant\frac{\sqrt{3}}{2}\|p\|_2.$ Since   \ref{alpha_t} implies that $\|p\alpha_t(p)\|_2>\frac{9}{10}\|p\|_2$, we altogether deduce that $\limsup_n\|p\eta_{t,n}\|_2> (\frac{9}{10}-\frac{\sqrt{3}}{2})\|p\|_2$.

The last three paragraphs imply that for some $n\in I$,  $\eta_k=\eta_{t,n}$ satisfies $\|x\eta_k-\eta_k x\|_2<\varepsilon$, for all $x\in X$, $\|y\eta_k\|_2\leqslant 2\|y\|_2+\varepsilon$, for all $y\in Y$, and $\|p\eta_k\|_2>(\frac{9}{10}-\frac{\sqrt{3}}{2})\|p\|_2$. It is now clear that the net $(\eta_k)$ has the desired properties.

Finally, by the definition of $\mathcal H$, the $M$-$M$ bimodule $L^2(\langle\tilde M,e_{M_1}\rangle)\ominus\mathcal H$ is isomorphic to the $M$-$M$ bimodule  $(L^2(\tilde M)\ominus L^2(M)){\otimes}_{M_1}L^2(\tilde M)$. Since $\tilde M=M*_{B}(B\bar{\otimes}L(\mathbb F_2))$, Lemma \ref{BM} (1) provides a $B$-$M$ bimodule $\mathcal K$ such that $L^2(\tilde M)\ominus L^2(M)\cong L^2(M){\otimes}_{B}\mathcal K$. Thus, we have the following isomorphism of $M$-$M$ bimodules
$$L^2(\langle\tilde M,e_{M_1}\rangle)\ominus\mathcal H\cong L^2(M){\otimes}_B(\mathcal K{\otimes}_{M_1}L^2(\tilde M)).$$

Since $\eta_k\in L^2(\langle\tilde M,e_{M_1}\rangle)\ominus\mathcal H$, for all $k$, by Lemma \ref{op} there is a non-zero projection $p'\in\mathcal Z(A'\cap pMp)$ such that $Ap'$ is amenable relative to $B$. This finishes the proof of {\bf Case 2}.

Now, to get the conclusion,  let $p_0\in\mathcal Z(A'\cap pMp)$ be the maximal projection such that $Ap_0$ is amenable relative to $B$. It is easy to see that $p_0\in\mathcal N_{pMp}(A)'\cap pMp$.

 Let $p_1=p-p_0$. 
If $p_1=0$, then $A$ is amenable relative to $B$. If $p_1\not=0$, then $Ap_1$ is amenable relative to $M_1$. By the first part of the proof either $Ap'$ is amenable relative to $B$, for some non-zero projection $p'\in\mathcal Z(A'\cap pMp)p_1$, or $\mathcal N_{p_1Mp_1}(Ap_1)''\prec_{M}M_1$. By the maximality of $p_0$, the former is impossible; since $\mathcal N_{pMp}(A)p_1\subset \mathcal N_{p_1Mp_1}(Ap_1),$ the latter implies that $\mathcal N_{pMp}(A)''\prec_{M}M_1$.
\end{proof}

\subsection{Random walks on countable groups} We end this section with some facts from the theory of random walks on countable groups that we will need in Section \ref{conjugacy}. Let $\mu$ and $\nu$ be probability measures on a countable group $\Gamma$. The {\it support} of $\mu$ is the set of $g\in\Gamma$ with $\mu(g)\not=0$. The convolution of $\mu$ and $\nu$ is the probability measure on $\Gamma$ given by $(\mu*\nu)(g)=\sum_{h\in\Gamma}\mu(gh^{-1})\nu(h)$.
For $n\geqslant 1$, we denote $\mu^{*n}=\underbrace{\mu*\mu*...*\mu}_{\text{$n$ times}}$. 

The next lemma is well-known (see for instance \cite[Theorems 2.2 and 2.28]{Fu02}). For the reader's convenience, we include a proof.
 
\begin{lemma}\label{random} Let $\Gamma$ be a finitely generated group and denote by $\ell_S:\Gamma\rightarrow\mathbb N$ the word length  with respect to a finite set of generators $S$.  Let $\mu$ be a probability measure on $\Gamma$ whose support generates a non-amenable subgroup and contains the identity element.

\begin{enumerate}
\item Then $\mu^{*n}(g)\rightarrow 0$, for all $g\in\Gamma$.

 \item Assume that  $\sum_{g\in\Gamma}\ell_S(g)^p\mu(g)<+\infty$, for some $p\in (0,1]$. 
If $\Sigma<\Gamma$ is a finitely generated nilpotent (e.g. cyclic) subgroup, then $\mu^{*n}(h\Sigma k)\rightarrow 0$, for all $h,k\in\Gamma$.
\end{enumerate}

\end{lemma} 
\begin{proof} (1) Let $\lambda:\Gamma\rightarrow\mathcal U(\ell^2(\Gamma))$ be the left regular representation of $\Gamma$.
Define the operator $T:\ell^2(\Gamma)\rightarrow\ell^2(\Gamma)$ by $T=\sum_{g\in\Gamma}\mu(g)\lambda(g)$. Since the support of $\mu$ generates 
a non-amenable group, by Kesten's characterization of amenability (see e.g. \cite[Appendix G.4]{BdHV08}) we have that $\|T\|<\sum_{g\in\Gamma}\mu(g)=1$.

 Denote by $\{\delta_g\}_{g\in\Gamma}$ the canonical orthonormal basis of $\ell^2(\Gamma)$. Then for $n\geqslant 1$ and $g\in\Gamma$ we have  $$\mu^{*n}(g)=\sum_{\substack{g_1,g_2,..,g_n\in\Gamma\\ g_1g_2...g_n=g}}\mu(g_1)\mu(g_2)...\mu(g_n)=\langle T^n(\delta_e),\delta_g\rangle.$$

This implies that $\mu^{*n}(g)\leqslant \|T\|^n$ and since $\|T\|<1$, we are done.

(2) Define the product probability space $(\Omega,\nu)=(\Gamma^{\mathbb N},\mu^{\mathbb N})$ together with the shift $T:\Omega\rightarrow\Omega$  given by $(T\omega)_n=\omega_{n+1}$, for all $\omega=(\omega_n)_n\in\Omega$. Then $T$ is an ergodic, measure preserving transformation of $(\Omega,\nu)$. For $n\geqslant 1$, define $X_n:\Omega\rightarrow\Gamma$ by letting $X_n(\omega)=\omega_1\omega_2...\omega_n$. Note that  $\mu^{*n}=(X_n)_{*}(\nu).$

Further, let $p\in (0,1]$ as in the hypothesis and define $S_n:\Omega\rightarrow [0,\infty)$ by $S_n(\omega)=l_S(X_n(\omega))^p$. Since $p\in (0,1]$, we have that $(a+b)^p\leqslant a^p+b^p$, for all $a,b\geqslant 0$. Recall that for every $g,h\in\Gamma$ we have that $\ell_S(gh)\leqslant \ell_S(g)+\ell_S(h)$. Also we have that $X_{n+m}(\omega)=X_n(\omega)X_m(T^n(\omega))$, for all $n,m\geqslant 1$ and $\omega\in\Omega$. By combining these three facts we deduce that \begin{equation}\label{king} S_{n+m}(\omega)\leqslant S_n(\omega)+S_m(T^n(\omega)),\;\text{for all\; $\omega\in\Omega$\; and \;$n,m\geqslant 1$}\end{equation}

Additionally, by using the hypothesis we get that \begin{equation}\label{l^1}\int_{\Omega}S_1(\omega)d\nu(\omega)=\int_{\Omega}\ell_S(X_1(\omega))^pd\nu(\omega)=\int_{\Gamma}\ell_S(\omega_1)^p d\mu(\omega_1)<+\infty\end{equation}

Since $T$ is ergodic, equations \ref{king} and \ref{l^1} guarantee that we can apply Kingman's subadditive ergodic theorem. Thus, we can find a constant $\alpha\in [0,\infty)$ such that $\frac{1}{n}S_n(\omega)\rightarrow \alpha$, for $\nu$-almost every $\omega\in\Omega$. It follows that $\nu(\{\omega\in\Omega|S_n(\omega)>(\alpha+1)n\})\rightarrow 0$, as $n\rightarrow\infty$.

Hence,  if we let $f(n)=((\alpha+1)n)^{\frac{1}{p}}$, then $\nu(\{\omega\in\Omega|\;\ell_S(X_n(\omega))>f(n)\})\rightarrow 0,$ as $n\rightarrow\infty$.
Since $(X_n)_*(\nu)=\mu^{*n}$, we deduce that \begin{equation}\label{bound}\varepsilon_n:=\mu^{*n}(\{g\in\Gamma|\;\ell_S(g)>f(n)\})\rightarrow 0,\;\text{as $n\rightarrow\infty$}\end{equation}

Now, since $\Sigma$ is a finitely generated  nilpotent group, it has polynomial growth. Thus, we can find $a,b>0$ such that $|\{g\in\Sigma|\; \ell_S(g)\leqslant n\}|\leqslant an^b$, for all $n$. Denoting $c=\ell_S(h)+\ell_S(k)$, we get that \begin{equation}\label{growth}|\{g\in h\Sigma k|\;\ell_S(g)\leqslant n\}|\leqslant a(n+c)^b,\;\text{for all $n$}\end{equation}

Recall from the proof of part (1) that $\mu^{*n}(g)\leqslant \|T\|^n$, for all $g\in\Gamma$ and $n\geqslant 1$. Combining this fact with \ref{bound} and \ref{growth} yields that $$\mu^{*n}(h\Sigma k)\leqslant \varepsilon_n+\mu^{*n}(\{g\in h\Sigma k| \; \ell_S(g)\leqslant f(n)\})\leqslant$$
 $$\varepsilon_n+a\|T\|^n(f(n)+c)^b,\;\text{for all $n\geqslant 1$}.
$$

As $\varepsilon_n\rightarrow 0$, $\|T\|<1$ and $f(n)$ grows polynomially in $n$, we conclude that $\mu^{*n}(h\Sigma k)\rightarrow 0$.
\end{proof}

\section{A conjugacy result for subalgebras of AFP algebras}\label{conjugacy}

Let $(M_1,\tau_1)$ and $(M_2,\tau_2)$ be two tracial von Neumann algebras with a common von Neumann subalgebra $B$ such that ${\tau_1}_{|B}={\tau_2}_{|B}$. Denote  $M=M_1*_BM_2$ and let  $\tilde M=M*_{B}(B\overline{\otimes}L(\mathbb F_2))$. For $t\in\mathbb R$, we consider the automorphism $\theta_t:\tilde M\rightarrow\tilde M$   defined in Section \ref{ipp}. We denote by $\{u_g\}_{g\in\mathbb F_2}\subset L(\mathbb F_2)$ the canonical unitaries and  consider the notations from \ref{H_n}.

In this context, we have

\begin{lemma}\label {perp} Let $\mathcal I=(i_1,i_2,..,i_n)\in S_n$ and $\mathcal J=(j_1,j_2,..,j_m)\in S_m$, for some $n,m\geqslant 1$. 
Let $x_1\in M_{i_1}\ominus B,x_2\in M_{i_2}\ominus B,...,x_n\in M_{i_n}\ominus B$ and $y_1\in M_{j_1}\ominus B, y_2\in M_{j_2}\ominus B,..,y_m\in M_{j_m}\ominus B$.\\
Let $g_1,g_2,..,g_{n+1}, h_1,h_2,..,h_{m+1}\in\mathbb F_2$.

Then $$\langle u_{g_1}x_1u_{g_2}x_2...u_{g_{n}}x_nu_{g_{n+1}},u_{h_1}y_1u_{h_2}y_2...u_{h_{m}}y_mu_{h_{m+1}}\rangle=$$
$$
\begin{cases} \langle x_1x_2...x_n,y_1y_2...y_m\rangle,\;\text{if}\; n=m, \mathcal I=\mathcal J,\text{and}\; g_k=h_k,\text{for all $k\in\{1,2,..,n+1\}$},\;\text{and} \\ 0, \;\; \text{otherwise}. \end{cases}
$$
\end{lemma}
\begin{proof} Denote $A_0=\{u_g\}_{g\in\mathbb F_2\setminus\{e\}}$, $A_1=M_1\ominus B$ and $A_2=M_2\ominus B$. We say that $z=z_1z_2...z_n$ is an {\it alternating product} if for all $i$ we have that $z_i\in A_j$, for some $j\in\{0,1,2\}$ and that $z_i$ and $z_{i+1}$ belong to different ${A_j}$'s. It is clear that $\tau(z)=0$, for any alternating product $z$.

We proceed by induction on $\max\{n,m\}$.
Denote by $\alpha$ the quantity that we want to compute. We have that $$\alpha=
\tau(u_{h_{m+1}}^*y_m^*...y_2^*u_{h_2}^*y_1^*u_{h_1^{-1}g_1}x_1u_{g_2}x_2...x_nu_{g_{n+1}})$$

Assuming that $\alpha\not=0$, let us prove that the first alternative holds.

Firstly, we must have that $g_1=h_1$ and $i_1=j_1$, otherwise $\alpha$ would be the trace of an alternating product. Hence $x_1,y_1\in M_{i_1}\ominus B$ and $\alpha=\tau(u_{h_{m+1}}^*y_m^*...y_2^*u_{h_2}^*(y_1^*x_1)u_{g_2}x_2...x_nu_{g_{n+1}})$. Write $y_1^*x_1=b+z$, where $b\in B$ and $z\in M_{i_1}\ominus B$. Since $u_{h_{m+1}}^*y_m^*...y_2^*u_{h_2}^*zu_{g_2}x_2...x_nu_{g_{n+1}}$ is an alternating product and $b$ commutes with $\mathbb F_2$ we deduce that $$\alpha=\tau(u_{h_{m+1}}^*y_m^*...y_2^*u_{h_2}^*bu_{g_2}x_2...x_nu_{g_{n+1}})=\langle u_{g_2}(bx_2)u_{g_3}...x_nu_{g_{n+1}},u_{h_2}y_2u_{h_3}...y_{m}u_{h_{m+1}}\rangle$$

By induction we get that $n=m$, $i_2=j_2,...,i_n=j_n$ and that $g_2=h_2....g_n=h_n$. It also follows that $\alpha=\langle bx_2x_3...x_n,y_2y_3...y_n\rangle$. Since the latter is equal to $\langle x_1x_2...x_n,y_1y_2...y_n\rangle$, we are done.
\end{proof}

 Next, we present a crossed product decomposition of $\tilde M$ (see \cite[Remark 4.5]{Io06}).  Let $N$ be the subalgebra of $\tilde M$ generated by $\{u_gMu_g^*|g\in\mathbb F_2\}$. Then $N$ 
is normalized by $\mathbb F_2=\{u_g\}_{g\in\mathbb F_2}$. Since $\tilde M$ is generated by $N$ and $\mathbb F_2$, and $E_N(u_g)=0$, for all $g\in\mathbb F_2\setminus\{e\}$, we conclude that $\tilde M=N\rtimes\mathbb F_2$, where $\mathbb F_2$ acts on $N$ by conjugation.

 Moreover,  if $\Sigma<\mathbb F_2$ is a subgroup, then for all $g_1,g_2,...,g_{n+1}\in\mathbb F_2$ and every $x_1,...,x_n\in M$, we have that
\begin{equation}\label{E_N}E_{N\rtimes\Sigma}(u_{g_1}x_1u_{g_2}x_2...u_{g_n}x_nu_{g_{n+1}})=\begin{cases} u_{g_1}x_1u_{g_2}x_2...u_{g_n}x_nu_{g_{n+1}},\;\text{if}\;\; g_1g_2...g_{n+1}\in\Sigma, \;\text{and}  \\ 0, \;\text{if}\; g_1g_2...g_ng_{n+1}\not\in\Sigma .\end{cases}\end{equation}

 Note that the subalgebras $\{u_gMu_g^*\}_{g\in\mathbb F_2}$ of $\tilde M$ are freely independent over $B$. Therefore, $N$ is isomorphic to the infinite amalgamated free product algebra $M*_B*M*_B...$. If we index the copies of $M$ by $\mathbb F_2$, then the action of $\mathbb F_2$ on $N\cong M*_B*M*_B...$ is the {\it free Bernoulli shift}.

We are now ready to state the main result of this section.

\begin{theorem}\label{inter} Let $A\subset pMp$ be a von Neumann subalgebra, for some projection $p\in M$.

Let $t\in (0,1)$. 
Assume that $\theta_t(A)\prec_{\tilde M}N$. More generally, assume that $\theta_t(A)\prec_{\tilde M}N\rtimes\Sigma$, where $\Sigma=\langle a\rangle$ is a cyclic subgroup of $\mathbb F_2$.

Then either $A\prec_{M}B$ or $\mathcal N_{M}(A)''\prec_{M}M_i$, for some $i\in\{1,2\}$.

\end{theorem}

Theorem \ref{inter} is an immediate consequence of Theorem \ref{ipp} and the next lemma.

\begin{lemma}\label{estimate1}
Let $t\in (0,1)$ and $x_k\in (M)_1$ be a sequence such that $\tau(\theta_t(x_k)x_k^*)\rightarrow 0$.

Then $\|E_N(y\theta_t(x_k)z)\|_2\rightarrow 0$, for  every $y,z\in\tilde M$.

More generally, if $\Sigma$ is a cyclic subgroup of $\mathbb F_2$, then  $\|E_{N\rtimes\Sigma}(y\theta_t(x_k)z)\|_2\rightarrow 0$, for every $y,z\in\tilde M$.
\end{lemma}

\vskip 0.05in
{\it Proof of Theorem \ref{inter}}.  If $\theta_t(A)\prec_{\tilde M}N\rtimes\Sigma$, then by Theorem \ref{corner} we can find $v\in\tilde M$ such that $\inf_{u\in\mathcal U(A)}\|E_{N\rtimes\Sigma}(v\theta_t(u)v^*)\|_2>0$. Lemma \ref{estimate1} then  implies that   $\inf_{u\in\mathcal U(A)}\tau(\theta_t(u)u^*)>0$. Finally, the conclusion follows from Theorem \ref{ipp}.\hfill$\square$

\vskip 0.05in 
{\it Proof of Lemma \ref{estimate1}}. Since $\tilde M=N\rtimes\mathbb F_2$, by Kaplansky's density theorem  we may assume that $y=u_g$ and $z=u_h$, for some $g,h\in\mathbb F_2$. Thus, our goal is to prove that $\|E_{N\rtimes \Sigma}(u_g\theta_t(x_k)u_h)\|_2\rightarrow 0$.
Let us first show that this  is a consequence of the next lemma whose proof we postpone for now.

\begin{lemma}\label{estimate2} Fix $t\in (0,1)$ and for $n\geqslant 0$, define $c_n=\sup_{x\in\mathcal H_n,\;\|x\|_2\leqslant 1}\|E_{N\rtimes\Sigma}(u_g\theta_t(x)u_h)\|_2$. 

Then $c_n\rightarrow 0$, as $n\rightarrow\infty$.
\end{lemma}
Assuming Lemma \ref{estimate2}, let us finish the proof of Lemma \ref{estimate1}. Write $x_k=\sum_{n=0}^{\infty}x_{k,n}$, with $x_{k,n}\in\mathcal H_n$.
By equation \ref{theta_t} we have that $\tau(\theta_t(x_k)x_k^*)=\sum_{n=0}^{\infty}(\frac{\sin(\pi t)}{\pi t})^{2n}\|x_{k,n}\|_2^2$. Since $\tau(\theta_t(x_k)x_k^*)\rightarrow 0$ and $\sin(\pi t)>0$, we derive that $\|x_{k,n}\|_2\rightarrow 0$, for all $n\geqslant 0$.

For  $n\geqslant 1$ and $\mathcal I=(i_1,i_2,..,i_n)\in S_n$, we let $\mathcal K_{\mathcal I}\subset L^2(\tilde M)$ be the closure of the linear span of  $$\{u_{h_1}x_1u_{h_2}x_2...u_{h_{n}}x_nu_{h_{n+1}}|h_1,..,h_{n+1}\in\mathbb F_2, x_1\in M_{i_1}\ominus B,x_2\in M_{i_2}\ominus B,...,x_n\in M_{i_n}\ominus B\}.$$

By Lemma \ref{perp} we have that if $\mathcal I\in S_n$ and $\mathcal J\in\mathcal S_m$, then $\mathcal K_{\mathcal I}\perp\mathcal K_{\mathcal J}$, unless $n=m$ and $\mathcal I=\mathcal J$.
Thus, denoting $\mathcal K_n=\oplus_{\mathcal I\in S_n}\mathcal K_{\mathcal I}$, we have that $\mathcal K_n\perp\mathcal K_m$, for all $n\not=m$.

By using the definition of $\theta_t$ and equation \ref{E_N} we derive that $\theta_t(\mathcal H_{\mathcal I})\subset\mathcal K_{\mathcal I}\;\;\text{and}\;\; E_{N\rtimes\Sigma}(\mathcal K_{\mathcal I})\subset \mathcal K_{\mathcal I}$.
Since $\mathcal K_{\mathcal I}$ is an $L(\mathbb F_2)$-$L(\mathbb F_2)$ bimodule, we deduce that $E_{N\rtimes\Sigma}(u_g\theta_t(\mathcal H_{\mathcal I})u_h)\subset\mathcal K_{\mathcal I}$. From this we get that $E_{N\rtimes\Sigma}(u_g\theta_t(\mathcal H_n)u_h)\subset\mathcal K_n$, for all $n\geqslant 1$. 

Since the Hilbert spaces $\{\mathcal K_n\}_{n\geqslant 1}$ are mutually orthogonal,  the vectors $\{E_{N\rtimes\Sigma}(u_g\theta_t(x_{k,n})u_h)\}_{n\geqslant 1}$ are mutually orthogonal, for all $k\geqslant 1$. By using this fact, the inequality $\|\xi+\eta\|_2^2\leqslant 2(\|\xi\|_2^2+\|\eta\|_2^2)$ and the definition of $c_n$, we get that $$\|E_{N\rtimes\Sigma}(u_g\theta_t(x_k)u_h)\|_2^2\leqslant 2\|E_{N\rtimes\Sigma}(u_g\theta_t(x_{k,0})u_h)\|_2^2+2\|\sum_{n=1}^{\infty}E_{N\rtimes\Sigma}(u_g\theta_t(x_{k,n})u_h)\|_2^2=$$  $$2\sum_{n=0}^{\infty}\|E_{N\rtimes\Sigma}(u_g\theta_t(x_{k,n})u_h)\|_2^2\leqslant 2\sum_{n=0}^{\infty}c_n^2\|x_{k,n}\|_2^2.$$

Finally, let $\varepsilon>0$. Since $c_n\rightarrow 0$ by Lemma \ref{estimate2}, we can find $n_0\geqslant 1$ such that $c_n\leqslant\varepsilon$, for all $n\geqslant n_0$. Since $\|x_{k,n}\|_2\rightarrow 0$, for all $n$, we can also find $k_0\geqslant 1$ such that $\|x_{k,i}\|_2\leqslant \frac{\varepsilon}{n_0},$ for all $k\geqslant k_0$ and all $i\in\{1,2,..,n_0-1\}$. Also, note that $c_n\leqslant 1$, for all $n$.

By using the above equation and the inequality $\sum_{n=n_0}^{\infty}\|x_{k,n}\|_2^2\leqslant \|x_k\|_2^2=1$, it follows that $$\|E_{N\rtimes\Sigma}(u_g\theta_t(x_k)u_h)\|_2^2\leqslant 2(n_0(\frac{\varepsilon}{n_0})^2+\varepsilon^2\sum_{n=n_0}^{\infty}\|x_{k,n}\|_2^2)\leqslant 4\varepsilon^2, \;\;\text{for all}\;\;k\geqslant k_0.$$

Since $\varepsilon>0$ was arbitrary, we are done.
\hfill$\square$

\vskip 0.05in
{\it Proof of Lemma \ref{estimate2}}. For  $\mathcal I\in S_n$, let $c_{\mathcal I}=\sup_{x\in\mathcal H_{\mathcal I},\|x\|_2=1}\|E_{N\rtimes\Sigma}(u_g\theta_t(x)u_h)\|_2$.  Recall that $\mathcal H_n=\oplus_{\mathcal I\in S_n}\mathcal H_{\mathcal I}$. Since $u_g\theta_t(\mathcal H_{\mathcal I})u_h\subset\mathcal K_{\mathcal I}$ and the Hilbert spaces $\{\mathcal K_{\mathcal I}\}_{\mathcal I\in S_n}$ are mutually orthogonal by Lemma \ref{perp}, it follows that $c_n=\max_{\mathcal I\in S_n}c_{\mathcal I}$.
 
 In the {\it first part of the proof}, we will find a formula for $c_{\mathcal I}$, for a fixed $\mathcal I=(i_1,i_2,...,i_n)\in S_n$.

Recall that $a_1$ and $a_2$ denote the generators of $\mathbb F_2$. Let $G_1=\langle a_1\rangle$ and $G_2=\langle a_2\rangle$ be the cyclic subgroups generated by $a_1$ and $a_2$. 

Let $g_1,h_1\in G_{i_1}$, $g_2,h_2\in G_{i_2}$,...,$g_n,h_n\in G_{i_n}$. Then by Lemma \ref{perp}, the map given by 
\begin{equation}\label{V}V_{g_1,h_1,g_2,h_2,..,g_n,h_n}(x_1x_2...x_n)=u_{g_1}x_1u_{h_1}^*u_{g_2}x_2u_{h_2}^*...u_{g_n}x_nu_{h_n}^*,\end{equation}
for all  $x_1\in M_{i_1}\ominus B,x_2\in M_{i_2}\ominus B,...,x_n\in M_{i_n}\ominus B$ extends to an isometry $$V_{g_1,h_1,g_2,h_2,..,g_n,h_n}:\mathcal H_{\mathcal I}\rightarrow L^2(\tilde M)$$ Moreover, Lemma \ref{perp} implies that $V_{g_1,h_1,g_2,h_2,..,g_n,h_n}(\mathcal H_{\mathcal I})\perp V_{g_1',h_1',g_2',h_2',..,g_n',h_n'}(\mathcal H_{\mathcal I})$, 
unless we have that $g_1=g_1',h_1^{-1}g_2=h_1'^{-1}g_2',h_2^{-1}g_3=h_2'^{-1}g_3',...,h_{n-1}^{-1}g_n=h_{n-1}'^{-1}g_n',h_n^{-1}=h_n'^{-1}$. Since $G_1\cap G_2=\{e\}$, this implies that
$g_1=g_1',h_1=h_1',...,g_n=g_n',h_n=h_n'$.

Now, let $\beta_1:G_1\rightarrow\mathbb C$ and $\beta_2:G_2\rightarrow\mathbb C$ be given by $\beta_1(g_1)=\tau(u_1^tu_{g_1}^*)$ and $\beta_2(g_2)=\tau(u_2^tu_{g_2}^*)$. Since $u_1^t\in L(G_1)$ and $u_2^t\in L(G_2)$, we can decompose \begin{equation}\label{sum}u_1^t=\sum_{g_1\in G_1}\beta_1(g_1)u_{g_1}\;\;\;\text{and}\;\;\; u_2^t=\sum_{g_2\in G_2}\beta_2(g_2)u_{g_2}\end{equation}
  where the sums converge in $\|.\|_2$. Since $u_1^t$ and $u_2^t$ are unitaries, we have that \begin{equation}\label{l^2}\sum_{g_1\in G_1}|\beta_1(g_1)|^2=\sum_{g_2\in G_2}|\beta_2(g_2)|^2=1\end{equation} 

If $x=x_1x_2...x_n$, for some $x_1\in M_{i_1}\ominus B,x_2\in M_{i_2}\ominus B,...,x_n\in M_{i_n}\ominus B$, then  by \ref{sum} we have

$$u_g\theta_t(x)u_h=u_gu_{i_1}^tx_1{u_{i_1}^t}^*u_{i_2}^tx_2{u_{i_2}^t}^*...u_{i_n}^tx_n{u_{i_n}^t}^*u_h=$$ $$ 
\sum_{g_1,h_1\in G_{i_1} g_2,h_2\in G_{i_2},..,g_nh_n\in G_{i_n}}\beta_{i_1}(g_1)\overline{\beta_{i_1}(h_1)}\beta_{i_2}(g_2)
\overline{\beta_{i_2}(h_2)}...\beta_{i_n}(g_n)\overline{\beta_{i_n}(h_n)}\;\;u_gu_{g_1}x_1u_{h_1}^*u_{g_2}x_2u_{h_2}^*...u_{g_n}x_nu_{h_n}^*u_h
$$
By using equations \ref{E_N} and \ref{V}, we further deduce that

\begin{equation}\label{bigsum}E_{N\rtimes\Sigma}(u_g\theta_t(x)u_h)=\end{equation}
$$
\sum_{\substack{g_1,h_1\in G_{i_1},g_2,h_2\in G_{i_2},..,g_nh_n\in G_{i_n}\\ gg_1h_1g_2h_2...g_nh_nh\in\Sigma}}\beta_{i_1}(g_1)\overline{\beta_{i_1}(h_1)}\beta_{i_2}(g_2)\overline{\beta_{i_2}(h_2)}...\beta_{i_n}(g_n)\overline{\beta_{i_n}(h_n)}\;\;u_gV_{g_1,h_1,g_2,h_2,..,g_n,h_n}(x)u_h.$$
Since the linear span such elements $x$ is dense in $\mathcal H_{\mathcal I}$, this formula holds for every $x\in\mathcal H_{\mathcal I}$.
Since the isometries $V_{g_1,h_1,g_2,h_2,..,g_n,h_n}$ have mutually orthogonal ranges, formula \ref{bigsum} implies that $$\|E_{N\rtimes\Sigma}(u_g\theta_t(x)u_h)\|_2^2=$$ $$\|x\|_2^2\sum_{\substack{g_1,h_1\in G_{i_1},g_2,h_2\in G_{i_2},...,g_nh_n\in G_{i_n}\\ gg_1h_1g_2h_2...g_nh_nh\in\Sigma}}|\beta_{i_1}(g_1)|^2|\beta_{i_1}(h_1)|^2|\beta_{i_2}(g_2)|^2|\beta_{i_2}(h_2)|^2...|\beta_{i_n}(g_n)|^2|\beta_{i_n}(h_n)|^2,$$ for all $x\in\mathcal H_{\mathcal I}.$

Thus, \begin{equation}\label{c_I}c_{\mathcal I}=\sum_{\substack{g_1,h_1\in G_{i_1},g_2,h_2\in G_{i_2},..,g_nh_n\in G_{i_n}\\ gg_1h_1g_2h_2...g_nh_nh\in\Sigma}}|\beta_{i_1}(g_1)|^2|\beta_{i_1}(h_1)|^2|\beta_{i_2}(g_2)|^2|\beta_{i_2}(h_2)|^2..|\beta_{i_n}(g_n)|^2|\beta_{i_n}(h_n)|^2\end{equation}

 In the {\it second part of the proof}, we use this formula for $c_{\mathcal I}$ to conclude that $c_n\rightarrow 0$. By \ref{l^2} we can define probability measures $\mu_1$ and $\mu_2$ on $\mathbb F_2$ by letting \begin{equation}\label{prob}\mu_i(g)=\begin{cases}|\beta_i(g)|^2,\;\text{ if $g\in G_i$,\; and}\\ \text{0,\; if $g\not\in G_i$.}\end{cases}\end{equation}

Denote $\mu=\mu_1*\mu_1*\mu_2*\mu_2$. Then we have 

{\bf Claim.} $\mu^{*n}(g\Sigma h)\rightarrow 0$, for all $g,h\in\mathbb F_2$.

Assuming the claim, let us show that $c_n\rightarrow 0$. Firstly, the claim gives that $(\nu_1*\mu^{*n}*\nu_2)(g\Sigma h)\rightarrow 0$, for any probability measures $\nu_1,\nu_2$ on $\mathbb F_2$ and all $g,h\in\mathbb F_2$.
Secondly, the formula \ref{c_I} rewrites as $$c_{\mathcal I}=(\mu_{i_1}*\mu_{i_1}*\mu_{i_2}*\mu_{i_2}...*\mu_{i_n}*\mu_{i_n)}(g^{-1}\Sigma h^{-1}).$$
Since $i_1\not=i_2,i_2\not=i_3,..,i_{n-1}\not=i_n$, we have that $\mu_{i_1}*\mu_{i_1}*\mu_{i_2}*\mu_{i_2}...*\mu_{i_n}*\mu_{i_n}\in\{\mu^{*[\frac{n}{2}]},\mu^{*[\frac{n}{2}]}*\mu_1*\mu_1,\mu_2*\mu_2*\mu^{*[\frac{n}{2}]},\mu_2*\mu_2*\mu^{*[\frac{n-1}{2}]}*\mu_1*\mu_1\}$. By combining these facts it follows that $c_n\rightarrow 0$, as claimed.

{\it Proof of the claim.}
Firstly, let us prove the claim in the case $\Sigma=\{e\}$. By Lemma \ref{random} (1) it suffices to show that the support of $\mu$ generates a non-amenable group.

 Recall  that $u_{a_1}=\exp(i\alpha_1)$ and $u_1^t=\exp(it\alpha_1)$. Thus  if $n\in\mathbb Z$, then   \begin{equation}\label{mu}\mu_1(a_1^n)=|\tau(u_1^tu_{a_1^n}^*)|^2=|\tau(u_1^{t-n})|^2=(\frac{\sin(\pi(t-n))}{\pi(t-n)})^2=\frac{(\sin(\pi t))^2}{\pi^2(n-t)^2}.\end{equation}
Since $t\in (0,1)$, it follows that $\mu_1(a_1^n)\not=0$ and similarly that $\mu_2(a_2^n)\not=0$, for all $n\in\mathbb Z$. As a consequence the support of $\mu$ contains $a_1$ and $a_2$, and thus generates the whole $\mathbb F_2$.

In general, assume that $\Sigma=\langle a\rangle$, for some $a\in\mathbb F_2$. Let $\ell:\mathbb F_2\rightarrow\mathbb N$ be the word length on $\mathbb F_2$ with respect to the generating set $S=\{a_1,a_1^{-1},a_2,a_2^{-1}\}$.  Note that \ref{mu} also implies that $\mu_1(a_1^n)=\mu_2(a_2^n)\leqslant \frac{C}{|n|^2+1}$, for all $n\in\mathbb Z$, where $C=\frac{2}{t^2(1-t)^2}$.

Let $p\in (0,1)$. Since $|i+j|^p\leqslant |i|^p+|j|^p$, for $i,j\geqslant 0$, we get that $$\sum_{n\in\mathbb Z}|n|^p(\mu_1*\mu_1)(a_1^n)=\sum_{n\in\mathbb Z}|n|^p(\sum_{i+j=n}\mu_1(a_1^i)\mu_1(a_1^j))\leqslant$$ $$C^2\sum_{i,j\in\mathbb Z}\frac{|i|^p+|j|^p}{(|i|^2+1)(|j|^2+1)}=2C^2(\sum_{i\in\mathbb Z}\frac{|i|^p}{|i|^2+1})(\sum_{j\in\mathbb Z}\frac{1}{|j|^2+1})<\infty.$$

 Now,  the support of $\mu$ is  $\{a_1^ma_2^n|m,n\in\mathbb Z\}$ and $\ell(a_1^ma_2^n)=|m|+|n|$, for every $m,n\in\mathbb Z$. By using the last inequality and the analogous one for $\mu_2$ we derive that $$\sum_{g\in\mathbb F_2}\ell(g)^p\mu(g)=\sum_{m,n\in\mathbb Z}(|m|+|n|)^p(\mu_1*\mu_1)(a_1^m)(\mu_2*\mu_2)(a_2^n)\leqslant$$ $$\sum_{m\in\mathbb Z} |m|^p(\mu_1*\mu_1)(a_1^m)+\sum_{n\in\mathbb Z}|n|^p(\mu_2*\mu_2)(a_2^n)<\infty.$$

Since $\Sigma$ is a cyclic group, we can now apply Lemma \ref{random} (2) to get the conclusion of the claim. This finishes the proof of the lemma.
\hfill$\square$

\section{Relative amenability and subalgebras of AFP algebras, I}\label{relative}

Assume the notations from Sections \ref{ipp1} and \ref{conjugacy}. Thus, $(M_1,\tau_1),(M_2,\tau_2)$ are tracial von Neumann algebras, $M=M_1*_{B}M_2$, $\tilde M=M*_B(B\bar{\otimes}L(\mathbb F_2))$ and $N=\{u_gMu_g^*|g\in\mathbb F_2\}''$. 

Our goal in the next two sections is to understand what subalgebras $A\subset M$ have the property that $\theta_t(A)$ is amenable relative to $N$, for some (or  all) $t\in (0,1)$.

We start by considering the case $A=M$.

\begin{theorem}\label{amena} Suppose that $M=M_1*_{B}M_2$ is a factor and let $p\in M$ be a projection.

If $\theta_t(pMp)$ is amenable relative to $N$ inside $\tilde M$, for some $t\in (0,1)$, then either

\begin{enumerate}
\item $M_1p_1$ is amenable relative to $B$ inside $M_1$, for some non-zero projection $p_1\in\mathcal Z(M_1)$, or
\item $M_2p_2$ is amenable relative to $B$ inside $M_2$, for some non-zero projection $p_2\in\mathcal Z(M_2)$.
\end{enumerate}

\end{theorem} 

In particular, if $B$ is amenable and $M_1,M_2$ have no amenable direct summands, then $\theta_t(pMp)$ is not amenable relative $N$, for any $t\in (0,1)$.
 It would be interesting to determine whether the conclusion of Theorem \ref{amena} can be strengthened to ``$M$ is amenable  relative to $B$".

In preparation for the proof of Theorem \ref{amena}, we establish a useful decomposition of the $M$-$M$ bimodule $L^2(\langle\tilde M,e_N\rangle)$. 
Note that $u_gMu_g^*\subset N$, for all $g\in\mathbb F_2$. Equivalently, $[u_ge_Nu_g^*,M]=0$, for every $g\in\mathbb F_2$. Therefore,
$L^2(\langle\tilde{M},e_N\rangle)$  contains an infinite direct sum of trivial $M$-$M$ bimodules: $$\mathcal H=\bigoplus_{g\in\mathbb F_2}L^2(M)u_ge_Nu_g^*.$$

If we let $\mathcal H_2=L^2(\langle\tilde M,e_N\rangle)\ominus\mathcal H$, then we have the following

\begin{lemma}\label{bim} 
There is a $B$-$M$ bimodule $\mathcal K$ such that $\mathcal H_2\cong L^2(M){\otimes}_B\mathcal K,$ as $M$-$M$ bimodules.
\end{lemma}

\begin{proof}  Since $\tilde M=N\rtimes\mathbb F_2$, we have that $$L^2(\langle\tilde M,e_N\rangle)=\bigoplus_{g,h\in\mathbb F_2}L^2(N)u_ge_Nu_h^*.$$ 
For $g\in\mathbb F_2$, let $\sigma_g$ be the automorphism of $N$ given by $\sigma_g(x)=u_gxu_g^*$, for $x\in N$. 
Then the $N$-$N$ bimodule $L^2(N)u_ge_Nu_h^*$ is isomorphic to   $L^2(N)$ endowed with the $N$-$N$ bimodule structure given by $x\cdot\xi\cdot y=x\xi\sigma_{gh^{-1}}(y)$, for all $x,y\in N$ and $\xi\in L^2(N)$. For simplicity, we denote this bimodule by ${_N}L^2(N)_{\sigma_{gh^{-1}}(N)}.$

Next, we define the $M$-$M$ bimodules $\mathcal L=L^2(N)\ominus L^2(M)$ and $\mathcal L_g=_ML^2(N)_{\sigma_{g}(M)}.$
The first  paragraph implies that  $\mathcal H_2\cong\mathcal\oplus_{i=1}^{\infty}(\mathcal L\oplus\bigoplus_{g\in\mathbb F_2\setminus\{e\}}\mathcal L_g),$ as $M$-$M$ bimodules.

Now, denote $P=(\cup_{k\in\mathbb F_2\setminus\{e\}}u_kMu_k^*)''$ and $P_g=(\cup_{k\in\mathbb F_2\setminus\{e,g\}}u_kMu_k^*)''$, for $g\in\mathbb F_2\setminus\{e\}$. Then $N=M*_{B}P$ and $N=M*_{B}\sigma_g(M)*_BP_g$.
By using Lemma \ref{BM} we can find a $B$-$M$ bimodule $\mathcal L'$ and a $B$-$\sigma_g(M)$ bimodule $\mathcal L_g'$ such that $\mathcal L=L^2(M){\otimes}_B\mathcal L'$ and $\mathcal L_g=L^2(M){\otimes}_B\mathcal L_g'$, for all $g\in\mathbb F_2\setminus\{e\}$. 
In combination with the last paragraph this yields the conclusion.
\end{proof}

In the proof of Theorem \ref{amena} we will also need a technical result showing that for $t\in (0,1)$, the angle between the Hilbert spaces ${u_1^t}\mathcal H{u_1^t}^*$ and ${u_2^t}\mathcal H{u_2^t}^*$ is positive.

\begin{lemma}\label{ortho}  Let $t\in (0,1)$ and $u_1^t,u_2^t\in L(\mathbb F_2)$ be the unitaries defined in Section \ref{ipp1}. For $i\in\{1,2\}$, we denote by $P_i$ the orthogonal projection from $L^2(\langle\tilde M,e_N\rangle)$ onto $\mathcal L_i={u_i^t}\mathcal H{u_i^t}^*$.

Then $\|P_1P_2\|<1$.
\end{lemma}

\begin{proof} Let $S={P_1}_{|\mathcal L_2}:\mathcal L_2\rightarrow\mathcal L_1$. Since $\|P_1P_2\|=\|S\|$ it suffices to prove that $\|S\|<1$. We will achieve this by identifying $S$ with the inflation of a certain contraction from $L(\mathbb F_2)$.

Given $g\in\mathbb F_2$, let $\alpha_g=|\tau({u_1^t}^*{u_2^t}u_g^*)|^2$. Note that $\sum_{g\in\mathbb F_2}\alpha_g=1$. If we define the operator $T=\sum_{g\in\mathbb F_2}\alpha_g\lambda(g)\in L(\mathbb F_2)$, then it is clear that $\|T\|\leqslant 1$. 

We claim that  $\|T\|<1$. To see this, recall that $a_1$ and $a_2$ are generators of $\mathbb F_2$. By using  the same calculation as in \ref{mu} we get that $u_1^t=\sum_{n\in\mathbb Z}\frac{\sin(\pi(t-n))}{\pi(t-n)}u_{a_1^n}\;\;\;\text{and}\;\;\;u_2^t=\sum_{n\in\mathbb Z}\frac{\sin(\pi(t-n))}{\pi(t-n)}u_{a_2^n}.$ It follows that $\alpha_g\not=0$ if and only if $g\in\{a_1^ma_2^n|m,n\in\mathbb Z\}$. Thus, the support of $\alpha$ generates the whole $\mathbb F_2$. Since $\mathbb F_2$ is non-amenable and $\alpha_g\geqslant 0$, for all $g\in\mathbb F_2$, we deduce that $\|T\|<\sum_{g\in\mathbb F_2}\alpha_g=1$.

Next, for $i\in\{1,2\}$, we define the unitary operator $U_i:L^2(M)\bar{\otimes}\ell^2(\mathbb F_2)\rightarrow\mathcal L_i$ given by $$U_i(\xi\otimes\delta_g)={u_i^t}u_g\xi e_Nu_g^*{u_i^t}^*,\;\;\text{for}\;\;\xi\in L^2(M)\;\;\text{and}\;\;g\in\mathbb F_2.$$

Let $g,h\in\mathbb F_2$. Since $u_h^*{u_1^t}^*u_2^tu_g\in L(\mathbb F_2)$, we get that $E_N(u_h^*{u_1^t}^*u_2^tu_g)=\tau(u_h^*{u_1^t}^*u_2^tu_g)1$. 
 Thus, for every $\xi,\eta\in L^2(M)$ we get that $$\langle U_1^*SU_2(\xi\otimes\delta_g),\eta\otimes\delta_h\rangle=\langle P_1(u_2^tu_g\xi e_Nu_g^*{u_2^t}^*),u_1^tu_h\eta e_Nu_h^*{u_1^t}^*\rangle=$$ $$ \langle u_2^tu_g\xi e_Nu_g^*{u_2^t}^*,u_1^t u_h\eta e_Nu_h^*{u_1^t}^*\rangle=|\tau(u_h^*{u_1^t}^*u_2^tu_g)|^2\langle \xi,\eta\rangle =$$ $$\alpha_{hg^{-1}}\langle\xi,\eta\rangle=\langle (1\otimes T)(\xi\otimes\delta_g),\eta\otimes\delta_h\rangle.$$ Therefore, $S=U_1(1\otimes T)U_2^*$ and since $\|T\|<1$ we get that $\|S\|<1$.
\end{proof}

{\it Proof of Theorem \ref{amena}.} Assume that $\theta_t(pMp)$ is amenable relative to $N$, for some non-zero projection $p\in M$. Since $M$ is a II$_1$ factor it follows that $\theta_t(M)$ is amenable relative to $N$ (see Remark \ref{embedamen}).
By \cite[Definition 2.2]{OP07} we can find a net of vectors $\xi_n\in L^2(\langle\tilde M,e_N\rangle)$ such that
 $\langle x\xi_n,\xi_n\rangle\rightarrow\tau(x)$, for all $x\in\tilde M$, and $\|y\xi_n-\xi_ny\|_2\rightarrow 0$, for all $y\in\theta_t(M)$.

We denote $\xi_{1,n}={u_1^t}^*\xi_nu_1^t$ and $\xi_{2,n}={u_2^t}^*\xi_nu_2^t$.
Since $\theta_t(y)=u_i^ty{u_i^t}^*$, for all $y\in M_i$ and $i\in\{1,2\}$, we derive that \begin{equation}\label{central}\|y\xi_{1,n}-\xi_{1,n}y\|\rightarrow 0,\;\;\;\text{for all}\;\;y\in M_1,\;\;\;\text{and}\;\;\;  \|y\xi_{2,n}-\xi_{2,n}y\|\rightarrow 0,\;\;\;\text{for all}\;\;y\in M_2.\end{equation}

We also clearly have that \begin{equation}\label{traceal}\langle x\xi_{1,n},\xi_{1,n}\rangle\rightarrow\tau(x)\;\;\;\text{and}\;\;\;\langle x\xi_{2,n},\xi_{2,n}\rangle\rightarrow\tau(x),\;\;\;\text{for all}\;\;\;x\in\tilde M\end{equation}

Denote by $e$ and $f$ the orthogonal projections from $L^2(\langle\tilde M,e_N\rangle)$ onto $\mathcal H_2=L^2(\langle\tilde M,e_N\rangle)\ominus\mathcal H$ and onto  $\mathcal H=\oplus_{g\in\mathbb F_2}L^2(M)u_ge_Nu_g^*$, respectively. Since $e+f=1$, we are in one of the following three cases:

{\bf Case 1.} $\limsup_n\|e(\xi_{1,n})\|_2>0$.

{\bf Case 2}. $\limsup_n\|e\xi_{2,n})\|_2>0$.

{\bf Case 3.} $\|\xi_{1,n}-f(\xi_{1,n})\|_2\rightarrow 0$ and $\|\xi_{2,n}-f(\xi_{2,n})\|_2\rightarrow 0$.

\vskip 0.05in
In {\bf Case 1},   since $\mathcal H_2$ is a $M$-$M$ bimodule, equations \ref{traceal} and \ref{central} imply that $\limsup_n\|xe(\xi_{1,n})\|_2\leqslant \|x\|_2$, for all $x\in\tilde M$, and $\|ye(\xi_{1,n})-e(\xi_{1,n})y\|_2\rightarrow 0$, for all $y\in M_1$. 

We claim that there is a $B$-$M_1$ bimodule $\mathcal K_2$ such that $\mathcal H_2\cong L^2(M_1){\otimes}_B\mathcal K_2$, as $M_1$-$M_1$ bimodules. Assume for now that the claim holds. Then, since $\limsup_n\|e(\xi_{1,n})\|_2>0$, Lemma \ref{op} implies that $M_1p_1$ is amenable relative to $B$ inside $M_1$, for some non-zero projection $p_1\in\mathcal Z(M_1)$.

Now, let us justify the claim. Firstly, Lemma \ref{bim} provides a $B$-$M$ bimodule $\mathcal K$ such that $\mathcal H_2\cong L^2(M){\otimes}_B\mathcal K$, as $M$-$M$ bimodules.  Since $M=M_1*_BM_2$, by Lemma \ref{BM}  we can find a $B$-$M_1$ bimodule $\mathcal K_1$ such that  $L^2(M)\cong L^2(M_1){\otimes}_B\mathcal K_1,$ as $M_1$-$M_1$ bimodules. Finally, it is clear that the $B$-$M_1$ bimodule $\mathcal K_2=\mathcal K_1{\otimes}_B\mathcal K$ satisfies $\mathcal H_2\cong L^2(M_1){\otimes}_B\mathcal K_2$, as $M_1$-$M_1$ bimodules.

Similarly, in {\bf Case 2}, we get that $M_2p_2$ is amenable relative to $B$, for a non-zero projection $p_2\in\mathcal Z(M_2)$.

Finally, let us show that {\bf Case 3} is impossible. Indeed, in this case we would have that $\|\xi_n-u_1^tf(\xi_{1,n}){u_1^t}^*\|_2\rightarrow 0$ and $\|\xi_n-u_2^tf(\xi_{2,n}){u_2^t}^*\|_2\rightarrow 0$. Now, as in Lemma \ref{ortho}, for $i\in\{1,2\}$, we let $P_i$ be the orthogonal projection from $L^2(\langle\tilde M,e_N\rangle)$ onto $\mathcal L_i=u_i^t\mathcal H{u_i^t}^*$. Since $u_i^tf(\xi_{i,n}){u_i^t}^*\in\mathcal L_i$, we deduce that $\|\xi_n-P_1(\xi_n)\|_2\rightarrow 0$ and $\|\xi_n-P_2(\xi_n)\|_2\rightarrow 0$. 

Thus, $\|\xi_n-P_1P_2(\xi_n)\|_2\rightarrow 0$. On the other hand, Lemma \ref{ortho} shows that $\|P_1P_2\|<1$. By combining these two facts we derive that $\|\xi_n\|_2\rightarrow 0$, which is a contradiction.\hfill$\square$

\vskip 0.05in
We end this section by noticing that Theorem \ref{amena} yields a particular case of Theorem \ref{main}:

{\it Proof of Theorem \ref{main} in the case $\Gamma_1$ and $\Gamma_2$ are non-amenable, and $\Lambda$ is amenable}. Therefore, let $\Gamma\curvearrowright (X,\mu)$ be a free ergodic pmp action of $\Gamma=\Gamma_1*_{\Lambda}\Gamma_2$. Recall that  $\cap_{i=1}^ng_i\Lambda g_i^{-1}$ is finite, for some $g_1,g_2,...,g_n\in\Gamma$, and denote $M=L^{\infty}(X)\rtimes\Gamma$.

We claim that any Cartan subalgebra $A$ of $M$ is unitarily conjugate to $L^{\infty}(X)$.
To this end, notice that $M=M_1*_{B}M_2$, where $M_1=L^{\infty}(X)\rtimes\Gamma_1, M_2=L^{\infty}(X)\rtimes\Gamma_2$ and $B=L^{\infty}(X)\rtimes\Lambda$. Let $\tilde M$, $\{\theta_t\}_{t\in\mathbb R}\subset$ Aut$(\tilde M)$ and $N$ be defined as above. 

Let $t\in (0,1)$. 
Since $\tilde M=N\rtimes\mathbb F_2$, by  applying Theorem \ref{pv} to $\theta_t(A)\subset\tilde M$ we have that either
$\theta_t(A)\prec_{\tilde M}N$ or $\theta_t(M)$ is amenable relative to $N$ inside $\tilde M$.

In the first case, Theorem \ref{inter} gives that either $A\prec_{M}B=L^{\infty}(X)\rtimes\Lambda$ or $M\prec_{M}M_i$, for some $i\in\{1,2\}$.   
 If the first condition holds, then since $M$ is a factor, \cite[Proposition 8]{HPV10} implies that $A\prec_{M}L^{\infty}(X)\rtimes(\cap_{i=1}^ng_i\Lambda g_i^{-1})$. Thus, $A\prec_{M}L^{\infty}(X)$ and \cite[Theorem A.1]{Po01} gives that $A$ and $L^{\infty}(X)$ are indeed unitarily conjugate. On the other hand, the second condition cannot hold true. To see this, let $g_1\in\Gamma_1\setminus\Lambda$ and $g_2\in\Gamma_2\setminus\Lambda$. Then the unitary $u=u_{g_1g_2}$ satisfies $\|E_{M_i}(xu^ny)\|_2\rightarrow 0$, for every $x,y\in M$.

In the second case, Theorem \ref{amena} implies that $M_ip_i$ is amenable relative to $B$ for some $p_i\in\mathcal Z(M_i)$ and some $i\in\{1,2\}$. Since $B$ is amenable, this would imply that $M_ip_i$ is amenable. Since $L(\Gamma_i)\subset M_i$ and $\Gamma_i$ is non-amenable, this case is impossible.\hfill$\square$

\section{Relative amenability and subalgebras of AFP algebras, II}\label{relative2}

 Let $(M_1,\tau_1)$ and $(M_2,\tau_2)$ be two tracial von Neumann algebras. Following the notations from Sections \ref{ipp1} and \ref{conjugacy}, we denote $M=M_1*_{B}M_2$, $\tilde M=M*_B(B\bar{\otimes}L(\mathbb F_2))$ and $N=\{u_gMu_g^*|g\in\mathbb F_2\}''$.
 
In this section we prove two structural results for subalgebras $A\subset M$ with the property that $\theta_t(A)$ is amenable relative to $N$, for any $t\in (0,1)$. Firstly, we show:

\begin{theorem}\label{relamen} Let $A\subset pMp$ be a von Neumann subalgebra, for some projection $p\in M$. Let $\omega$ be a free ultrafilter on $\mathbb N$ and suppose  that  $A'\cap (pMp)^{\omega}=\mathbb Cp$.

If $\theta_t(A)$ is amenable relative to $N$ inside $\tilde M$, for any $t\in (0,1)$, then either

\begin{enumerate}
\item $A\prec_{M}M_i$, for some $i\in\{1,2\}$, or
\item $A$ is amenable relative to $B$ inside $M$.
\end{enumerate}

\end{theorem}

It seems to us that this theorem should hold without assuming that $A'\cap (pMp)^{\omega}=\mathbb Cp$, but we were unable to prove this. This assumption is verified for instance if $A=M$ and $M$ is a II$_1$ factor without property $\Gamma$.
By \cite[Corollary 3.2]{CH08} if $B$ is amenable and $M_1$ is a II$_1$ factor without property $\Gamma$, then $M=M_1*_BM_2$ is a II$_1$ factor which does not have property $\Gamma$. In the next section we will see  more situations in which the above assumption holds. 

Nevertheless, the condition $A'\cap (pMp)^{\omega}=\mathbb C$ is not satisfied in other situations to which we would like to apply Theorem \ref{relamen}. For instance, let $\Gamma=\Gamma_1*\Gamma_2$ be a free product group and $\Gamma\curvearrowright (X,\mu)$ be a free ergodic but not strongly ergodic action. Then the amalgamated free product II$_1$ factor $M=L^{\infty}(X)\rtimes\Gamma=(L^{\infty}(X)\rtimes\Gamma_1)*_{L^{\infty}(X)}(L^{\infty}(X)\rtimes\Gamma_2)$ has property $\Gamma$. 

In order to treat such situations, we prove the following variant of Theorem \ref{relamen}:

\begin{theorem}\label{relamen2}  In the above setting, assume that we can decompose $B=P\bar{\otimes}Q_0$, $M_1=P\bar{\otimes}Q_1$ and $M_2=P\bar{\otimes}Q_2$, for some  tracial von Neumann algebras $P,Q_0,Q_1$ and $Q_2$. Note that $M=P\bar{\otimes}Q$, where $Q=Q_1*_{Q_0}Q_2$.

Let $A\subset M$ be a von Neumann subalgebra. Suppose that there exist a subgroup $\mathcal U\subset\mathcal U(P)$ and a homomorphism $\rho:\mathcal U\rightarrow\mathcal U(Q)$ such that 

\begin{itemize}\item $u\otimes\rho(u)\in A$, for all $u\in\mathcal U$, and 
\item the von Neumann subalgebra $A_0\subset Q$ generated by $\{\rho(u)|u\in\mathcal U\}$ satisfies $A_0'\cap Q^{\omega}=\mathbb C$.
\end{itemize}

If $\theta_t(A)$ is amenable relative to $N$ inside $\tilde M$, for any $t\in (0,1)$, then either

\begin{enumerate}
\item $A_0\prec_{Q}Q_i$, for some $i\in\{1,2\}$, or
\item $A_0$ is amenable relative to $Q_0$ inside $Q$.
\end{enumerate}
\end{theorem}

In the rest of this section, we first prove Theorem \ref{relamen} and then use it to deduce Theorem \ref{relamen2}.

\vskip 0.05in
{\it Proof of Theorem \ref{relamen}}. Suppose by contradiction that conditions (1) and (2) fail. 

We begin by introducing the following notation:

\begin{itemize}
\item $\mathcal H_0=\bigoplus_{g\in\mathbb F_2}\mathbb Cu_ge_Nu_g^*$ and $\mathcal H_1=\bigoplus_{g\in\mathbb F_2}(L^2(M)\ominus\mathbb C)u_ge_Nu_g^*.$
\item $\mathcal H=\mathcal H_0\oplus\mathcal H_1=\bigoplus_{g\in\mathbb F_2}L^2(M)u_ge_Nu_g^*$ and $\mathcal H_2=L^2(\langle\tilde M,e_N\rangle)\ominus\mathcal H$.
\item $\mathcal K_0=\bigoplus_{g\in\mathbb F_2}\mathbb Cp\;u_ge_Nu_g^*$ and $\mathcal K_1=\bigoplus_{g\in\mathbb F_2}(L^2(pMp)\ominus\mathbb Cp)u_ge_Nu_g^*.$
\item $\mathcal K=\mathcal K_0\oplus\mathcal K_1=\bigoplus_{g\in\mathbb F_2}L^2(pMp)u_ge_Nu_g^*$ and $\mathcal K_2=pL^2(\langle\tilde M,e_N\rangle)p\ominus\mathcal K$.
\end{itemize}

Note that $L^2(\langle\tilde M,e_N\rangle)=\mathcal H_0\oplus\mathcal H_1\oplus\mathcal H_2$ and $pL^2(\langle\tilde M,e_N\rangle)p=\mathcal K_0\oplus\mathcal K_1\oplus\mathcal K_2$.
For $j\in\{0,1,2\}$, we denote by $e_j$ the orthogonal projection from $L^2(\langle\tilde M,e_N\rangle)$ onto $\mathcal K_j$. We also denote by $e=e_0+e_1$ the orthogonal projection onto $\mathcal K$.

 We denote by $I$ the set  of $4$-tuples $i=(X,Y,\delta,t)$ where $X\subset \tilde M$ and $Y\subset\mathcal U(A)$ are finite subsets, $\delta\in (0,1)$ and $t\in (0,1)$. We make $I$ a directed set by letting: $(X,Y,\delta,t)\leqslant (X',Y',\delta',t')$ if and only if $X\subset X',Y\subset Y',\delta'\leqslant\delta$ and $t'\leqslant t$.

Let $i=(X,Y,\delta,t)\in I$.
Since $\theta_t(A)$ is amenable relative to $N$ inside $\tilde M$,   by \cite[Definition 2.2]{OP07} we can find a vector $\xi_{i}\in L^2(\langle \tilde M,e_N\rangle)$ such that  $$|\langle x\xi_{i},\xi_{i}\rangle-\tau(x)|\leqslant\delta,\;\;\text{for all $x\in X$,}$$ $$ |\langle (\theta_t(y)-y)^*(\theta_t(y)-y)\xi_i,\xi_i\rangle-\tau((\theta_t(y)-y)^*(\theta_t(y)-y))|\leqslant\delta\;\;\text{and} $$ $$ \|\theta_t(y)\xi_{i}-\xi_{i}\theta_t(y)\|_2\leqslant\delta,\;\;\text{for all $y\in Y$}.$$

Moreover, following the proof of \cite[Theorem 2.1]{OP07} we may assume that 
$\xi_i=\eta_i^{\frac{1}{2}}$, for some $\eta_i\in L^1(\langle\tilde M,e_N\rangle)_{+}$. 
Thus, $\langle x\xi_i,\xi_i\rangle=Tr(x\eta_i)=\langle \xi_i x,\xi_i\rangle$, for all $x\in \tilde M$ and $i\in I$.

The {\it first part of the proof} consists of three claims.

{\bf Claim 1.} We have that $\langle x\xi_i,\xi_i\rangle\rightarrow\tau(x)$, for all $x\in \tilde M$, and $\|y\xi_i-\xi_i y\|_2\rightarrow 0$, for all $y\in\mathcal U(A)$.

{\it Proof of Claim 1.}  The first assertion is clear. To prove the second assertion, let $i=(X,Y,\delta,t)\in I$ and $y\in Y$.
Then we have

$$\|(\theta_t(y)-y)\xi_i\|_2^2=\langle (\theta_t(y)-y)^*(\theta_t(y)-y)\xi_i,\xi_i\rangle\leqslant\delta+\|\theta_t(y)-y\|_2^2.$$

Similarly, we have that $\|\xi_i(\theta_t(y)-y)\|_2^2\leqslant\delta+\|\theta_t(y)-y\|_2^2$. By combining these inequalities we deduce that $$\|y\xi_i-\xi_iy\|_2\leqslant \|\theta_t(y)\xi_i-\xi_i\theta_t(y)\|_2+\|(\theta_t(y)-y)\xi_i\|_2+\|\xi_i(\theta_t(y)-y)\|_2\leqslant$$ $$\delta+2\;\sqrt{\delta+\|\theta_t(y)-y\|_2^2}.$$

Since $\|\theta_t(y)-y\|_2\rightarrow 0$, as $t\rightarrow 0$, it follows that $\|y\xi_i-\xi_iy\|_2\rightarrow 0$.
\hfill$\square$

\vskip 0.05in
For $i\in I$, we denote $\zeta_i=p\xi_i p\in pL^2(\langle\tilde M,e_N\rangle)p$. Note that $e_j(\xi_i)=e_j(\zeta_i)$, for all $j\in\{0,1,2\}$.

{\bf Claim 2.} $\|\zeta_i-e_0(\zeta_i)\|_2\rightarrow 0$.
 
{\it Proof of Claim 2.} Since $e_0(\zeta)+e_1(\zeta)+e_2(\zeta)=\zeta$, for every $\zeta\in pL^2(\langle\tilde M,e_N\rangle)p$, it suffices to show that $\|e_1(\zeta_i)\|_2\rightarrow 0$ and $\|e_2(\zeta_i)\|_2\rightarrow 0$.

Firstly, since $\mathcal K$ is a $pMp$-$pMp$ bimodule,  Claim 1 implies that the vectors $e(\zeta_i)=e(p\xi_ip)\in\mathcal K$ satisfy $\lim\limits_{i}\|xe(\zeta_i)-e(\zeta_i)x\|_2=0$, for all $x\in A$. Also, we get that $\limsup\limits_i\|ye(\zeta_i)\|_2\leqslant \|y\|_2$  for every $y\in M$.
Indeed, if $y\in M$, then for all $i$ we have that

 $$\|ye(\zeta_i)\|_2^2=\langle (py^*yp)e(\zeta_i),e(\zeta_i)\rangle=\|e((py^*yp)^{\frac{1}{2}}p\xi_ip)\|_2^2\leqslant \|(py^*yp)^{\frac{1}{2}}\xi_i\|_2^2=\langle(py^*yp)\xi_i,\xi_i\rangle$$

Since $\lim\limits_{i}\langle(py^*yp)\xi_i,\xi_i\rangle=\tau(py^*yp)\leqslant \|y\|_2^2$, this proves our assertion. Similarly, it follows that $\limsup\limits_i\|e(\zeta_i)y\|_2\leqslant \|y\|_2$, for all $y\in M$.
Note that $\mathcal K\cong L^2(pMp)\otimes\ell^2$, as a Hilbert $pMp$-$pMp$ bimodule. Since $A'\cap (pMp)^{\omega}=\mathbb Cp$,  the inclusion $A\subset pMp$ has $w$-spectral gap, and by applying Theorem \ref{spgap} 
we get that $\lim\limits_{i}\|e(\zeta_i)-e_0(\zeta_i)\|_2=0$. Thus, $\lim\limits_{i}\|e_1(\zeta_i)\|_2=0$.

Secondly, since $\mathcal K_2=p\mathcal H_2p$ is a $pMp$-$pMp$ bimodule, $e_2$ is $pMp$-$pMp$ bimodular and therefore we have that $$\limsup_i\|xe_2(\zeta_i)\|_2=\limsup_i\|xe_2(\xi_i)\|_2=\limsup_i\|e_2(x\xi_i)\|_2\leqslant \limsup_i\|x\xi_i\|_2= $$
$$ \limsup_i\sqrt{\langle x^*x\xi_i,\xi_i\rangle}=\|x\|_2,\;\;\;\text{for all $x\in M$}$$
and that $\|ye_2(\zeta_i)-e_2(\zeta_i)y\|_2=\|e_2(y\xi_i-\xi_i y)\|_2\leqslant \|y\xi_i-\xi_iy\|_2\rightarrow 0,$ for all $y\in\mathcal U(A)$.

Now,  recall that Lemma \ref{bim} shows that $\mathcal H_2\cong L^2(M){\otimes}_B\mathcal K$, for some $B$-$M$ bimodule $\mathcal K$. Thus, if $\limsup_i\|e_2(\zeta_i)\|_2>0$, then by Lemma \ref{op} we could find a non-zero projection $z\in\mathcal Z(A'\cap pMp)$ such that $Az$ is amenable relative to $B$ inside $M$. Since $A'\cap pMp=\mathbb C$, this would imply that $A$ is amenable relative to $B$ inside $M$, leading to a contradiction.
\hfill$\square$

Before proving our third claim, let us state two lemmas whose proofs we postpone for now. Denote by $\lambda:\mathbb F_2\rightarrow\mathcal U(\ell^2(\mathbb F_2))$ the left regular representation of $\mathbb F_2$. Then we have 

\begin{lemma}\label{calc}
Define the unitary operator $U:\mathcal H_0\rightarrow\ell^2(\mathbb F_2)$ given  by $U(u_ge_Nu_g^*)=\delta_g$, for $g\in\mathbb F_2$.

If $\eta\in\mathcal H_0$ and $y\in\tilde M$, then $$\|y\eta-\eta y\|_2^2=\sum_{g\in\mathbb F_2} \|\lambda(g)(U(\eta))-U(\eta)\|^2 \|E_N(yu_g^*)\|_2^2.$$
\end{lemma}

\begin{lemma}\label{gap}
 There exists $c>0$ such that if two elements $g,h\in\mathbb F_2$ satisfy  $\|\lambda(g)(\eta)-\eta\|\leqslant c\|\eta\|$ and $\|\lambda(h)(\eta)-\eta\|\leqslant c\|\eta\|$, for some non-zero vector $\eta\in\ell^2(\mathbb F_2)$, then $g$ and $h$ commute.
\end{lemma}

Going back to the proof of Theorem \ref{relamen}, recall that Claim 2 yields that $\|\zeta_i-e_0(\zeta_i)\|_2\rightarrow 0$. Moreover, Claim 1 gives that $\|\zeta_i\|_2\rightarrow \|p\|_2$ and that $\|p\xi_i-\xi_i p\|_2\rightarrow 0$. 

Thus, we can find $i=(X,Y,\delta,t)\in I$ such that for every $i'\geqslant i$ we have that $$\|\zeta_{i'}-e_0(\zeta_{i'})\|_2<\min\{\frac{c\|p\|_2}{128},\frac{\|p\|_2}{4}\},\;\;\;\;\|\zeta_{i'}\|_2\geqslant\frac{\|p\|_2}{2},\;\;\;\;\text{and}\;\;\;\;\|p\xi_{i'}-\xi_{i'} p\|_2\leqslant\frac{c\|p\|_2}{64}.$$

Note that $\|p\theta_t(y)p\|_2\geqslant \|p\|_2-2\|\theta_t(p)-p\|_2$, for all $y\in\mathcal U(p\tilde Mp)$. Since $\lim_{t\rightarrow 0}\|\theta_t(p)-p\|_2=0$, after eventually shrinking $t$, we may also assume that 
\begin{equation}\label{teta}\|p\theta_t(y)p\|_2\geqslant\frac{\|p\|_2}{2},\;\;\;\text{for all}\;\;\;y\in\mathcal U(p\tilde Mp)\end{equation}
 
 \vskip 0.05in
 Let $i'\geqslant i$. Then $\|e_0(\zeta_{i'})\|_2\geqslant\frac{\|p\|_2}{4}$. Since $e_0(\zeta_{i'})\in \mathcal K_0=p\mathcal H_0$, we can write $e_0(\zeta_{i'})=\eta_{i'}p=p\eta_{i'}$, for some $\eta_{i'}\in\mathcal H_0$. Then $\|\eta_{i'}\|_2=\frac{\|e_0(\zeta_{i'})\|_2}{\|p\|_2}$ and therefore $\|\eta_{i'}\|_2\geqslant\frac{1}{4}$.

Also, we have that $\|\zeta_{i'}-\xi_{i'}p\|_2=\|p\xi_{i'}p-\xi_{i'}p\|_2\leqslant \|p\xi_{i'}-\xi_{i'}p\|_2\leqslant\frac{c\|p\|_2}{64}$ and similarly that  $\|\zeta_{i'}-p\xi_{i'}\|_2\leqslant\frac{c\|p\|_2}{64}$. By using these inequalities we derive the following
\vskip 0.05in

\vskip 0.05in
{\bf Claim 3.} Let $c$ be the constant provided by Lemma \ref{gap}. Then for  every finite set $F\subset\mathcal U(A)$ we can find a unit vector $\eta\in\mathcal H_0$ depending on $F$ such that $$\|(p\theta_t(y)p)\eta-\eta(p\theta_t(y)p)\|_2\leqslant\frac{c\|p\|_2}{4},\;\;\text{for all}\;\; y\in F. $$

{\it Proof of Claim 3.}
Let $i'=(X,Y\cup F,t,\min\{\delta,\frac{c\|p\|_2}{64}\})$ and define $\eta:=\frac{\eta_{i'}}{\|\eta_{i'}\|_2}\in\mathcal H_0$. 

Let $y\in F$. By the definition of $\xi_{i'}$ we have that $\|\theta_t(y)\xi_{i'}-\xi_{i'}\theta_t(y)\|_2\leqslant\frac{c\|p\|_2}{64}$. 
Since $i'\geqslant i$, by using the previous inequalities we derive that
\begin{equation}\label{etainv1}\|(p\theta_t(y)p)\eta-\eta(p\theta_t(y)p)\|_2=\frac{1}{\|\eta_{i'}\|_2}\|p\theta_t(y)e_0(\zeta_{i'})-e_0(\zeta_{i'})\theta_t(y)p\|_2\leqslant\end{equation}  $$4\|p\theta_t(y)\zeta_{i'}-\zeta_{i'}\theta_t(y)p\|_2+8\|\zeta_{i'}-e_0(\zeta_{i'})\|_2$$ Additionally, we have that \begin{equation}\label{etainv2}\|p\theta_t(y)\zeta_{i'}-\zeta_{i'}\theta_t(y)p\|_2\leqslant \|p\theta_t(y)\xi_{i'}p-p\xi_{i'}\theta_t(y)p\|_2+\|\zeta_{i'}-\xi_{i'}p\|_2+\|\zeta_{i'}-p\xi_{i'}\|_2\leqslant\end{equation} $$\|\theta_t(y)\xi_{i'}-\xi_{i'}\theta_t(y)\|_2+\frac{c\|p\|_2}{32}\leqslant \frac{3c\|p\|_2}{64}.$$

Since $\|\zeta_{i'}-e_0(\zeta_{i'})\|_2\leqslant\frac{c\|p\|_2}{128}$, by combining equations \ref{etainv1} and \ref{etainv2} the claim follows.\hfill$\square$

\vskip 0.05in

In the {\it second part of the proof} we combine Lemmas \ref{calc}, \ref{gap} and Claim 3  to get a contradiction.
Since  $A\nprec_{M}M_i$, for all $i\in\{1,2\}$, Theorem \ref{inter} implies that $\theta_t(A)\nprec_{\tilde M}N$ and moreover that $\theta_t(A)\nprec_{\tilde M}N\rtimes\Sigma$, for any cyclic subgroup $\Sigma<\mathbb F_2$. 

Thus, we can find $y\in\mathcal U(A)$ such that $\|E_N(p\theta_t(y)p)\|_2\leqslant \frac{\|p\|_2}{4}$. If we write $p\theta_t(y)p=\sum_{g\in\mathbb F_2}y_gu_g$, where $y_g\in N$, then $\|y_e\|_2\leqslant\frac{\|p\|_2}{4}$. By applying Claim 3 to $F=\{y\}$ we can find a unit vector $\eta\in\mathcal H_0$ such that  $\|(p\theta_t(y)p)\eta-\eta(p\theta_t(y)p)\|_2\leqslant \frac{c\|p\|_2}{4}$. 

Let $S_1=\{g\in\mathbb F_2| \|\lambda(g)(U(\eta))-U(\eta)\|>c\}$ and $S_2=\{g\in\mathbb F_2\setminus\{e\}|\|\lambda(g)(U(\eta))-U(\eta)\|\leqslant c\}$. 
By using Lemma \ref{calc} we get that $$\frac{c^2\|p\|_2^2}{16}\geqslant \|(p\theta_t(y)p)\eta-\eta(p\theta_t(y)p)\|_2^2=\sum_{g\in\mathbb F_2}\|\lambda(g)(U(\eta))-U(\eta)\|^2\|y_g\|_2^2\geqslant$$ $$ c^2\sum_{g\in S_1}\|y_g\|_2^2.$$ 
Hence, we derive that  \begin{equation}\label{5}\sum_{g\in S_1\cup\{e\}}\|y_g\|_2^2=\|y_e\|_2^2+\sum_{g\in S_1}\|y_g\|_2^2\leqslant \frac{\|p\|_2^2}{16}+\frac{\|p\|_2^2}{16}=\frac{\|p\|_2^2}{8}.\end{equation}

 Since $\sum_{g\in\mathbb F_2}\|y_g\|_2^2=\|p\theta_t(y)p\|_2^2\geqslant\frac{\|p\|_2^2}{4}$ by equation \ref{teta}, we get that $S_2=\mathbb F_2\setminus (S_1\cup\{e\})\not=\emptyset$. On the other hand, by Lemma \ref{gap}, any two elements $g,h\in S_2$ commute. If follows that we can find $k\in\mathbb F_2\setminus\{e\}$ such that $S_2\subset\Sigma$, where $\Sigma=\{k^n|n\in\mathbb Z\}$. Moreover, we can pick $k$ such that if $k'\in\mathbb F_2$ commutes with $k^m$, for some $m\in\mathbb Z\setminus\{0\}$, then $k'\in\Sigma$.

Further, since $\theta_t(A)\nprec_{\tilde M}N\rtimes\Sigma$, we can find $z\in\mathcal U(A)$ such that $\|E_{N\rtimes\Sigma}(p\theta_t(z)p)\|_2\leqslant\frac{\|p\|_2}{4}$.  Since $y,z\in\mathcal U(A)$,  by applying Claim 3 to $F=\{y,z\}$ we can find a unit vector $\zeta\in\mathcal H_0$ such that $\|(p\theta_t(y)p)\zeta-\zeta(p\theta_t(y)p)\|_2\leqslant \frac{c\|p\|_2}{4}$ and  $\|(p\theta_t(z)p)\zeta-\zeta(p\theta_t(z)p)\|_2\leqslant \frac{c\|p\|_2}{4}$.

Let $T_1=\{g\in\mathbb F_2| \|\lambda(g)(U(\zeta))-U(\zeta)\|>c\}$ and $T_2=\{g\in\mathbb F_2\setminus\{e\}|\|\lambda(g)(U(\zeta))-U(\zeta)\|\leqslant c\}$. Write $p\theta_t(z)p=\sum_{g\in\mathbb F_2}z_gu_g$, where $z_g\in N$. The same calculation as above then shows that \begin{equation}\label{6}\sum_{g\in T_1}\|y_g\|_2^2\leqslant\frac{\|p\|_2^2}{16}\;\;\;\text{and}\;\;\;\sum_{g\in T_1}\|z_g\|_2^2\leqslant\frac{\|p\|_2^2}{16}\end{equation} 

By combining inequalities  \ref{5} and \ref{6} it follows that $\sum_{g\in T_1\cup (S_1\cup\{e\})}\|y_g\|_2^2\leqslant\frac{3\|p\|_2^2}{16}.$ Since we also have that $\sum_{g\in\mathbb F_2}\|y_g\|_2^2=\|p\theta_t(y)p\|_2^2\geqslant\frac{\|p\|_2^2}{4}$, we get that $T_1\cup S_1\cup\{e\}\not=\mathbb F_2$. Hence $S_2\cap T_2\not=\emptyset$.

Fix $k'\in S_2\cap T_2$. If $k''\in T_2$, then  Lemma \ref{gap} implies that $k''$  commutes with $k'$. Since $k'\in S_2\subset\Sigma\setminus\{e\}$, we get that $k''\in\Sigma$ and therefore $T_2\subset\Sigma$.

Thus,  $T_2\cup\{e\}\subset\Sigma$ and so $\sum_{g\in T_2\cup\{e\}}\|z_g\|_2^2\leqslant \|E_{N\rtimes\Sigma}(p\theta_t(z)p)\|_2^2\leqslant\frac{\|p\|_2^2}{16}$. Since $T_1\cup T_2\cup\{e\}=\mathbb F_2$, combining this inequality with \ref{6} yields that $\sum_{g\in\mathbb F_2}\|z_g\|_2^2\leqslant \frac{\|p\|_2^2}{8}.$ This however contradicts the fact that $\|p\theta_t(z)p\|_2\geqslant\frac{\|p\|_2}{2}$ and finishes the proof.
\hfill$\square$

{\it Proof of Lemma \ref{calc}.} Write $\eta=\sum_{g\in\mathbb F_2}\eta_gu_ge_Nu_g^*$, where $\eta_g\in\mathbb C$, and $y=\sum_{k\in\mathbb F_2}y_ku_k$, where $y_k\in N$. Recall that the canonical semi-finite trace on $\langle\tilde M,e_N\rangle$ is given by $Tr(xe_Ny)=\tau(xy)$. If we denote by $(\sigma_g)_{g\in\mathbb F_2}$ the conjugation action of $\mathbb F_2$ on $N$ (i.e., $\sigma_g(x)=u_gxu_g^*$), then we have $$\langle y\eta,\eta y\rangle=\sum_{g,h,k,l\in\mathbb F_2}\langle y_ku_k\eta_gu_ge_Nu_g^*,\eta_hu_he_Nu_h^*y_lu_l\rangle=$$ $$\sum_{g,h,k,l\in\mathbb F_2}\eta_g\overline{\eta_h}Tr(y_ku_ku_g\;e_N\;u_g^*u_l^*y_l^*u_h\;e_N\;u_h^*)=$$ $$\sum_{g,h,k,l\in\mathbb F_2}\eta_g\overline{\eta_h}\tau(E_N(u_h^*y_ku_ku_g)E_N(u_g^*u_l^*y_l^*u_h)).$$ 

If $g$, $k$ are fixed and the expression $\tau(E_N(u_h^*y_ku_ku_g)E_N(u_g^*u_l^*y_l^*u_h))$ is non-zero, then $h=kg$ and $l=k$. Moreover, in this case this expression is equal to $\tau(\sigma_{(kg)^{-1}}(y_k)\sigma_{(kg)^{-1}}(y_k^*))=\|y_k\|_2^2$. Thus, we deduce that

$$\langle y\eta,\eta y\rangle=\sum_{g,k\in\mathbb F_2}\eta_g\overline{\eta_{kg}}\|y_k\|_2^2=\sum_{k\in\mathbb F_2}(\sum_{g\in\mathbb F_2}\eta_{k^{-1}g}\overline{\eta_g})\|y_k\|_2^2=$$ $$\sum_{k\in\mathbb F_2}\langle \lambda(k)(U(\eta)),U(\eta)\rangle\; \|E_N(yu_k^*)\|_2^2.$$

Since we also have that $\|y\eta\|_2=\|\eta y\|_2=\|y\|_2\|\eta\|_2$, the lemma follows. 
\hfill$\square$

{\it Proof of Lemma \ref{gap}.} Let $a$ and $b$ be generators of $\mathbb F_2$. Since $\mathbb F_2$ is non-amenable, there exists $c>0$ such that any non-zero vector $\eta\in\ell^2(\mathbb F_2)$ satisfies  $\|\lambda(a)(\eta)-\eta\|^2+\|\lambda(b)(\eta)-\eta\|^2> 2c^2\|\eta\|^2$.

Now, let $g,h\in\mathbb F_2$ such that $\|\lambda(g)(\eta)-\eta\|\leqslant c\|\eta\|$ and $\|\lambda(h)(\eta)-\eta\|\leqslant c\|\eta\|$, for some non-zero vector $\eta\in\ell^2(\mathbb F_2)$. From this we get that  $\|\lambda(g)(\eta)-\eta\|^2+\|\lambda(h)(\eta)-\eta\|^2\leqslant 2c^2\|\eta\|^2$.

 Let $\Delta<\mathbb F_2$ be the subgroup generated by $g$ and $h$, and  $\gamma:\Delta\rightarrow\mathcal U(\ell^2(\Delta))$ be the its left regular representation. Since   $\mathbb F_2=\sqcup_{g\in S}\Delta g$, for a set $S$ of representatives,  the restriction $\lambda_{|\Delta}$ is a subrepresentation of  $\oplus_{n=1}^{\infty}\gamma:\Delta\rightarrow\mathcal U(\oplus_{n=1}^{\infty}\ell^2(\Delta))$. 
If we write $\eta=(\eta_n)_{n=1}^{\infty}$, where $\eta_n\in\ell^2(\Delta)$, then  we can find $n$ such that $\|\gamma(g)(\eta_n)-\eta_n\|^2+\|\gamma(h)(\eta_n)-\eta_n\|^2\leqslant 2c^2\|\eta_n\|^2$ and $\eta_n\not=0$.

If $g$ and $h$ do not commute, then they generate a copy of $\mathbb F_2$. In other words, there exists an isomorphism $\rho:\Delta\rightarrow\mathbb F_2$ such that $\rho(g)=a$ and $\rho(h)=b$. In combination with the above, this leads to a contradiction. \hfill$\square$

\vskip 0.1in

{\it Proof of Theorem \ref{relamen2}}.  Recall that $B=P\bar{\otimes}Q_0$, $M_1=P\bar{\otimes}Q_1$ and $M_2=P\bar{\otimes}Q_2$.  Therefore, $M=P\bar{\otimes}Q$, where $Q=Q_1*_{Q_0}Q_2$. Also, recall that $\tilde M=M*_{B}(B\bar{\otimes}L(\mathbb F_2))$ and that $N=\{u_ge_Mu_g^*|g\in\mathbb F_2\}''$. We define $\tilde Q=Q*_{Q_0}(Q_0\bar{\otimes}L(\mathbb F_2))$ and $N_0=\{u_gQu_g^*|g\in\mathbb F_2\}''\subset \tilde Q$.  Note that $\tilde M=P\bar{\otimes}\tilde Q$ and that $N=P\bar{\otimes}N_0$.

We denote by $\{\alpha_t\}_{t\in\mathbb R}\subset\text{Aut}(\tilde Q)$ the free malleable deformation associated to the AFP decomposition $Q=Q_1*_{Q_0}Q_2$ (see Section \ref{ipp}). Then for every $x\in P$ and $y\in\tilde Q$ we have that $\theta_t(x\otimes y)=x\otimes\alpha_t(y)$.

Let $t\in (0,1)$. We claim that $\alpha_t(A_0)$ is amenable relative to $N_0$ inside $\tilde Q$. Once this claim is proven the conclusion follows by applying Theorem \ref{relamen} to the inclusion  $A_0\subset Q=Q_1*_{Q_0}Q_2$.

Since $\theta_t(A)$ is amenable relative to $N$ inside $\tilde M$, by \cite[Definition 2.2]{OP07} we can find a $\theta_t(A)$-central state $\Phi:\langle\tilde M,e_N\rangle\rightarrow\mathbb C$ such that  $\Phi_{|\tilde M}=\tau$.

Since $\tilde M=P\bar{\otimes}\tilde Q$ and that $N=P\bar{\otimes}N_0$, we have that
 $\langle\tilde M,e_N\rangle=P\bar{\otimes}\langle\tilde Q,e_{N_0}\rangle$. Define a state $\Psi:\langle\tilde Q,e_{N_0}\rangle\rightarrow\mathbb C$ by $\Psi(T)=\Phi(1\otimes T)$ and let $u\in\mathcal U$. Since $u\otimes\rho(u)\in A$ we have that $u\otimes\alpha_t(\rho(u))=\theta_t(u\otimes\rho(u))\in\theta_t(A).$ Thus for every $T\in\langle\tilde Q,e_{N_0}\rangle$ we have that $$\Psi(\alpha_t(\rho(u))T\alpha_t(\rho(u))^*)=\Phi(1\otimes \alpha_t(\rho(u))T\alpha_t(\rho(u))^*)=$$ $$\Phi((u\otimes\alpha_t(\rho(u))(1\otimes T)(u\otimes\alpha_t(\rho(u))^*)=\Phi(1\otimes T)=\Psi(T).$$
 
 Thus, $\Psi(\alpha_t(\rho(u))T)=\Psi(T\alpha_t(\rho(u))$, for every $u\in\mathcal U$ and $T\in\langle\tilde Q,e_{N_0}\rangle$.
 Since  $\{\alpha_t(\rho(u))|u\in\mathcal U\}$ generates $\alpha_t(A_0)$ and  $\Psi_{|\tilde Q}=\tau$, we get that $\Psi$ is $\alpha_t(A_0)$-central. Thus $\alpha_t(A_0)$ is amenable relative to $N_0$ inside $\tilde Q$. This proves the claim and finishes the proof.
  \hfill$\square$

\section{Property $\Gamma$ for subalgebras of AFP algebras}\label{Gamma}

Let $Q$ be a von Neumann subalgebra of an amalgamated free product algebra $M=M_1*_{B}M_2$. In this section we study the position of the relative commutant  $Q'\cap M^{\omega}$ inside $M^{\omega}$. We start by considering  the case $Q=M$.

\begin{lemma}\label{bar}  Let $(M_1,\tau_1)$ and $(M_2,\tau_2)$ be tracial von Neumann algebras with a common von Neumann subalgebra $B$ such that ${\tau_1}_{|B}={\tau_2}_{|B}$. Denote $M=M_1*_{B}M_2$. Assume that there exist unitary elements $u\in M_1$ and $v,w\in M_2$ such that $E_B(u)=E_B(v)=E_B(w)=E_B(w^*v)=0$. 

If $\omega$ is a free ultrafilter on $\mathbb N$, then $M'\cap M^{\omega}\subset B^{\omega}$. 

\end{lemma}
In the case $B=\mathbb C1$ this result was proved in \cite[Theorem 11]{Ba95}. The proof of Theorem \ref{bar} is a straightforward adaptation of the proof of \cite[Theorem 11]{Ba95} to the case when $B$ is arbitrary.
\begin{proof}
 We denote by $S_1\subset M$  the set of alternating words in $M_1\ominus B$ and $M_2\ominus B$ that begin in $M_1\ominus B$. Concretely,  $x\in S_1$ if we can write  $x=x_1x_2...x_n$, for some $x_1\in M_1\ominus B,x_2\in M_2\ominus B,x_3\in M_1\ominus B...$. Similarly, we denote by $S_2\subset M$ the set of alternating words in $M_1\ominus B$ and $M_2\ominus B$ that begin in $M_2\ominus B$. For $i\in\{1,2\}$, we denote by $\mathcal H_i\subset L^2(M)$ the $\|.\|_2$ closure of the linear span of $S_i$ and by $P_i$ the orthogonal projection onto $\mathcal H_i$. 
 
 Note that if $x\in M_1\ominus B$ and $y\in M_2\ominus B$, then $x\mathcal H_2x^*\subset\mathcal H_1$ and $y\mathcal H_1 y^*\subset\mathcal H_2$.
 The hypothesis therefore implies that \begin{equation}\label{subsp} u\mathcal H_2u^*\subset\mathcal H_1,\;\; v\mathcal H_1v^*\subset\mathcal H_2,\;\;w\mathcal H_1w^*\subset\mathcal H_2\;\;\text{and}\;\;v\mathcal H_1v^*\perp w\mathcal H_1w^* \end{equation}
 The last fact holds because $(w^*v)\mathcal H_1(w^*v)^*\subset\mathcal H_2$ and hence $(w^*v)\mathcal H_1(w^*v)^*\perp\mathcal H_1$.

Now, let $\xi\in L^2(M)$. Notice that if $P_{\mathcal K}$ is the orthogonal projection onto a closed subspace $\mathcal K\subset L^2(M)$ and $u\in \mathcal U(M)$, then $P_{u\mathcal K u^*}(\xi)=uP_{\mathcal K}(u^*\xi u)u^*$ and therefore  $\|P_{u\mathcal Ku^*}(\xi)\|_2=\|P_{\mathcal K}(u^*\xi u)\|_2$. By combining this fact with equation \ref{subsp} we get that \begin{equation}\label{orthog}\|P_2(u^*\xi u)\|_2\leqslant \|P_1(\xi)\|_2\;\;\text{and}\;\; \|P_1(v^*\xi v)\|_2^2+\|P_1(w^*\xi w)\|_2^2\leqslant \|P_2(\xi)\|_2^2\end{equation}

Let $x=(x_n)_n\in M'\cap M^{\omega}$. Then $\|u^*x_nu-x_n\|_2,\|v^*x_nv-x_n\|_2,\|w^*x_nw-x_n\|_2\rightarrow 0,$ as $n\rightarrow\omega$. Using this fact and applying \ref{orthog} to $\xi=x_n$ we get that $\lim_{n\rightarrow\omega}\|P_2(x_n)\|_2\leqslant \lim_{n\rightarrow\omega}\|P_1(x_n)\|_2$ and $\sqrt{2}\lim_{n\rightarrow\omega}\|P_1(x_n)\|_2\leqslant\lim_{n\rightarrow\omega} \|P_2(x_n)\|_2.$ Therefore, we have that $\|P_1(x_n)\|_2\rightarrow 0$ and $\|P_2(x_n)\|_2\rightarrow 0$, as $n\rightarrow\omega$.

Since $L^2(M)=L^2(B)\oplus\mathcal H_1\oplus\mathcal H_2$, it follows that $\lim_{n\rightarrow\omega}\|x_n-E_B(x_n)\|_2=0$ and thus $x\in B^{\omega}$.
\end{proof}

Lemma \ref{bar} implies that a large class of AFP groups give rise to II$_1$ factors without property $\Gamma$.

\begin{corollary}\label{inner}
Let $\Gamma=\Gamma_1*_{\Lambda}\Gamma_2$ be an amalgamated free product group such that $[\Gamma_1:\Lambda]\geqslant 2$ and $[\Gamma_2:\Lambda]\geqslant 3$. Assume that there exist $g_1,g_2,...,g_m\in\Gamma$ such that $\cap_{i=1}^mg_i\Lambda g_i^{-1}=\{e\}$. 

Then $L(\Gamma)$ is a II$_1$ factor without property $\Gamma$.

Moreover, $\Gamma$ is not  inner amenable, i.e. the unitary representation $\pi:\Gamma\rightarrow\mathcal U(\ell^2(\Gamma\setminus\{e\}))$ given by $\pi(g)(\delta_h)=\delta_{ghg^{-1}},$ for $g\in\Gamma$ and $h\in\Gamma\setminus\{e\}$, does not have almost invariant vectors.
\end{corollary}

\begin{proof}
Let $x=(x_n)_n\in L(\Gamma)'\cap L(\Gamma)^{\omega}$.
Firstly, by Lemma \ref{bar} we get that $x\in L(\Lambda)^{\omega}$.

Secondly, for $i\in\{1,2,...,m\}$, denote by $E_i$ the conditional expectation onto $L(g_i\Lambda g_i^{-1})$.  Then $E_i(x)=u_{g_i}E_{L(\Lambda)}(u_{g_i}^*xu_{g_i})u_{g_i}^*$, for every $x\in L(\Gamma)$. Since $(x_n)_n\in L(\Gamma)'\cap L(\Lambda)^{\omega}$ it follows that $\|E_i(x_n)-x_n\|_2\rightarrow 0$, as $n\rightarrow\omega$, for every $i\in\{1,2,...,m\}$. 

On the other hand, since $\cap_{i=1}^mg_i\Lambda g_i^{-1}=\{e\}$, we derive that $E_1E_2...E_m(x)=\tau(x)1$, for all $x\in L(\Gamma)$. Altogether, it follows that $\|\tau(x_n)1-x_n\|_2\rightarrow 0$, as $n\rightarrow\omega$, i.e. $(x_n)_n\in\mathbb C1$.

We leave it the reader to modify the above proof to show that $\Gamma$ is indeed non-inner amenable.
\end{proof}

Next, we show that if a von Neumann subalgebra $Q\subset M=M_1*_{B}M_2$ is ``large" (i.e. if conditions (2) and (3) below are not satisfied) then a corner of $Q'\cap M^{\omega}$ embeds into $B^{\omega}$. Thus, the phenomenon from Theorem \ref{bar} extends in some sense to arbitrary subalgebras $Q\subset M$.

\begin{theorem}\label{afpgamma} Let $(M_1,\tau_1)$ and $(M_2,\tau_2)$ be tracial von Neumann algebras with a common von Neumann subalgebra $B$ such that ${\tau_1}_{|B}={\tau_2}_{|B}$.
Let $M=M_1*_{B}M_2$ and $Q\subset pMp$ be a von Neumann subalgebra, for some projection $p\in M$.
 Let $\omega$ be a free ultrafilter on $\mathbb N$. Denote by $P$ the von Neumann subalgebra of $M^{\omega}$ generated by $M$ and $B^{\omega}$.

Then one of the following conditions holds true:

\begin{enumerate}
\item $Q'\cap (pMp)^{\omega}\subset  P$ and $Q'\cap (pMp)^{\omega}\prec_{P}B^{\omega}$. 
\item $\mathcal N_{pMp}(Q)''\prec_{M}M_i$, for some $i\in\{1,2\}$.
\item $Qp'$ is amenable relative to $B$, for some non-zero projection $p'\in\mathcal Z(Q'\cap pMp)$.
\end{enumerate}
\end{theorem}
To prove Theorem \ref{afpgamma} we will need the following result.

\begin{theorem}\cite{CH08}\label{ch} Let $(M_1,\tau_1)$ and $(M_2,\tau_2)$ be tracial von Neumann algebras with a common von Neumann subalgebra $B$ such that ${\tau_1}_{|B}={\tau_2}_{|B}$.
Let $M=M_1*_{B}M_2$ and $Q\subset pMp$ be a von Neumann subalgebra, for some projection $p\in M$. 

Then one of the following conditions holds:

\begin{enumerate}
\item $Q'\cap pMp\prec_{M}B$.
\item $\mathcal N_{pMp}(Q)''\prec_{M}M_i$, for some $i\in\{1,2\}$.
\item $Qp'$ is amenable relative to $B$, for some non-zero projection $p'\in\mathcal Z(Q'\cap pMp)$.
\end{enumerate}

\end{theorem}

In the case when $B$ is amenable and $Q$ has no amenable direct summand this result was proved by I. Chifan and C. Houdayer  \cite[Theorem 1.1]{CH08}. 
The argument that we include below follows closely their proof. 

Note that part (1) of Theorem \ref{afpgamma} implies part (1) of of Theorem \ref{ch}.  Indeed, if $Q'\cap (pMp)^{\omega}\prec_{P}B^{\omega}$, then $(Q'\cap pMp)^{\omega}\prec_{M^{\omega}}B^{\omega}$. This readily implies that $Q'\cap pMp\prec_{M}B$.
Therefore Theorem \ref{afpgamma} is stronger than Theorem \ref{ch}.

Before proceeding to the proofs of Theorems \ref{afpgamma} and \ref{ch}, let us fix some notations.
 Let $\tilde M=M*(B\bar{\otimes}L(\mathbb F_2))$ and $\{\theta_t\}_{t\in\mathbb R}$ be the automorphisms of $\tilde M$ defined in Section \ref{ipp}. We extend $\theta_t$ to an automorphism of ${\tilde M}^{\omega}$ by putting $\theta_t((x_n)_n)=(\theta_t(x_n))_n$.
 For  $x\in\tilde M^{\omega}$, we denote $$\delta_t(x)=\theta_t(x)-E_{M^{\omega}}(\theta_t(x))\in\tilde M^{\omega}\ominus M^{\omega}.$$ Note that if $x\in\tilde M$, then $\delta_t(x)\in\tilde M\ominus M$.

Let $\beta$ be the  automorphism of $\tilde M$ satisfying $\beta(x)=x$ if $x\in M$, $\beta(u_{a_1})=u_{a_1}^*$ and $\beta(u_{a_2})=u_{a_2}^*$, where $a_1,a_2$ are the generators of $\mathbb F_2$ chosen in Section \ref{ipp}. We still denote by $\beta$ the extension of $\beta$ to $\tilde M^{\omega}$. It is easy to check that $\beta^2=id_{\tilde M^{\omega}}$ and $\beta\theta_t\beta=\theta_{-t}$, for all $t\in\mathbb R$.

By \cite[Lemma 2.1]{Po06a}, the existence of  $\beta$ implies that \begin{equation}\label{malleable}\|\theta_{2t}(x)-x\|_2\leqslant 2\|\delta_t(x)\|_2,\;\;\text{for all}\;\; x\in M\;\;\text{and every}\;\;t\in\mathbb R.\end{equation}

In the proofs of Theorems \ref{afpgamma} and \ref{ch} we assume for simplicity that $p=1$, the general case being treated similarly. We continue with the following lemma which is key in both proofs.

\begin{lemma}\label{deltat} Let $(M_1,\tau_1)$ and $(M_2,\tau_2)$ be tracial von Neumann algebras with a common von Neumann subalgebra $B$ such that ${\tau_1}_{|B}={\tau_2}_{|B}$.
Let $Q\subset M=M_1*_{B}M_2$ be a von Neumann subalgebra such that $Qp'$ is not amenable relative to $B$, for any non-zero projection $p'\in\mathcal Z(Q'\cap M)$.

 Then we have that $\;\;\sup_{x\in (Q'\cap M^{\omega})_1}\|\delta_t(x)\|_2\rightarrow 0$, as $t\rightarrow 0$.
\end{lemma}
\begin{proof}It is easy to see that the map $\mathbb R\ni t\rightarrow \|\delta_t(x)\|_2\in [0,\infty)$ is even on $\mathbb R$, and decreasing on $[0,\infty)$, for every $x\in\tilde M^{\omega}$. Thus, if the lemma is false, then there exists $c>0$ such that $\sup_{x\in (Q'\cap M^{\omega})_1}\|\delta_t(x)\|_2>c$, for every $t\in\mathbb R\setminus\{0\}$.

For $m\geqslant 1$, put $t_m=2^{-m}$. Let $x_m\in (Q'\cap M^{\omega})_1$ such that $\xi_m=\delta_{t_m}(x_m)$ satisfies $\|\xi_m\|_2>c$. 

Fix  $y\in M$ and $z\in (Q)_1$. Then we have that $$\|y\xi_m\|_2=\|(1-E_{M^{\omega}})(y\theta_{t_m}(x_n))\|_2\leqslant \|y\theta_{t_m}(x_m)\|_2\leqslant \|y\|_2.$$ Also, since $zx_m=x_mz$, by using S. Popa's spectral gap argument \cite{Po06b} we get that $$\|z\xi_m-\xi_m z\|_2=\|(1-E_M)(z\theta_{t_m}(x_m)-\theta_{t_m}(x_m)z)\|_2\leqslant \|z\theta_{t_m}(x_m)-\theta_{t_m}(x_m)z\|_2= $$ $$\|\theta_{-t_m}(z)x_m-x_m\theta_{-t_m}(z)\|_2\leqslant 2\|\theta_{-t_m}(z)-z\|_2\longrightarrow 0.$$ 

By writing $\xi_m=(\xi_{m,n})_n$, where $\xi_{m,n}\in \tilde M\ominus M$,  we find a net $\eta_k\in\tilde M\ominus M$ such that $\|\eta_k\|_2>c$, $\limsup_k\|y\eta_k\|_2\leqslant \|y\|_2$, for every $y\in M$, and $\|z\eta_k-\eta_k z\|_2\rightarrow 0$, for every $z\in Q$.

Now, since $\tilde M=M*(B\bar{\otimes}L(\mathbb F_2))$, by Lemma \ref{BM} we have that $L^2(\tilde M)\ominus L^2(M)\cong L^2(M){\otimes}_B\mathcal K$, for some $B$-$M$ bimodule $\mathcal K$. 
We may  therefore apply Lemma \ref{op} to conclude that $Qp'$ is amenable relative to $B$, for a non-zero projection $p'\in\mathcal Z(Q'\cap M)$, which gives a contradiction.
\end{proof}

{\it Proof of Theorem \ref{ch}}. Assuming that condition (3) is false, we prove that either (1) or (2) holds. 

Since $Q'\cap M\subset Q'\cap M^{\omega}$,  Lemma \ref{deltat} implies that $\sup_{x\in (Q'\cap M)_1}\|\delta_t(x)\|_2\rightarrow 0$, as $t\rightarrow 0$. Together with  inequality \ref{malleable} this yields $t>0$ such that that $||\theta_{t}(x)-x||_2\leqslant\frac{1}{2}$, for all $x\in (Q'\cap M)_1$. 

Thus, $\tau(\theta_t(u)u^*)\geqslant\frac{1}{2}$, for every $u\in\mathcal U(Q'\cap M)$.
Applying Theorem \ref{ipp} gives that either $Q'\cap M\prec_{M}B$ or $\mathcal N_{M}(Q'\cap M)''\prec_{M}M_i$, for some $i\in\{1,2\}$. Since $\mathcal N_{M}(Q)\subset\mathcal N_{M}(Q'\cap M)$, this finishes the  proof.
\hfill$\square$

In the proof of Theorem \ref{afpgamma} we will also use the following technical result:

\begin{lemma}\label{P}
  Let $\tilde P$ be the von Neumann subalgebra of $\tilde M^{\omega}$ generated by $\tilde M$ and $B^{\omega}$.
  
Then we have
\begin{enumerate}
\item $M_1^{\omega}$ and $M_2^{\omega}$ are freely independent over $B^{\omega}$, 
\item $M^{\omega}\perp (\tilde P\ominus P)$ and
 \item $(\tilde M\ominus M)(M^{\omega}\ominus P)\perp M^{\omega}(\tilde M\ominus M)$.
 \end{enumerate}
\end{lemma}

\begin{proof} Let $x_1\in M_{i_1}^{\omega}\ominus B^{\omega},x_2\in M_{i_2}^{\omega}\ominus B^{\omega}$,..., $x_m\in M_{i_m}^{\omega}\ominus B^{\omega}$, for some indices $i_1,i_2,...,i_m\in\{1,2\}$ such that $i_k\not= i_{k+1}$, for all $1\leqslant k\leqslant m-1$. Then we can represent $x_k=(x_{k,n})_n$, where $x_{k,n}\in M_{i_k}\ominus B$, for all $n$ 
and every $1\leqslant k\leqslant m$. Since $E_{B^{\omega}}(x_1x_2...x_m)=\lim_{n\rightarrow\omega}E_B(x_{1,n}x_{2,n}...x_{m,n})=0$, the first assertion follows.

Towards the second assertion, define $P_1=\{M_1,B^{\omega}\}'',P_2=\{M_2,B^{\omega}\}''$ and $P_3=\{B\bar{\otimes}L(\mathbb F_2),B^{\omega}\}''$. All of these algebras contain $B^{\omega}$ and we have that $P_1\subset M_1^{\omega}$, $P_2\subset M_2^{\omega}$ and $P_3\subset (B\bar{\otimes}L(\mathbb F_2))^{\omega}$.
 Now, the first assertion implies that $M_1^{\omega}$, $M_2^{\omega}$ and $(B\bar{\otimes}L(\mathbb F_2))^{\omega}$ are freely independent over $B^{\omega}$. Since $P=\{P_1,P_2\}''$ and $\tilde P=\{P_1,P_2,P_3\}''$, we deduce that $\tilde P=P*_{B^{\omega}}P_3$.
 
 This implies that $\tilde P\ominus P$ is contained in the $\|.\|_2$-closure of the linear span of elements of the form $x=v_0w_1v_1...v_{m-1}w_mv_m$, where $v_0,v_m\in P_3,v_1,...,v_{m-1}\in P_3\ominus B^{\omega},$ and $w_1,...,w_m\in  P\ominus B^{\omega}$, for some $m\geqslant 1$. Since $P\ominus B^{\omega}\subset M^{\omega}\ominus B^{\omega}$ and 
 $P_3\ominus B^{\omega}\subset (B\bar{\otimes}L(\mathbb F_2))^{\omega}\ominus B^{\omega}$, we can represent $v_i=(v_{i,n})_n$ and $w_i=(w_{i,n})_n$, where 
 $v_{0,n},v_{m,n}\in B\bar{\otimes}L(\mathbb F_2)$, $v_{1,n},...,v_{m-1,n}\in (B\bar{\otimes}L(\mathbb F_2))\ominus B$, and $w_{1,n},...,w_{m,n}\in M\ominus B$, for all $n$.
It is now clear that $x=(v_{0,n}w_{1,n}v_{1,n}...v_{{m-1},n}w_{m,n}v_{m,n})_n$ belongs to $\tilde M^{\omega}\ominus M^{\omega}$.
This shows that $\tilde P\ominus P\subset\tilde M^{\omega}\ominus M^{\omega}$, thereby proving (2).
 
 Finally, let $z_1,z_2\in \tilde M\ominus M$, $y_1\in M^{\omega}\ominus P$ and $y_2\in M^{\omega}$ such that $\|y_1\|,\|y_2\|\leqslant 1$. Write $y_1=(y_{1,n})_n, y_2=(y_{2,n})_n$, where $y_{1,n},y_{2,n}\in (M)_1$. 
 Our goal is to prove that $\langle z_1y_1,y_2z_2\rangle=0$ or, equivalently, that $\lim_{n\rightarrow\omega}\langle z_1y_{1,n},y_{2,n}z_2\rangle=0$.
 
  Since $\tilde M=M*_{B}(B\bar{\otimes}L(\mathbb F_2))$, by Lemma \ref{BM} we can find a $M$-$B$ bimodule $\mathcal K$ such that $L^2(\tilde M)\ominus L^2(M)=\mathcal K{\otimes}_BL^2(M)$. Viewing $z_1,z_2$ as vectors in $L^2(\tilde M)\ominus L^2(M)$ and using approximations in $\|.\|_2$, we may assume that $z_1=\xi_1\otimes_B\eta_1$, $z_2=\xi_2\otimes_B\eta_2$, where $\xi_1,\xi_2\in\mathcal K$ and  $\eta_1,\eta_2\in M$. Moreover, we may take $\xi_1$ to be right bounded, i.e. such that $\|\xi_1 y\|_2\leqslant C\|y\|_2$, for all $y\in M$, for some constant $C>0$. By using the definition of Connes' tensor product we get that $$|\langle z_1y_{1,n},y_{2,n}z_2\rangle|=|\langle y_{2,n}^*\xi_1\otimes_B\eta_1y_{1,n},\xi_2\otimes_B\eta_2\rangle|=|\langle y_{2,n}^*\xi_1E_B(\eta_1y_{1,n}\eta_2^*),\xi_2\rangle|\leqslant$$ $$ C\|E_B(\eta_1y_{1,n}\eta_2^*)\|_2\|\xi_2\|_2. $$
  
  Since $y_1\perp P$ and $\eta_1^*B^{\omega}\eta_2\subset P$, we get  that $y_1\perp \eta_1^*B^{\omega}\eta_2$. Hence,  $\lim_{n\rightarrow\omega}\|E_B(\eta_1y_{1,n}\eta_2^*)\|_2=\|E_{B^{\omega}}(\eta_1y_1\eta_2^*)\|_2=0$, which proves the last assertion.
 \end{proof}

To prove Theorem \ref{afpgamma} we adapt the proof of \cite[Lemma 3.3]{Io10} (see also the proof of \cite[Theorem 3.8]{Bo12}) to the case of AFP algebras.  In the proof of Theorem \ref{afpgamma} we apply Theorem 6.4 and \cite[Theorems 1.1 and 3.1]{IPP05} to {\it non-separable} tracial von Neumann algebras. While these results are only stated for separable algebras, their proofs can be easily modified to handle non-separable algebras. We leave the details to the reader.

{\it Proof of Theorem \ref{afpgamma}}. For simplicity, we assume that $p=1$.
Assuming that (2) and (3) are false, we will deduce that (1) holds. 
The proof is divided between two claims, each proving one assertion from (1).

{\bf Claim 1}. $Q'\cap M^{\omega}\subset P$.

{\it Proof of Claim 1.} Assume by contradiction that there exists $x\in Q'\cap M^{\omega}$ such that $x\not\in P$ and put $y=x-E_P(x)\not=0$. Fix $z\in (Q)_1$ and $t\in\mathbb R$.

 Since $E_{M^{\omega}}(\theta_t(z))=(E_{M^{\omega}}\circ E_{\tilde M})(\theta_t(z))=E_M(\theta_t(z))$ and $y\in M^{\omega}$ we get that

\begin{equation}\label{unu} \|\delta_t(z)y-y\delta_t(z)\|_2=\|(1-E_M)(\theta_t(z))y-y(1-E_M)(\theta_t(z))\|_2=\end{equation}
$$\|(1-E_{M^{\omega}})(\theta_t(z)y-y\theta_t(z))\|_2\leqslant \|\theta_t(z)y-y\theta_t(z)\|_2$$

Since $zx=xz$ and $z\in M\subset P$, we get that $zy=yz$. Thus, we derive that 

\begin{equation}\label{doi}\|\theta_t(z)y-y\theta_t(z)\|_2=\|z\theta_{-t}(y)-\theta_{-t}(y)z\|_2\leqslant 2\|\theta_{-t}(y)-y\|_2=2\|\theta_t(y)-y\|_2\end{equation}

 On the other hand, since $x\in M^{\omega}$, Lemma \ref{P} (2) gives that $E_{\tilde P}(x)=E_P(x)$. Since $\theta_t$ leaves $\tilde P$ globally invariant we conclude that $\theta_t(E_P(x))=\theta_t(E_{\tilde P}(x))=E_{\tilde P}(\theta_t(x))$. As a consequence, we have \begin{equation}\label{trei} \|\theta_t(y)-y\|_2=\|(1-E_{\tilde P})(\theta_t(x)-x)\|_2\leqslant \|\theta_t(x)-x\|_2\end{equation}

By combining  \ref{unu}, \ref{doi} and \ref{trei} we get that $\|\delta_t(z)y-y\delta_t(z)\|_2\leqslant 2\|\theta_t(x)-x\|_2$.

Since $\delta_t(z)\in\tilde M\ominus M$ and $y\in M^{\omega}\ominus P$, Lemma \ref{P} (3) implies that $\delta_t(z)y\perp y\delta_t(z)$.  Therefore we derive  that $\|\delta_t(z)y\|_2\leqslant 2\|\theta_t(x)-x\|_2$. Since  $$\|\delta_t(z)y-\delta_t(zy)\|_2\leqslant \|\theta_t(z)y-\theta_t(zy)\|_2\leqslant \|\theta_t(y)-y\|_2,$$ we altogether deduce that $\|\delta_t(zy)\|_2\leqslant 3\|\theta_t(x)-x\|_2$, for every $z\in (Q)_1$ and $t\in\mathbb R$. 

By using this inequality together with \ref{malleable} and \ref{trei} we derive that \begin{equation}\label{patru}||\theta_t(z)y-zy||_2\leqslant \|\theta_t(zy)-zy\|_2+\|\theta_t(y)-y\|_2\leqslant \end{equation} $$2\|\delta_{\frac{t}{2}}(zy)\|_2+\|\theta_t(y)-y\|_2\leqslant 6\|\theta_{\frac{t}{2}}(x)-x\|_2+\|\theta_t(x)-x\|_2\leqslant$$ $$ 12\|\delta_{\frac{t}{4}}(x)\|_2+2\|\delta_{\frac{t}{2}}(x)\|_2,\;\;\text{for all}\;\;z\in (Q)_1\;\;\text{and}\;\;t\in\mathbb R.$$

Now, since (3) is assumed false, Lemma \ref{deltat} implies that 
 $\sup_{x\in (Q'\cap M^{\omega})_1}\|\delta_t(x)\|_2\rightarrow 0$, as $t\rightarrow 0$.
In combination with \ref{patru} it follows that we can find $t>0$ such that $\|\theta_t(z)y-zy\|_2\leqslant\frac{\|y\|_2}{2},$ for all $z\in (Q)_1$. Thus, if we let $w=E_{\tilde M}(yy^*)$, then   $$\Re\;\tau(\theta_t(z)wz^*)=\Re\;\tau(\theta_t(z)yy^*z^*)\geqslant\frac{
\|y\|_2^2}{2},\;\;\text{for all}\;\; z\in\mathcal U(Q).$$
 By using a standard averaging argument we can find $0\not=v\in \tilde M$ such that $\theta_t(z)v=vz,$ for all $z\in Q$. By \cite[Theorem 3.1]{IPP05} we would conclude that $Q\prec_{M}M_i$, for some $i\in\{1,2\}$.
 
 If we denote $\mathcal N=\mathcal N_{M}(Q)''$, then  \cite[Theorem 1.1]{IPP05} would imply that either $\mathcal N\prec_{M}M_1$, $\mathcal N\prec_{M}M_2$ or  $Q\prec_{M}B$. Since the last condition implies that there is a non-zero projection $p'\in\mathcal Z(Q'\cap M)$ such that $Qp'$ is amenable relative to $B$, we altogether get a contradiction.\hfill$\square$
 
To end the proof we are left with showing:

{\bf Claim 2}. $Q'\cap M^{\omega}\prec_{P}B^{\omega}$.

{\it Proof of Claim 2.} Recall from the proof of Lemma \ref{P} that $P_1=\{M_1,B^{\omega}\}''$ and $P_2=\{M,B^{\omega}\}''$ are freely independent over $B^{\omega}$, and that $P=P_1*_{B^{\omega}}P_2$.

 By applying Theorem \ref{ch} to the inclusion $Q\subset P$ it follows that we are in one of the following three cases: (a) $Q'\cap P\prec_{P}B^{\omega}$, (b) $\mathcal N_{P}(Q)''\prec_{P}P_i$, for some $i\in\{1,2\}$, or  (c) $Qz$ is amenable relative to $B^{\omega}$ inside $P$, for some non-zero projection $z\in\mathcal Z(Q'\cap P)$.
 
 In case (a), Claim 1 implies that $Q'\cap M^{\omega}=Q'\cap P\prec_{P}B^{\omega}$ and thus (1) is satisfied. 
 Let us  show that cases (b) and (c) contradict our assumption that conditions (2) and (3) are false. 
 
 Firstly, since $\mathcal N=\mathcal N_{M}(Q)''\subset\mathcal N_P(Q)''$, $P_i\subset M_i^{\omega}$ and $P\subset M^{\omega}$, case (b) implies that $\mathcal N\prec_{M^{\omega}}M_i^{\omega}$. By Remark \ref{embedamen} it follows that $\mathcal Np_0$ is amenable relative to $M_i^{\omega}$ inside $M^{\omega}$, for some non-zero projection $p_0\in\mathcal N'\cap M^{\omega}$. Lemma \ref{24} further implies that $\mathcal Np'$ is amenable relative to $M_i$ inside $M$, for some non-zero projection $p'\in\mathcal N'\cap M$.
 By Corollary \ref{211} we get that either (b$_1$) $\mathcal N p'$ is amenable relative to $B$ inside $M$ or (b$_2$) $\mathcal N\prec_{M}M_i$.
 In the case (b$_1$) we get in particular that $Qp''$ is amenable relative to $B$ inside $M$, contradicting the assumption that (3) is false. In turn, case (b$_2$) contradicts the assumption that (2) does not hold.

Finally, in case (c), Lemma \ref{24} implies that $Qp'$ is amenable relative to $B$, for some non-zero projection $p'\in\mathcal Z(Q'\cap M)$. In other words, (3) holds, a contradiction.
\hfill$\square$

\section{Uniqueness of Cartan subalgebras for II$_1$ factors \\ arising from  actions of AFP groups}\label{proofs}

The main goal of this section is to prove Theorem \ref{main} and derive several consequences.

\subsection{Uniqueness of Cartan subalgebras}  
Towards proving Theorem \ref{main}  we first establish a general technical result.

\begin{theorem} \label{maintech} Let $\Gamma_1$ and $\Gamma_2$ be two countable groups with a common subgroup $\Lambda$ such that $[\Gamma_1:\Lambda]\geqslant 2$ and  $[\Gamma_2:\Lambda]\geqslant 3$.  Denote $\Gamma=\Gamma_1*_{\Lambda}\Gamma_2$ and suppose that there exist $g_1,g_2,...,g_n\in\Gamma$ such that $\cap_{i=1}^n g_i\Lambda g_i^{-1}$ is finite.  

Let $\Gamma\curvearrowright (D,\tau)$ be any trace preserving action of $\Gamma$ on a tracial von Neumann algebra $(D,\tau)$. Denote $M=D\rtimes\Gamma$ and suppose that $M$ is a factor.

If $A$ is a regular amenable von Neumann subalgebra of $M$, then $A\prec_{M}D$.

\end{theorem}
Before proceeding to the proof of Theorem \ref{maintech}, let us introduce some notations that will essentially allow us to reduce to the case when $\cap_{i=1}^ng_i\Lambda g_i^{-1}$ is trivial and not only finite. 

Since $\cap_{i=1}^{n}g_i\Lambda g_i^{-1}$ is finite, $\Sigma=\cap_{g\in\Gamma}g\Lambda g^{-1}$ is a finite  group  and there exist $h_1,h_2,...,h_m\in\Gamma$ such that $\Sigma=\cap_{j=1}^mh_j\Lambda h_j^{-1}$. Since $\Sigma<\Lambda$ is a normal subgroup of $\Gamma$, we can define the following groups $\Gamma'=\Gamma/\Sigma$, $\Gamma_1'=\Gamma_1/\Sigma$, $\Gamma_2'=\Gamma_2/\Sigma$ and $\Lambda'=\Lambda/\Sigma$. Note that $\Gamma'=\Gamma_1'*_{\Lambda'}\Gamma_2'$ and let $\rho:\Gamma\rightarrow\Gamma'$ be the quotient homomorphism. Note also that $\cap_{j=1}^mk_j\Lambda'k_j^{-1}=\{e\}$, where $k_j=\rho(h_j)$.

Denote $\mathcal M=M\bar{\otimes}L(\Gamma')$ and let  $\Delta:M\rightarrow\mathcal M$ be the {\it comultiplication} \cite{PV09} defined by $$\Delta(au_g)=au_g\otimes u_{\rho(g)},\;\;\;\text{for every}\;\;\;a\in D\;\;\;\text{and all}\;\;\;g\in\Gamma.$$ We next record a property of $\Delta$ that will be of later use.

\begin{lemma}\label{delta} Let $Q\subset M$ be a von Neumann subalgebra and $\Gamma_0<\Gamma$ be a subgroup.

If $\Delta(Q)\prec_{\mathcal M}M\bar{\otimes}L(\rho(\Gamma_0))$, then $Q\prec_{M}D\rtimes\Gamma_0$.

\end{lemma}

{\it Proof of Lemma \ref{delta}}. Assume by contradiction that $Q\nprec_{M}D\rtimes\Gamma_0$. Then we can find a sequence of unitaries $u_n\in Q$ such that $\|E_{D\rtimes\Gamma_0}(xu_ny)\|_2\rightarrow 0$, for all $x,y\in M$. We claim that $\|E_{M\bar{\otimes}L(\rho(\Gamma_0))}(v\Delta(u_n)w)\|_2\rightarrow 0$, for all $v,w\in\mathcal M$. This will provide the desired contradiction.

To prove the claim, by Kaplansky's density theorem, we may assume that $v=1\otimes u_{\rho(h)}$ and $w=1\otimes u_{\rho(k)}$, for some $h,k\in\Gamma$. For every $n$, write $u_n=\sum_{g\in\Gamma}x_{n,g}u_g$, where $x_{n,g}\in D$. Then $\Delta(u_n)=\sum_{g\in\Gamma}x_{n,g}u_g\otimes u_{\rho(g)}$. Since $\ker(\rho)=\Sigma$, it follows that $$E_{M\bar{\otimes}L(\rho(\Gamma_0))}(v\Delta(u_n)w)=\sum_{g\in\Gamma}x_{n,g}u_g\otimes E_{L(\rho(\Gamma_0))}(u_{\rho(hgk)})=\sum_{g\in h^{-1}\Gamma_0\Sigma k^{-1}}x_{n,g}u_g\otimes u_{\rho(hgk)}.$$ 

Further, since $\Sigma$ is finite we deduce that $$\|E_{M\bar{\otimes}L(\rho(\Gamma_0))}(v\Delta(u_n)w)\|_2^2=\sum_{g\in h^{-1}\Gamma_0\Sigma k^{-1}}\|x_{n,g}\|_2^2\leqslant\sum_{l\in\Sigma}\|E_{D\rtimes\Gamma_0}(u_hu_nu_{kl})\|_2^2.$$

Since $\|E_{D\rtimes\Gamma_0}(u_hu_nu_{kl})\|_2\rightarrow 0$, as $n\rightarrow\infty$, the lemma is proven.
\hfill$\square$

\vskip 0.1in

{\it Proof of Theorem \ref{maintech}}. Define $\mathcal M_1=M\bar{\otimes}L(\Gamma_1')$, $\mathcal M_2=M\bar{\otimes}L(\Gamma_2')$ and $B=M\bar{\otimes}L(\Lambda')$. Then we have that $\mathcal M={\mathcal M_1}*_{B}\mathcal M_2$.

Define $\tilde{\mathcal M}=\mathcal M*_B(B\bar{\otimes}L(\mathbb F_2))$ and let $\{\theta_t\}_{t\in\mathbb R}\subset$ Aut$(\tilde{\mathcal M})$ be the  deformation defined in Section \ref{ipp}. Also, let $N$ be the von Neumann subalgebra of $\tilde{\mathcal M}$ generated by $\{u_g\mathcal Mu_g^*|g\in\mathbb F_2\}$.
Recall from Section \ref{conjugacy} that $\tilde{\mathcal M}=N\rtimes\mathbb F_2$, where $\mathbb F_2=\{u_g\}_{g\in\mathbb F_2}$ acts on $N$ by conjugation.

Let $t\in (0,1)$ and consider the amenable von Neumann subalgebra $\theta_t(\Delta(A))\subset \tilde{\mathcal M}$. By S. Popa and S. Vaes' dichotomy (Theorem \ref{pv})
we get that either $\theta_t(\Delta(A))\prec_{\tilde{\mathcal M}}N$ or $\mathcal N_{\tilde{\mathcal M}}(\theta_t(\Delta(A)))''$ is amenable relative to $N$ inside $\tilde{\mathcal M}$. 

Since $A$ is regular in $M$, we have that $\theta_t(\Delta(M))\subset\mathcal N_{\tilde{\mathcal M}}(\theta_t(\Delta(A))''.$ 
Therefore, we are in one of the following two cases:
\vskip 0.05in
{\bf Case 1}. There exists $t\in (0,1)$ such that $\theta_t(\Delta(A))\prec_{\tilde{\mathcal M}}N$.

{\bf Case 2}. For every $t\in (0,1)$ we have that $\theta_t(\Delta(M))$ is amenable relative to $N$ inside $\tilde{\mathcal M}$.

\vskip 0.05in
In {\bf Case 1},  Theorem \ref{inter} gives that either $\Delta(A)\prec_{\mathcal M}B$ or $\mathcal N_{\mathcal M}(\Delta(A))''\prec_{\mathcal M}\mathcal M_i$, for some $i\in\{1,2\}$. Since $A$ is regular in $M$, the latter condition implies that $\Delta(M)\prec_{\mathcal M}\mathcal M_i$.

By using Lemma \ref{delta} we derive that either $A\prec_{M}D\rtimes\Lambda$ or $M\prec_{M}D\rtimes\Gamma_i$, for some $i\in\{1,2\}$.
If $A\prec_{M}D\rtimes\Lambda$, then as $M$ is a factor, \cite[Proposition 8]{HPV10}  implies that $A\prec_{M}D\rtimes (\cap_{i=1}^ng_i\Lambda g_i^{-1})$.
Since $\cap_{i=1}^ng_i\Lambda g_i^{-1}$ is finite, we conclude that $A\prec_{M}D$, as claimed.

 Now, since $[\Gamma_1:\Lambda]\geqslant 2$ and $[\Gamma_2:\Lambda]\geqslant 2$, we can find $g_1\in\Gamma_1\setminus\Lambda$ and $g_2\in\Gamma_2\setminus\Lambda$. Let $u=u_{g_1g_2}\in\mathcal U(L(\Gamma))$. Then we have that $\|E_{D\rtimes\Gamma_i}(xu^ny)\|_2\rightarrow 0$, for every $x,y\in M$ and $i\in\{1,2\}$. Thus, $L(\Gamma)\nprec_{M}D\rtimes\Gamma_i$ and hence  $M\nprec_{M}D\rtimes\Gamma_i$. This shows that the second alternative is impossible and  finishes the proof of {\bf Case 1}.

In {\bf Case 2}, since  $[\Gamma_1':\Lambda']\geqslant 2$, $[\Gamma_2':\Lambda']\geqslant 3$ and $\cap_{j=1}^mk_j\Lambda'k_j^{-1}=\{e\}$, Corollary \ref{inner} implies that  $L(\Gamma')'\cap L(\Gamma')^{\omega}=\mathbb C1$.

Note that $u_g\otimes u_{\rho(g)}\in\Delta(M)$, for every $g\in\Gamma$. Moreover, the von Neumann algebra $A_0$ generated by $\{u_{\rho(g)}\}_{g\in\Gamma}$ is equal to $L(\Gamma')$ and satisfies $A_0'\cap L(\Gamma')^{\omega}=\mathbb C1$. Since $\theta_t(\Delta(M))$ is amenable relative to $N$, for any $t\in (0,1)$, by Theorem \ref{relamen2} we deduce that either $L(\Gamma')\prec_{L(\Gamma')}L(\Gamma_i')$, for some $i\in\{1,2\}$, or $L(\Gamma')$ is amenable relative $L(\Lambda')$ inside $L(\Gamma')$.

Since $[\Gamma_1':\Lambda']\geqslant 2$ and $[\Gamma_2':\Lambda']\geqslant 2$, we can choose $g_1\in\Gamma_1'\setminus\Lambda'$ and $g_2\in\Gamma_2'\setminus\Lambda'$. Then $u=u_{g_1g_2}\in L(\Gamma')$ satisfies $\|E_{L(\Gamma_1')}(xu^ny)\|_2\rightarrow 0$ and $\|E_{L(\Gamma_2')}(xu^ny)\|_2\rightarrow 0$, for all $x,y\in L(\Gamma')$, showing that the first alternative is impossible.

Finally, if $L(\Gamma')$ is amenable relative to $L(\Lambda')$ inside $L(\Gamma')$, then $\Lambda'$ is {\it co-amenable} in $\Gamma'$, i.e. there exists 
a $\Gamma'$-invariant state $\Phi:\ell^{\infty}(\Gamma'/\Lambda')\rightarrow\mathbb C$ (see \cite[Proposition 3.5]{AD95}). Let us show that is impossible as well. 

Let $g_1\in\Gamma_1'\setminus\Lambda'$ and $g_2,g_3\in\Gamma_2'\setminus\Lambda'$ such that $g_3^{-1}g_2\not\in\Lambda'$.
Let $S_1$ and $S_2$ be the set of words in $\Gamma_1'\setminus\Lambda'$ and $\Gamma_2'\setminus\Lambda'$ beginning in $\Gamma_1'\setminus\Lambda'$ and in $\Gamma_2'\setminus\Lambda'$, respectively. Then $\Gamma'=S_1\sqcup S_2\sqcup\Lambda'$ and we have
$\Lambda'\subset g_1S_1,g_1S_2\subset S_1,g_2S_1\subset S_2,g_3S_1\subset S_2.$

Now, let $q:\Gamma'\rightarrow\Gamma'/\Lambda'$ be quotient map and define $T_1=q(S_1), T_2=q(S_2)$. Then we have $\Gamma'/\Lambda'=T_1\sqcup T_2\sqcup\{e\Lambda'\}$ and $e\Lambda'\in g_1T_1,g_1T_2\subset T_1,g_2T_1\subset T_2,g_3T_1\subset T_2$.
Moreover, since $g_3^{-1}g_2T_1\subset T_2$, we get that $g_2T_1\cap g_3T_1=\emptyset$. Hence,  $g_2T_1\sqcup g_3T_1\subset T_2$.

For a subset $T\subset\Gamma'/\Lambda'$, let $m(T)=\Phi(1_T)\in [0,1]$. Then $m$ is a finitely additive $\Gamma'$-invariant probability measure on $\Gamma'/\Lambda'$. The relations from the last paragraph therefore imply that $m(e\Lambda')\leqslant m(T_1),m(T_2)\leqslant m(T_1)$ and $2m(T_1)\leqslant m(T_2)$. This would  imply that $m(e\Lambda')=m(T_1)=m(T_2)=0$, contradicting the fact that $m(e\Lambda')+m(T_1)+m(T_2)=m(\Gamma'/\Lambda')=1$.\hfill$\square$

\vskip 0.1in
{\it Proof of Theorem \ref{main}}. Assume that $\Gamma=\Gamma_1\times\Gamma_2\times...\times\Gamma_n$, where   $\Gamma_i=\Gamma_{i,1}*_{\Lambda_i}\Gamma_{i,2}$ is an amalgamated free product group satisfying the hypothesis of Theorem \ref{main}, for every $i\in\{1,2,...,n\}$. We denote by $G_i<\Gamma$ the product of all $\Gamma_j$ with $j\in\{1,2,...,n\}\setminus\{i\}$.

Let $\Gamma\curvearrowright (X,\mu)$ be a free ergodic pmp action. Let $A$ be a Cartan subalgebra of $M=L^{\infty}(X)\rtimes\Gamma$.
For a subset $S\subset\Gamma$, we denote by $e_S$ the orthogonal projection from $L^2(M)$ onto the $\|.\|_2$ closed linear span of $\{L^{\infty}(X)u_g|g\in S\}$.

For $i\in\{1,2,...,n\}$, we  decompose  $M=(L^{\infty}(X)\rtimes G_i)\rtimes\Gamma_i.$ 
By applying Theorem \ref{maintech} we deduce that $A\prec_{M}L^{\infty}(X)\rtimes G_i$. 
Since $A\subset M$ is maximal abelian, it follows that we can find a non-zero projection $p\in A$ and $v\in M$ such that $Ap\subset v(L^{\infty}(X)\rtimes G_i)v^*$. 
By possibly shrinking $p$, we may assume that $\tau(p)=\frac{1}{m}$, for some $m\geqslant 1$. Since $A$ is a Cartan subalgebra we can find unitaries 
$u_1,u_2,..,u_m\in\mathcal N_{M}(A)$ such that $\sum_{j=1}^m u_jpu_j^*=1$. Thus, we get that $A\subset \sum_{j=1}^mu_j(Ap)u_j^*\subset
\sum_{j=1}^mu_jv(L^{\infty}(X)\rtimes G_i)v^*u_j^*$.
By using $\|.\|_2$-approximations, we conclude that for every $\varepsilon>0$ we can find a finite set 
$S\subset\Gamma$ such that $\|x-e_{SG_iS}(x)\|_2\leqslant\varepsilon$, for all $x\in (A)_1$.

Thus, we can find finite sets $S_1,S_2,...,S_n\subset\Gamma$ such that $$\|x-e_{S_iG_iS_i}(x)\|_2\leqslant\frac{1}{n+1},\;\;\text{for all}\;\; x\in (A)_1\;\;\text{and every}\;\; i\in\{1,2,...,n\}.$$

Let $S=\cap_{i=1}^nS_iG_iS_i$. Then $S$ is a finite subset of $\Gamma$ and $\|x-e_S(x)\|_2\leqslant\frac{n}{n+1}$, for every $x\in (A)_1$. Thus, $\|e_S(u)\|_2\geqslant\frac{1}{n+1}$, for every $u\in\mathcal U(A)$. Since $\|e_S(u)\|_2^2=\sum_{g\in S}\|E_{L^{\infty}(X)}(uu_g^*)\|_2^2$, Theorem \ref{corner} gives that $A\prec_{M}L^{\infty}(X)$. Since $A$ and $L^{\infty}(X)$ are  Cartan subalgebras, \cite[Theorem A.1]{Po01} implies that they are unitarily conjugate.\hfill$\square$

\subsection{Applications to W$^*$-superrigidity}
Next, we combine   Theorem \ref{main}  with S. Popa's cocycle superrigidity \cite{Po06a}   to provide a new class of W$^*$-superrigid actions. In particular, we will deduce Corollary \ref{super}.

 A free ergodic pmp action $\Gamma\curvearrowright (X,\mu)$ is called {\it W$^*$-superrgid} if whenever $L^{\infty}(X)\rtimes\Gamma\cong L^{\infty}(Y)\rtimes\Lambda$, for a free ergodic pmp action $\Lambda\curvearrowright (Y,\nu)$, the groups $\Gamma$ and $\Lambda$ are isomorphic and their actions are conjugate. This means that we can find a group isomorphism $\delta:\Gamma\rightarrow\Lambda$ and a measure space isomorphism $\theta:X\rightarrow Y$ such that $\theta(g\cdot x)=\delta(g)\cdot\theta(x)$, for all $g\in\Gamma$ and $\mu$-almost every $x\in X$.

Recall that any orthogonal representation $\pi:\Gamma\rightarrow\mathcal O(\mathcal H_{\mathbb R})$ onto a real Hilbert space $\mathcal H_{\mathbb R}$ gives rise to a pmp action $\Gamma\curvearrowright (X_{\pi},\mu_{\pi})$, called the {\it Gaussian action} associated to $\pi$ (see for instance \cite[Section 2.g]{Fu06}).

\begin{theorem}\label{generalsuper} Let $\Gamma=\Gamma_1*_{\Lambda}\Gamma_{2}$ and $\Gamma'=\Gamma_1'*_{\Lambda'}\Gamma_2'$ be amalgamated free product groups such that $[\Gamma_{1}:\Lambda]\geqslant 2$, $[\Gamma_{2}:\Lambda]\geqslant 3$, $[\Gamma_1':\Lambda']\geqslant 2$ and $[\Gamma_2':\Lambda']\geqslant 3$. Suppose that there exist $g_1,g_2,...,g_n\in\Gamma$ and $g_1',g_2',...,g_n'\in\Gamma'$ such that $\cap_{i=1}^ng_i\Lambda g_i^{-1}=\{e\}$ and $\cap_{i=1}^ng_i'\Lambda'{g_i'}^{-1}=\{e\}$. 

Let $G=\Gamma\times\Gamma'$ and $\pi:G\rightarrow\mathcal O(\mathcal H_{\mathbb R})$ be an orthogonal representation such that

\begin{itemize}
\item the representation $\pi_{|\Gamma}$ has stable spectral gap, i.e. ${\pi_{|\Gamma}}\otimes{\bar{\pi}_{|\Gamma}}$ has spectral gap, and
\item the representation $\pi_{|\Gamma'}$ is  weakly mixing, i.e. ${\pi_{|\Gamma'}}\otimes{\bar{\pi}_{|\Gamma'}}$ has no invariant vectors. 
\end{itemize}
Then any free ergodic pmp action $G\curvearrowright (X,\mu)$ which can be realized as a quotient of the Gaussian action $G\curvearrowright (X_{\pi},\mu_{\pi})$, is W$^*$-superrigid.
\end{theorem}

S. Popa and S. Vaes have very recently proven that the same holds when $\Gamma$ and $\Gamma'$ are icc weakly amenable groups that admit a proper 1-cocycle into a representation with stable spectral gap \cite[Theorem 12.2]{PV11}.

\begin{proof} Denote $M=L^{\infty}(X)\rtimes G$ and let $\Lambda\curvearrowright (Y,\nu)$ be a free ergodic pmp action such that we have an isomorphism $\theta:L^{\infty}(Y)\rtimes\Lambda\rightarrow M$. Then $\theta(L^{\infty}(Y))$ is a Cartan subalgebra of $M$. Thus, by Theorem \ref{main} we can find a unitary $u\in M$ such that $\theta(L^{\infty}(Y))=uL^{\infty}(X)u^*$.

This implies that the actions $G\curvearrowright (X,\mu)$ and $\Lambda\curvearrowright (Y,\nu)$ are orbit equivalent. Therefore, in order to show that the actions are actually conjugate, it suffices to argue that $G\curvearrowright (X,\mu)$ is orbit equivalent superrigid.

Let us show that we can apply \cite[Theorem 1.3]{Po06a} to $G\curvearrowright X$.
Firstly, by Corollary \ref{inner}, $\Gamma$ and $\Gamma'$ have no finite normal subgroup. Thus, $G$ has no finite normal subgroups. Secondly, by \cite[Theorem 1.2]{Fu06} the action $G\curvearrowright X$ is s-malleable. 

Thirdly, consider the unitary representation $\rho:G\curvearrowright L^2(X_{\pi})\ominus\mathbb C1.$ Then $\rho$ is a subrepresentation of $\pi\otimes\sigma$, where $\sigma=\oplus_{n\geqslant 0}\;\pi^{\otimes_n}$. Since $\pi_{|\Gamma}$ has stable spectral gap and $\pi_{|\Gamma'}$ is weakly mixing, the same properties hold for $\rho_{|\Gamma}$ and $\rho_{|\Gamma'}$. Thus, the action $\Gamma\curvearrowright X_{\pi}$ has stable spectral gap and the action $\Gamma'\curvearrowright X_{\pi}$ is weakly mixing.

Thus, we can apply \cite[Theorem 1.3]{Po06a} to deduce that the action $G\curvearrowright X$ is OE superrigid.
\end{proof}

{\it Proof of Corollary \ref{super}}. Note that the Bernoulli action $G\curvearrowright [0,1]^G$ can be identified with the Gaussian action associated to the left regular representation $\lambda:G\rightarrow\mathcal U(\ell^2(G))$. Since $\Gamma$ and $\Gamma'$ are non-amenable, the corollary follows from Theorem  \ref{generalsuper}.

\vskip 0.05in
\begin{remark} 
In \cite[Theorem 1.1]{Ki10}, Y. Kida proved the following: let Mod$^*(S)$ be the extended mapping class group of a surface of genus $g$ with $p$ boundary components. Suppose that $3g+p\geqslant 5$ and $(g,p)\not=(1,2),(2,0)$. Let $\Delta<\text{Mod}^*(S)$ be a finite index subgroup and $A<\Delta$ be an infinite, almost malnormal subgroup (i.e. $hAh^{-1}\cap A$ is finite, for all $h\in\Delta\setminus A$) and denote $\Gamma=\Delta*_{A}\Delta$.
Then any free ergodic pmp action $\Gamma\curvearrowright (X,\mu)$ whose restriction to $A$ is aperiodic is OE-superrigid.

Since $A<\Gamma$ is weakly malnormal, Theorem \ref{main} implies that all such actions of $\Gamma$ are moreover W$^*$-superrigid.
\end{remark}

\subsection{An application to W$^*$-rigidity} In combination with the orbit equivalence rigidity results of N. Monod and Y. Shalom, Theorem \ref{main}  implies the following.

\begin{theorem}\label{ms} Let $\Gamma_1,\Gamma_2,\Gamma_3$ and $\Gamma_4$ be any non-trivial  torsion-free  countable groups and define $\Gamma=(\Gamma_1*\Gamma_2)\times(\Gamma_3*\Gamma_4)$. Let $\Gamma\curvearrowright (X,\mu)$ be a free ergodic pmp action whose restrictions to $\Gamma_1*\Gamma_2$, $\Gamma_3*\Gamma_4$ and any finite index subgroup $\Gamma'<\Gamma$ are also ergodic.

Let $\Lambda\curvearrowright (Y,\nu)$ be an arbitrary free mildly mixing pmp action.

If $L^{\infty}(X)\rtimes\Gamma\cong L^{\infty}(Y)\rtimes\Lambda$, then $\Gamma\cong\Lambda$ and the actions $\Gamma\curvearrowright X$ and $\Lambda\curvearrowright Y$ are conjugate.

\end{theorem}

Following \cite[Definition 1.8]{MS02}, a measure preserving action $\Lambda\curvearrowright (Y,\nu)$ is  called {\it mildly mixing} if for any measurable set $A\subset Y$  and any sequence $\lambda_n\in\Lambda$ with $\lambda_n\rightarrow\infty$, one has $\nu(\lambda_nA\;\Delta\; A)\rightarrow 0$ if and only if $\nu(A)\in\{0,1\}$. 

{\it Proof of Theorem \ref{ms}}. By \cite[Theorem 1.3]{MS02}  the groups $\Gamma_1*\Gamma_2$ and $\Gamma_3*\Gamma_4$ belong to the class $\mathcal C_{\text{reg}}$. Applying \cite[Theorem 1.10]{MS02} then gives the conclusion.\hfill$\square$

\subsection {W$^*$ Bass-Serre rigidity}  We next combine Theorem \ref{main}   with results of A. Alvarez and D. Gaboriau \cite{AG08} to generalize part of \cite[Theorem 7.7]{IPP05} and \cite[Theorem 6.6]{CH08}. 

\begin{theorem}
Let $m,n\geqslant 2$ be integers and $\Gamma_1,\Gamma_2,...,\Gamma_m,\Lambda_1,\Lambda_2,...,\Lambda_n$ be non-amenable groups with vanishing first $\ell^2$-Betti numbers. Define $\Gamma=\Gamma_1*\Gamma_2*...*\Gamma_m$ and $\Lambda=\Lambda_1*\Lambda_2*...*\Lambda_n$. Let 
$\Gamma\curvearrowright (X,\mu)$ and $\Lambda\curvearrowright (Y,\nu)$ be free  pmp actions such that the restrictions $\Gamma_i\curvearrowright X$ and $\Lambda_j\curvearrowright Y$ are ergodic, for every $i\in\{1,2,...,m\}$ and $j\in\{1,2,...,n\}$. 

Let $\theta:L^{\infty}(X)\rtimes\Gamma\rightarrow (L^{\infty}(Y)\rtimes\Lambda)^t$ be an isomorphism, for some $t>0$.

Then $t=1$, $m=n$ and there exists a permutation $\alpha$ of $\{1,2,...,m\}$ such that the actions $\Gamma_i\curvearrowright X$ and $\Lambda_{\alpha(i)}\curvearrowright Y$ are orbit equivalent, for every $i\in\{1,2,...,m\}$.

Moreover, for every $i\in\{1,2,...,m\}$, there exists a unitary element $u_i\in L^{\infty}(Y)\rtimes\Lambda$ such that $\theta(L^{\infty}(X))=u_iL^{\infty}(Y)u_i^*$ and $\theta(L^{\infty}(X)\rtimes\Gamma_i)=u_i(L^{\infty}(Y)\rtimes\Lambda_{\alpha(i)})u_i^*$.
\end{theorem}

{\it Proof.} By Theorem \ref{main}, the II$_1$ factor $L^{\infty}(X)\rtimes\Gamma$ has a unique Cartan subalgebra, up to unitary conjugacy. Thus, we can find a unitary $u\in (L^{\infty}(Y)\rtimes\Lambda)^t$ such that $\theta(L^{\infty}(X))=u(L^{\infty}(Y))^tu^*$. Denoting by $\mathcal R(\Gamma\curvearrowright X)$ the equivalence relation induced by the action $\Gamma\curvearrowright X$, it follows that $\mathcal R(\Gamma\curvearrowright X)\cong \mathcal R(\Lambda\curvearrowright Y)^t$. By using \cite{Ga01} to calculate the first $\ell^2$-Betti number of both sides of this equation (see  the end of the proof of \cite[Theorem 7.7]{IPP05}) we deduce that $t=1$.
Now,  by \cite[Corollary 4.20]{AG08},  non-amenable groups with vanishing first $\ell^2$-Betti number are measurably freely indecomposable. Since $\mathcal R(\Gamma\curvearrowright X)=*_{i=1}^m\mathcal R(\Gamma_i\curvearrowright X)$ and $\mathcal R(\Lambda\curvearrowright Y)=*_{j=1}^n\mathcal R(\Lambda_j\curvearrowright Y)$, by applying \cite[Theorem 5.1]{AG08}, the conclusion follows.\hfill$\square$

\subsection{II$_1$ factors with trivial fundamental group} Theorem \ref{general} also leads to a new class of groups whose actions give rise to II$_1$ factors with trivial fundamental groups.

\begin{theorem}\label{fund}
Let $\Gamma_1$, $\Gamma_2$ be  two finitely generated, countable groups with $|\Gamma_1|\geqslant 2$ and $|\Gamma_2|\geqslant 3$. Denote $\Gamma=\Gamma_1*\Gamma_2$ and let $\Gamma\curvearrowright (X,\mu)$ be any free ergodic pmp action. 

Then the II$_1$ factor $M=L^{\infty}(X)\rtimes\Gamma$ has trivial fundamental group, $\mathcal F(M)=\{1\}$.

\end{theorem}

{\it Proof.} By Theorem \ref{general}, $L^{\infty}(X)\rtimes\Gamma$ has a unique Cartan subalgebra, up to unitary conjugacy. Therefore, we have that $\mathcal F(M)=\mathcal F(\mathcal R(\Gamma\curvearrowright X))$.
Since $\beta_1^{(2)}(\Gamma)\in (0,\infty)$, a well-known result of D. Gaboriau \cite{Ga01} implies that $\mathcal F(\mathcal R(\Gamma\curvearrowright X))=\{1\}.$
\hfill$\square$

\begin{remark}
Theorem \ref{fund} generalizes \cite[Theorem 1.2]{PV08}. Thus, it was shown in \cite{PV08} that the conclusion of Theorem \ref{fund} holds, for instance, if   $\Gamma_1$ is an icc property (T) group and $\Gamma_2$ is an infinite group.
Note that Theorem \ref{fund} fails if the groups involved are not finitely generated. Indeed, by \cite[Theorem 1.1]{PV08} if $\Lambda_1$ is a non-trivial group and $\Lambda_2$ is an infinite amenable group, then $\Gamma=\Lambda_1^{*\infty}*\Lambda_2$ does not satisfy the conclusion of Theorem \ref{fund}. In fact, as shown in \cite{PV08}, there are free ergodic pmp actions $\Gamma\curvearrowright X$ such that $\mathcal F(L^{\infty}(X)\rtimes\Gamma)$ is uncountable.
\end{remark}

\subsection{Absence of Cartan subalgebras} Finally, Theorem \ref{maintech}  allows us to provide a new class of II$_1$ factors without Cartan subalgebras:
\begin{corollary}\label{cormain}
Let $\Gamma=\Gamma_1*_{\Lambda}\Gamma_2$ be an amalgamated free product group such that $[\Gamma_1:\Lambda]\geqslant 2$ and  $[\Gamma_2:\Lambda]\geqslant 3$.  Assume that there exist $g_1,g_2,...,g_n\in\Gamma$ such that $\cap_{i=1}^ng_i\Lambda g_i^{-1}=\{e\}$.

Then $N\bar{\otimes}L(\Gamma)$ does not have a Cartan subalgebra, for any II$_1$ factor $N$.
\end{corollary}
 {\it Proof of Corollary \ref{cormain}}.
Let $N$ be a II$_1$ factor and denote $M=N\bar{\otimes}L(\Gamma)$. Assume by contradiction that $M$ has a Cartan subalgebra $A$. Since $M=N\rtimes\Gamma$, where $\Gamma$ acts trivially on $N$,  Theorem \ref{maintech} implies $A\prec_{M}N$. By taking relative commutants (see \cite[Lemma 3.5]{Va07}) we get that $L(\Gamma)\prec_{M}A'\cap M=A$. Since $A$ is abelian, while $\Gamma$ is non-amenable, we derive a contradiction.\hfill$\square$

\section{Cartan subalgebras of AFP algebras and classification of II$_1$ factors  \\ arising from  free product equivalence relations}

In this section we prove Theorem \ref{cartan} and  Corollary \ref{equirel}.

\subsection{Proof of Theorem \ref{cartan}} Let $A$ be a Cartan subalgebra of $M=M_1*_{B}M_2$. Recall that $B$ is amenable, $pM_1p\not=pBp\not=pM_2p$, for any non-zero projection $p\in B$, and that either

\begin{enumerate}
\item $M_1$ and $M_2$ have no amenable direct summands, or
\item $M$ does not have property $\Gamma$.
\end{enumerate}

We claim that  $M\nprec_{M}M_i$, for any $i\in\{1,2\}$.   Assume by contradiction that $M\prec_{M}M_i$, for some $i\in\{1,2\}$. By Theorem \ref{corner} we can find projections $p\in M,q\in M_i$, a non-zero partial isometry $v\in qMp$ such that $v^*v=p$, and a $*$-homomorphism $\phi:pMp\rightarrow qM_iq$ such that $\phi(x)v=vx$, for all $x\in pMp$.  Since $M$ is a non-amenable factor and $B$ is amenable, we have that $M\nprec_{M}B$.
Thus, by \cite[Remark 3.8]{Va07} we can moreover assume that $\phi(pMp)\nprec_{M_i}B$. 

Then \cite[Theorem 1.1]{IPP05} implies that $\phi(pMp)'\cap qMq\subset qM_iq$. In particular, $q_0:=vv^*\in qM_iq$. From this we get that $q_0Mq_0=q_0M_iq_0$. Let  $j\in\{1,2\}\setminus\{i\}$ and $x\in M_j\ominus B$.  Then  the orthogonal projection of $q_0xq_0$ onto $(L^2(M_i)\ominus L^2(B))\otimes_B (L^2(M_j)\ominus L^2(B))\otimes_B (L^2(M_i)\ominus L^2(B))$ is equal to $(q_0-E_B(q_0))x(q_0-E_B(q_0))$. Since $q_0xq_0\in M_i$, we deduce that $q_0-E_B(q_0)=0$.
Thus,  $q_0\in B$ and  $q_0M_jq_0\subset q_0M_iq_0\cap q_0M_jq_0=q_0Bq_0$. This contradicts our assumption that $q_0M_jq_0\not=q_0Bq_0$.

Next, consider $\tilde M=M*_B(B\bar{\otimes}L(\mathbb F_2))$ and the free malleable deformation $\{\theta_t\}_{t\in\mathbb R}\subset$ Aut$(\tilde M)$. Let $N=\{u_gMu_g^*|g\in\mathbb F_2\}''$.
 Since $\tilde M=N\rtimes\mathbb F_2$, by applying Theorem \ref{pv} we have two cases:

{\bf Case a.} $\theta_t(A)\prec_{\tilde M}N$, for some $t\in (0,1)$.

{\bf Case b.}  $\theta_t(M)$ is amenable relative to $N$ inside $\tilde M$, for any $t\in (0,1)$.

In {\bf Case a}, Theorem \ref{inter} gives that either $A\prec_{M}B$ or $M\prec_{M}M_i$, for some $i\in\{1,2\}$. Since the latter is impossible by the above, the conclusion holds in this case.

To finish the proof it is enough to argue that  {\bf Case b} contradicts  each of the above assumptions (1) and (2).
 Indeed, by applying Theorem \ref{amena} we get that  $M_ip_i$ is amenable relative to $B$, for some non-zero projection $p_i\in\mathcal Z(M_i)$ and some $i\in\{1,2\}$. Since $B$ is amenable, this would imply that either $M_1$ or $M_2$ has an amenable direct summand, contradicting assumption (1). 

Also, by applying Theorem \ref{relamen} we would get that either $M$ has property $\Gamma$, $M\prec_{M}M_i$, for some $i\in\{1,2\}$, or $M$ is amenable relative to $B$ (hence $M$ is amenable and therefore isomorphic to the hyperfinite II$_1$ factor). Since the hyperfinite II$_1$ factor has property $\Gamma$, this contradicts assumption (2).

\begin{remark} Theorem \ref{cartan} requires that $M=M_1*_{B}M_2$ is a factor. Note that when $B$ is a type I von Neumann algebra,  \cite[Theorem 5.8]{HV12} and \cite[Theorem 4.3]{Ue12} provide general conditions which guarantee that $M$ is a factor.
\end{remark}

\vskip 0.1in
\subsection{Proof of Corollary \ref{equirel}} 
Denote $M=L(\mathcal R)$, $M_1=L(\mathcal R_1)$, $M_2=L(\mathcal R_2)$ and $B=L^{\infty}(X)$. Then $M=M_1*_{B}M_2$. Since the restrictions of $\mathcal R_1$ and $\mathcal R_2$ to any set of positive measure have infinite orbits, we get that $pM_1p\not=pBp\not=pM_2p$, for any non-zero projection $p\in B$.

Now, if the restrictions of $\mathcal R_1$ and $\mathcal R_2$ to any set of positive measure are non-hyperfinite, then $M_1$ and $M_2$ have no amenable direct summand \cite{CFW81}. 

Next, let us show that if $\mathcal R$ is strongly ergodic, then $M$ does not have property $\Gamma$. Since the restrictions of $\mathcal R_1$ and $\mathcal R_2$ to any set of positive measure have infinite orbits, \cite[Lemma 2.6]{IKT08} provides $\theta_1\in [\mathcal R_1]$ and $\theta_2,\theta_3\in [\mathcal R_2]$ such that $\theta_1(x)\not=x,\theta_2(x)\not=x,\theta_3(x)\not=x$ and $\theta_2(x)\not=\theta_3(x)$, for $\mu$-almost every $x\in X$. Thus the unitaries $u=u_{\theta_1}\in M_1$, $v=u_{\theta_2}\in M_2$ and $w=u_{\theta_3}\in M_2$ satisfy $E_B(u)=E_B(v)=E_B(w)=E_B(w^*v)=0$. By Lemma \ref{bar} we get that $M'\cap M^{\omega}\subset B^{\omega}$.

Since $\mathcal R$ is strongly ergodic, we have that $M'\cap B^{\omega}=\mathbb C$, which  shows that $M$ does not have property $\Gamma$.

Altogether by applying Theorem \ref{cartan} we deduce that if $A$ is a Cartan subalgebra of $M$, then $A\prec_{M}B$. Hence, by \cite[Theorem A.1]{Po01} it follows that $A$ and $B$ are unitarily conjugate.

Finally, let $\mathcal S$ be a countable measure preserving equivalence relation on a probability space $(Z,\nu)$ and  $\theta:L(\mathcal S)\rightarrow M$ be an isomorphism. Then $\theta(L^{\infty}(Z))$ is a Cartan subalgebra of $M$ and so it must be conjugate to $B$. This shows that the inclusions $L^{\infty}(X)\subset L(\mathcal R)$ and $L^{\infty}(Z)\subset L(\mathcal S)$ are isomorphic, hence $\mathcal R\cong\mathcal S$.
\hfill$\square$

Note that, as one of the referees pointed out, one can alternatively use \cite[Theorem 4.8]{Ue12} to deduce that $M=L(\mathcal R)$ does not have property $\Gamma$.

\begin{remark} 
This proof moreover shows that if $v\in$ H$^2(\mathcal R,\mathbb T)$ is any 2-cocycle, then $L^{\infty}(X)$ is the unique Cartan subalgebra of the II$_1$ factor $L(\mathcal R,v)$, up to unitary conjugacy. Thus, if  $L(\mathcal R,w)\cong L(\mathcal S,v)$, for any ergodic countable measure preserving equivalence relation $\mathcal S$  on a standard probability space $(Y,\nu)$ and any 2-cocycle $w\in$ H$^2(\mathcal S,\mathbb T)$, then $\mathcal R\cong\mathcal S$ and the cocycles $v$ and $w$ are cohomologous. More precisely, there exists an isomorphism of probability spaces $\theta:X\rightarrow Y$ such that $(\theta\times\theta)(\mathcal R)=\mathcal S$ and $[v\circ(\theta\times\theta\times\theta)]=[w]$ in H$^2(\mathcal R,\mathbb T)$ (see \cite{FM77}).  \end{remark}

\section{Normalizers of amenable subalgebras of AFP algebras}
In the first part of this section we prove Theorem \ref{general} and Corollary \ref{corgeneral}, and then deduce Corollary \ref{free}.

\subsection{Proof of Theorem \ref{general}}
For simplicity of notation, we assume that $p=1$, and leave the details of the general case to the reader.
  Let $A\subset M=M_1*_{B}M_2$ be a von Neumann subalgebra that is amenable relative to $B$. Suppose that $P=\mathcal N_{M}(A)''$ satisfies $P'\cap M^{\omega}=\mathbb C1$.

Let $\tilde M=M*_{B}(B\bar{\otimes}L(\mathbb F_2))$ and $\{\theta_t\}_{t\in\mathbb R}\subset$ Aut$(\tilde M)$  the associated free malleable deformation. Let $N=\{u_gMu_g^*|g\in\mathbb F_2\}''$ and recall that $\tilde M=N\rtimes\mathbb F_2$. 
Since $A$ is amenable relative to $B$ and $\theta_t(B)=B\subset N$, we deduce that $\theta_t(A)$ is amenable relative to $N$, for any $t\in\mathbb R$.

By Theorem \ref{pv}  either there exists $t\in (0,1)$ such that $\theta_t(A)\prec_{\tilde M}N$ or else $\theta_t(P)$ is amenable relative to $N$ inside $\tilde M$, for every $t\in (0,1)$.

In the first case, Theorem \ref{inter} gives that either $A\prec_{M}B$ or $P\prec_{M}M_i$, for some $i\in\{1,2\}$.  
 In the second case, Theorem \ref{relamen} implies that either $P\prec_{M}M_i$, for some $i\in\{1,2\}$, or $P$ is amenable relative to $B$ inside $M$. Altogether, the conclusion follows.
\hfill$\square$

\vskip 0.1in
\subsection{Proof of Corollary \ref{corgeneral}} We establish the following more precise version of Corollary \ref{corgeneral}. If
$P\subset pMp$ and $Q\subset M$ are von Neumann subalgebras then we write $P\prec^{s}_{M}Q$ if $Pp'\prec_{M}Q$, for any non-zero projection $p'\in P'\cap pMp$.

\begin{corollary}\label{corgeneral2} Let $(M_1,\tau_1)$, $(M_2,\tau_2)$ be two tracial von Neumann algebras. Let $M=M_1*M_2$ and $A\subset M$ be a diffuse amenable von Neumann subalgebra. Denote  $P=\mathcal N_{M}(A)''$.

Then we can find projections $p_1,p_2,p_3\in \mathcal Z(P)$ satisfying $p_1+p_2+p_3=1$ and 
\begin{enumerate}
\item $Pp_1\prec^{s}_{M}M_1$,
\item $Pp_2\prec^{s}_{M}M_2$, and
\item $Pp_3$ is amenable.
\end{enumerate}

Moreover, if $M_1$ and $M_2$ are factors, then we can find unitary elements $u_1,u_2\in M$ such that $u_1Pp_1u_1^*\subset M_1$ and $u_2Pp_2u_2^*\subset M_2$.
\end{corollary}

\begin{proof} If a non-zero projection $p\in \mathcal Z(P)=P'\cap M$ satisfies $Pp\prec_{M}M_i$, for some $i\in\{1,2\}$, then there exists a non-zero projection $p'\in\mathcal Z(P)p$ such that $Pp'\prec^{s}_{M}M_i$. Thus, in order to get the first part of the conclusion, it suffices to argue that if $p\in \mathcal Z(P)$ is a non-zero projection such that $Pp$ has no amenable direct summand, then either $Pp\prec_{M}M_1$ or $Pp\prec_{M}M_2$.

By Theorem \ref{gammadec} we can find projections $e,f\in\mathcal Z((Pp)'\cap pMp)\cap\mathcal Z((Pp)'\cap (pMp)^{\omega})$ such that 
\begin{itemize}
\item $e+f=p$.
\item $((Pp)'\cap (pMp)^{\omega})e$ is completely atomic and $((Pp)'\cap (pMp)^{\omega})e=((Pp)'\cap (pMp))e$.
\item $((Pp)'\cap (pMp)^{\omega})f$ is diffuse.
\end{itemize}
Since $p\not=0$, we have that either $e\not=0$ or $f\not=0$.

In the first case, let $e_0\in ((Pp)'\cap (pMp)^{\omega})e$ be a minimal non-zero projection.
Then we have that $e_0\in p(P'\cap M^{\omega})p\cap p(P'\cap M)p$ and $e_0(P'\cap M^{\omega})e_0=\mathbb Ce_0$.
Therefore, $Pe_0$ is a von Neumann subalgebra of $e_0Me_0$ such that $(Pe_0)'\cap (e_0Me_0)^{\omega}=\mathbb Ce_0$.

Note that $Pe_0\subset\mathcal N_{e_0Me_0}(Ae_0)''$. Also, we have that $A$ and hence $Ae_0$ is diffuse. By  applying Theorem \ref{general} (in the case $B=\mathbb C$) we deduce that either $Pe_0\prec_{M}M_i$, for some $i\in\{1,2\}$, or $Pe_0$ is amenable. Since $e_0\leqslant p$,  $Pe_0$ cannot be amenable. Thus, we must have that $Pe_0\prec_{M}M_i$ and hence that $Pp\prec_{M}M_i$, for some $i\in\{1,2\}$.

In the second case, we have that $f\in  p(P'\cap M^{\omega})p\cap p(P'\cap M)p$ and that $f(P'\cap M^{\omega})f$ is diffuse. Thus, $Pf$ is a von Neumann subalgebra of $fMf$ such that $(Pf)'\cap (fMf)^{\omega}$ is diffuse.

By applying Theorem \ref{afpgamma}  (with $B=\mathbb C$) we deduce that either $Pf\prec_{M}M_i$, for some $i\in\{1,2\}$, or $Pf_0$ is amenable, for some non-zero projection $f_0\in\mathcal Z((Pf)'\cap fMf)$. Since $f_0\leqslant p$, the latter is impossible. Thus we conclude that $Pp\prec_{M}M_i$, for some $i\in\{1,2\}$, in this case as well.

The moreover part now follows by repeating the proof of \cite[Theorem 5.1 (2)]{IPP05}.
\end{proof}
 
 \subsection{Proof of Corollary \ref{free}} 
 Assume by contradiction that $M=M_1*M_2$ has a Cartan subalgebra $A$. Since $M_1\not=\mathbb C\not=M_2$ and $\text{dim}(M_1)+\text{dim}(M_2)\geqslant 5$, by \cite[Theorem 4.1]{Ue10} there exists a non-zero central projection $z\in M$ such that $Mz$ is a II$_1$ factor without property $\Gamma$, while $M(1-z)$ is completely atomic. In particular, $M$ is not amenable.
 
 To derive a contradiction we treat separately two cases
 
 {\bf Case 1.} $M_1$ and $M_2$ are completely atomic. 
 
 {\bf Case 2.} Either $M_1$ or $M_2$ has a diffuse direct summand.
 
 In the first case, since $\mathcal N_{M}(A)''=M$,  Corollary \ref{corgeneral2} yields projections $p_1,p_2,p_3\in\mathcal Z(M)$ such that $p_1+p_2+p_3=1$, $Mp_1\prec^{s}_{M}M_1$, $Mp_2\prec^{s}_{M}M_2$ and $Mp_3$ is amenable. Since $M_1,M_2$ are completely atomic, it follows that $Mp_1,Mp_2$ are completely atomic. Altogether, we derive that $M$ is amenable, a contradiction.
  
In the second case,  we may assume for instance that $M_1$ has a diffuse direct summand. Hence, there exists a non-zero projection $p\in\mathcal Z(M_1)$ such that $M_1p$ is diffuse. Since $M(1-z)$ is completely atomic, we must have that $p\leqslant z$.

Define $N=(\mathbb Cp+M_1(1-p))\vee M_2$. Then by \cite[Lemma 2.2]{Ue10} we have that $M_1p$ and $pNp$ are free and together generate $pMp$, i.e. $pMp=M_1p*pNp$. We also have that $pNp\not=\mathbb Cp$. Indeed, since $M_2\not=\mathbb C$, there exists a projection $q\in M_2$ with $q\not=0,1$. Then $pqp\in pNp$ and  $pqp=\tau(q)p+p(q-\tau(q))p$. This clearly implies that $pqp\notin\mathbb Cp$.

Now, note that $Az$ is a Cartan subalgebra of $Mz$. Since $Mz$ is a factor and $p\in Mz$, it follows that $pMp$ also has a Cartan subalgebra. Since $Mz$ does not have property $\Gamma$, it follows that $pMp$ does not have property $\Gamma$ as well. On the other hand, since $pMp=M_1p*pNp$ and $M_1p\not=\mathbb Cp\not=pNp$, by applying Theorem \ref{cartan} (2) in the case $B=\mathbb Cp$, we conclude that $pMp$ does not have a Cartan subalgebra. This leads to the desired contradiction.
 \hfill$\square$

\vskip 0.1in

\subsection{Strongly solid von Neumann algebras}\label{sstrongly} Our final aim is to prove Theorem \ref{strongly}. We begin by introducing some terminology motivated by the proof of \cite[Theorem 3.1]{Po03}. 

\begin{definition}\cite{Po03}\label{mixdef} Let $(M,\tau)$ be a tracial von Neumann algebra and $B\subset M$ be a von Neumann subalgebra. We say that the inclusion $B\subset M$ is {\it mixing} if for every $x,y\in M\ominus B$ and any sequence $b_n\in (B)_1$ such that $b_n\rightarrow 0$ weakly we have that $\|E_B(xb_ny)\|_2\rightarrow 0$.
\end{definition}

This notion has been considered in \cite{JS06} and \cite{CJM10}, where several examples of mixing inclusions of von Neuman algebras were exhibited.

\begin{remark} Let  $B\subset M$ be tracial von Neumann algebras.
\begin{enumerate}
\item It is easy to see that the inclusion $B\subset M$ is mixing if and only if the $B$-$B$ bimodule $L^2(M)\ominus L^2(B)$ is mixing in the sense of \cite[Definition 2.3]{PS09}. 
\item In particular, the inclusion $B\subset M$ is mixing whenever the $B$-$B$ bimodule $L^2(M)\ominus L^2(B)$ is isomorphic to a sub-bimodule of $\oplus_{i=1}^{\infty}(L^2(B)\otimes L^2(B))$. This is the case, for instance, if we can decompose $M=B*C$, for some von Neumann subalgebra $C\subset M$ (see the proof of \cite[Lemma 2.2]{Po06b}).
\item Let $\Lambda<\Gamma$ be an inclusion of countable groups. Then the inclusion of group von Neumann algebras $L(\Lambda)\subset L(\Gamma)$ is mixing if and only if $g\Lambda g^{-1}\cap\Lambda$ is finite, for every $g\in\Gamma\setminus\Lambda$ (see \cite[Theorem 3.5]{JS06} and the proof of Corollary \ref{strsolid}).
\item Let $(D,\tau)$ be a tracial von Neumann algebra and $\Gamma\curvearrowright D$ be a mixing trace preserving action. Then the inclusion $L(\Gamma)\subset D\rtimes\Gamma$ is mixing  (see the proof of \cite[Lemma 3.4]{Po03}).
\end{enumerate}
\end{remark}

In order to prove Theorem \ref{strongly} we need two technical lemmas. 

\begin{lemma}\cite{Po03}\label{commutant} Let $(M,\tau)$ be a tracial von Neumann algebra and $B\subset M$ be a von Neumann subalgebra. Assume that the inclusion $B\subset M$ is mixing. Let $A\subset pMp$ be a diffuse von Neumann subalgebra, for some projection $p\in M$, and denote $P=\mathcal N_{pMp}(A)''$. Then we have
\begin{enumerate}
\item If $A\subset B$, then $P\subset B$.
\item If $A\prec_{M}B$, then $P\prec_{M}B$.
\end{enumerate}
\end{lemma}

\begin{proof} For the reader's convenience let us briefly indicate how the lemma follows from \cite{Po03}.

Recall that the {\it quasi-normalizer} of a von Neumann subalgebra $Q\subset M$, denoted $q\mathcal N_{M}(Q)$, consists of those elements $x\in M$ for which we can find $x_1,...,x_n\in M$ such that $xQ\subset\sum_{i=1}^nQx_i$ and $Qx\subset\sum_{i=1}^nx_iQ$ (see \cite[Section 1.4.2]{Po01}). Note that $\mathcal N_{M}(Q)\subset q\mathcal N_{M}(Q)$.

Let $Q\subset rBr$ be a diffuse von Neumann subalgebra, for some projection $r\in B$.
Since the inclusion $B\subset M$ is mixing, the proof of \cite[Theorem 3.1]{Po03} shows that the  quasi-normalizer of $Q$ in $rMr$ is contained in $rBr$ (see also the proof of \cite[Theorem 1.1]{IPP05}).  This fact implies (1).
 
 To prove (2), assume that $A\prec_{M}B$. Then we can find projections $q\in A$, $r\in B$, a non-zero partial isometry $v\in rMq$ and a $*$-homomorphism $\phi:qAq\rightarrow rBr$ such that $\phi(x)v=vx$, for all $x\in qAq$. Since $\phi(qAq)\subset rBr$ is diffuse, the previous paragraph gives that $q\mathcal N_{rMr}(\phi(qAq))\subset rBr$.

Next, let $u\in\mathcal N_{pMp}(A)$. Following the proof of  \cite[Lemma 3.5]{Po03}, let  $z\in A$ be a central projection  
such that $z=\sum_{j=1}^mv_jv_j^*$, for some partial isometries $\{v_j\}_{j=1}^m$ in $pMp$ satisfying $v_j^*v_j\leqslant q$.
We claim that $qzuqz\in qMq$ belongs to the quasi-normalizer of $qAq$. Indeed, we have $$qzuqz(qAq)\subset qzuA=qzAu=qAzu\subset \sum_{j=1}^m(qAv_j)v_j^*u\subset \sum_{j=1}^m(qAq)v_j^*u$$ and similarly $(qAq)qzuqz\subset\sum_{j=1}^muv_j(qAq)$.

Now, it is clear that if $x\in q\mathcal N_{qMq}(qAq)$, then $vxv^*\in q\mathcal N_{rMr}(\phi(qAq))$. By combining the last two paragraphs we derive that $vqzuqzv^*\in rBr$. Since the central projections $z$ of the desired form approximate arbitrarily well the central support of $q$, we deduce that $vquqv^*\in rBr$. 
Thus, $vuv^*\in rBr$, for all $u\in\mathcal N_{pMp}(A)$. Hence $vPv^*\subset rBr$ and so we conclude that $P\prec_{M}B$.
\end{proof}

\begin{lemma}\label{ultrapower} Let $(M,\tau)$ be a tracial von Neumann algebra and $B\subset M$ be a von Neumann subalgebra. Assume that the inclusion $B\subset M$ is mixing. 

Let  $P\subset pMp$ be a separable von Neumann subalgebra, for some projection $p\in M$, and $\omega$ be a free ultrafilter on $\mathbb N$. Assume that $P'\cap (pMp)^{\omega}$ is diffuse and $P'\cap (pMp)^{\omega}\prec_{M^{\omega}}B^{\omega}$. 

Then $P\prec_{M}B$.
\end{lemma}

\begin{proof} We first prove the conclusion under the additional assumption that $P'\cap pMp=\mathbb Cp$. We assume for simplicity that $p=1$, the general case being treated similarly.
Denote $P_{\omega}=P'\cap M^{\omega}$ and let $\{y_n\}_{n\geqslant 1}$ be a $\|.\|_2$ dense sequence in $(P)_1$.

Since $P_{\omega}\prec_{M^{\omega}}B^{\omega}$, we can find $a_1,a_2,...,a_n,b_1,b_2,...,b_n\in M^{\omega}$ and $\delta>0$ such that \begin{equation}\label{bomega} \sum_{i=1}^n\|E_{B^{\omega}}(a_iub_i)\|_2^2>\delta,\;\;\;\text{for all}\;\;\; u\in\mathcal U(P_{\omega}).\end{equation}
For every $i\in\{1,2,...,n\}$, write $a_i=(a_{i,k})_k$ and $b_i=(b_{i,k})_k$, for some $a_{i,k},b_{i,k}\in M$.

{\bf Claim 1.} There exists $k\in\mathbb N$ such that \begin{equation}\label{20} \sum_{i=1}^n\|E_{B^{\omega}}(a_{i,k}ub_{i,k})\|_2^2\geqslant \delta,\;\;\;\text{for all}\;\;\; u\in\mathcal U(P_{\omega}).\end{equation}

{\it Proof of Claim 1.}
Suppose that the claim is false and fix $k\in\mathbb N$. Then  there is a unitary $u_k\in P_{\omega}$ such that $\sum_{i=1}^n\|E_{B^{\omega}}(a_{i,k}u_kb_{i,k})\|_2^2<\delta$. Write $u_k=(u_{k,l})_l$, where $u_{k,l}\in\mathcal U(M)$. Then the last inequality rewrites as $\lim_{l\rightarrow\omega}\sum_{i=1}^n\|E_B(a_{i,k}u_{k,l}b_{i,k})\|_2^2<\delta$. 
Also, we have that $\lim_{l\rightarrow\omega}\|[u_{k,l},y_j]\|_2=\|[u_k,y_j]\|_2=0$, for all $j\geqslant 1$. It altogether follows that we can find $l\in\mathbb N$ such that   $U_k:=u_{k,l}$ satisfies $\sum_{i=1}^n\|E_{B}(a_{i,k}U_kb_{i,k})\|_2^2<\delta$ and $\sum_{j=1}^k\|[U_k,y_j]\|_2\leqslant\frac{1}{k}$.

It is then clear that the unitary $U=(U_k)_k$ belongs to $P_{\omega}$ and satisfies $\sum_{i=1}^n\|E_{B^{\omega}}(a_iUb_i)\|_2^2\leqslant\delta$. This contradicts inequality \ref{bomega}.\hfill$\square$

We next use an idea of S. Vaes (see the proof of \cite[Theorem 3.1]{Io11a}). 

Denote by $\mathcal K$ the $\|.\|_2$ closure of the linear span of the set $\{axb|a,b\in M,x\in B^{\omega}\ominus B\}$. Then $\mathcal K$ is a Hilbert subspace of $L^2(M^{\omega})$ that is an $M$-$M$ bimodule. Denote by $e$ the orthogonal projection from $L^2(M^{\omega})$ onto $\mathcal K$. 

Since $P_{\omega}$ is diffuse we can find a unitary $u\in P_{\omega}$ such that $\tau(u)=0$. 
Since $E_M(u)\in P'\cap M$ and $P'\cap M=\mathbb C1$, it follows that $E_M(u)=\tau(E_M(u))1=0$. 
 
 Let $\xi=e(u)$. We claim that $\xi\not=0$. Let $k\in\mathbb N$ as in Claim 1 and  $\eta=\sum_{i=1}^na_{i,k}^*E_{B^{\omega}}(a_{i,k}ub_{i,k})b_{i,k}^*$. Note that $E_B(E_{B^\omega}(a_{i,k}ub_{i,k}))=E_B(a_{i,k}ub_{i,k})=E_B(E_M(a_{i,k}ub_{i,k}))=E_B(a_{i,k}E_M(u)b_{i,k})=0$. Thus $E_{B^{\omega}}(a_{i,k}ub_{i,k})\in B^{\omega}\ominus B$, for all $i\in\{1,2,...,n\}$, hence $\eta\in\mathcal K$.
On the other hand, inequality \ref{20} rewrites as $\langle u,\eta\rangle\geqslant\delta$. Combining the last two facts gives that $\xi\not=0$.

Since $\mathcal K$ is an $M$-$M$ bimodule and $u$ commutes with $P$ it follows that $y\xi=\xi y$, for all $y\in P$. Thus $\langle y\xi y^*,\xi\rangle=\|\xi\|_2^2>0$, for all $y\in\mathcal U(P)$. To finish the proof we use a second claim. 

{\bf Claim 2.} Let $v_n, w_n\in (M)_1$ be two sequences such that $\|E_B(a_2^*v_na_1)\|_2\rightarrow 0$, for all $a_1,a_2\in M$. Then for all $\xi_1,\xi_2\in \mathcal K$ we have that $\langle v_n\xi_1w_n,\xi_2\rangle\rightarrow 0$, as $n\rightarrow\infty$.

{\it Proof of Claim 2.} It suffices to prove the conclusion for $\xi_1$ and $\xi_2$ of the form $\xi_1=a_1x_1b_1$ and $\xi_2=a_2x_2b_2$, for some $a_1,a_2,b_1,b_2\in M$ and $x_1,x_2\in (B^{\omega}\ominus B)_1$. In this case,  we have $$|\langle v_n\xi_1w_n,\xi_2\rangle|=|\tau (x_2^*a_2^*v_na_1x_1b_1w_nb_2^*)|\leqslant \|E_{B^{\omega}}(a_2^*v_na_1x_1b_1w_nb_2^*)\|_2.$$

Since the inclusion $B\subset M$ is mixing, we have $E_{B^{\omega}}(cxd)=0$, for all $c,d\in M\ominus B$ and 
$x\in B^{\omega}\ominus B$. Thus
$E_{B^{\omega}}(a_2^*v_na_1x_1b_1w_nb_2^*)=E_B(a_2^*v_na_1)x_1E_B(b_1w_nb_2^*)$. In combination with the last inequality this implies that $|\langle v_n\xi_1w_n,\xi_2\rangle\leqslant\|E_B(a_2^*v_na_1)\|_2\rightarrow 0$.\hfill$\square$

Now, if the conclusion $P\prec_{M}B$ is false, then we can find a sequence of unitary elements $y_n\in P$ such that 
$\|E_B(a_2^*y_na_1)\|_2\rightarrow 0$, for all $a_1,a_2\in M$. Claim 2 then implies that $\langle y_n\xi y_n^*,\xi\rangle\rightarrow 0$, contradicting the fact that $\langle y_n\xi y_n^*,\xi\rangle=\|\xi\|_2^2>0$, for all $n$.
This finishes the proof of Lemma \ref{ultrapower} under the additional assumption that $P'\cap pMp=\mathbb Cp$. 

In general, assume again for simplicity that $p=1$. Then we can find projections $\{p_n\}_{n\geqslant 0}\in P'\cap M$ such that $p_0\in\mathcal Z(P'\cap M)$ and $(P'\cap M)p_0$ is diffuse, $p_n\in P'\cap M$ is a minimal projection, for all $n\geqslant 1$, and  $\sum_{n\geqslant 0}p_n=1$. Since $P_{\omega}\prec_{M^{\omega}}B^{\omega}$ we can find $n$ such that $p_n\not=0$ and $p_nP_{\omega}p_n\prec_{M^{\omega}}B^{\omega}$.
To derive the conclusion, we treat separately two cases. 

Firstly, assume that $n=0$. Since $((Pp_0)'\cap p_0Mp_0)^{\omega}\subset (Pp_0)'\cap (p_0Mp_0)^{\omega}=p_0P_{\omega}p_0$ and $p_0P_{\omega}p_0\prec_{M^{\omega}}B^{\omega}$, it  easily follows that $(Pp_0)'\cap p_0Mp_0\prec_{M}B$. Since $(Pp_0)'\cap p_0Mp_0=(P'\cap M)p_0$ is diffuse, Lemma \ref{commutant} readily gives that $Pp_0\prec_{M}B$ and hence $P\prec_{M}B$.

Secondly, suppose that  $n\geqslant 1$. Since $p_n\in P'\cap M$ is a minimal projection we get that $(Pp_n)'\cap p_nMp_n=\mathbb Cp_n$. Also, we have that $(Pp_n)'\cap (p_nMp_n)^{\omega}=p_nP_{\omega}p_n$ is diffuse and satisfies $(Pp_n)'\cap (p_nMp_n)^{\omega}\prec_{M^{\omega}}B^{\omega}$. By applying the first part of the proof to the subalgebra $Pp_n\subset p_nMp_n$ we deduce that $Pp_n\prec_{M}B$ and hence that $P\prec_{M}B$.
\end{proof}

{\it Proof of Theorem  \ref{strongly}}.  Since the inclusions $B\subset M_1$, $B\subset M_2$ are mixing, it follows easily that the inclusion $B\subset M$ is mixing. We claim that the inclusion $M_i\subset M$ is also mixing, for $i\in\{1,2\}$. 

To this end, let $j\in\{1,2\}$ with $j\not=i$. Let $b_n\in (M_i)_1$ be a sequence such that $b_n\rightarrow 0$ weakly.  The claim is equivalent to showing that $\|E_{M_i}(x^*b_ny)\|_2\rightarrow 0$, for all $x,y\in M\ominus M_i$.
We may assume that $x,y$ are of the following form: $x=x_1x_2...x_m$ and $y=y_1y_2...y_n$, where   $x_1\in M_i$, $x_2\in M_j\ominus B$, $x_3\in M_i\ominus B...$ and $y_1\in M_i,y_2\in M_j\ominus B,y_3\in M_i\ominus B$..., for  some integers $m,n\geqslant 2$. We may also assume that $\|x_k\|\leqslant 1$ and $\|y_l\|\leqslant 1$, for all $1\leqslant k\leqslant m$ and $1\leqslant l\leqslant n$.

A simple computation shows  that $E_{M_i}(x^*b_ny)=E_{M_i}(x_m^*...x_3^*E_B(x_2^*E_B(x_1^*b_ny_1)y_2)y_3....y_n)$. Thus, we get that
$\|E_{M_i}(x^*b_ny)\|_2\leqslant\|E_B(x_2^*E_B(x_1^*b_ny_1)y_2)\|_2$. Since $b_n\rightarrow 0$ weakly, we have that $E_B(x_1^*b_ny_1)\rightarrow 0$ weakly. Since $x_2,y_2\in M_j\ominus B$ and the inclusion $B\subset M_j$ is mixing, it follows that $\|E_B(x_2^*E_B(x_1^*b_ny_1)y_2)\|_2\rightarrow 0$. This proves that $\|E_{M_i}(x^*b_ny)\|_2\rightarrow 0$ and implies the claim.

Now, to show that $M$ is strongly solid, fix a diffuse amenable von Neumann subalgebra
 $A\subset M$  and denote $P=\mathcal N_{M}(A)''$.
Suppose by contradiction that $P$ is not amenable and let $z\in \mathcal Z(P)$ be the largest projection such that $Pz$ is amenable.
Then $p=1-z\not=0$. 

By Theorem \ref{gammadec} we can find projections $e,f\in\mathcal Z((Pp)'\cap pMp)\cap\mathcal Z((Pp)'\cap (pMp)^{\omega})$ such that 
\begin{itemize}
\item $e+f=p$.
\item $((Pp)'\cap (pMp)^{\omega})e$ is completely atomic and $((Pp)'\cap (pMp)^{\omega})e=((Pp)'\cap (pMp))e$.
\item $((Pp)'\cap (pMp)^{\omega})f$ is diffuse.
\end{itemize}
Since $p\not=0$, we have that either
 $e\not=0$ or
 $f\not=0$.

In  the first case, let $e_0\in ((Pp)'\cap (pMp)^{\omega})e$ be a minimal non-zero projection.
Then we have that $e_0\in p(P'\cap M^{\omega})p\cap p(P'\cap M)p$ and $e_0(P'\cap M^{\omega})e_0=\mathbb Ce_0$.
Therefore, $Pe_0$ is a von Neumann subalgebra of $e_0Me_0$ such that $(Pe_0)'\cap (e_0Me_0)^{\omega}=\mathbb Ce_0$.
Note that $Pe_0\subset\mathcal N_{e_0Me_0}(Ae_0)''$. 
Theorem \ref{general} implies that either 
  $Ae_0\prec_{M}B$,
 $Pe_0\prec_{M}M_i$, for some $i\in\{1,2\}$, or 
$Pe_0$ is amenable relative to $B$.
Moreover if, $Ae_0\prec_{M}B$, then since the inclusion $B\subset M$ is mixing, Lemma \ref{commutant} gives that $Pe_0\prec_{M}B$.

In the second case, 
we have that $f\in  p(P'\cap M^{\omega})p\cap p(P'\cap M)p$ and that $f(P'\cap M^{\omega})f$ is diffuse. Thus, $Pf$ is a von Neumann subalgebra of $fMf$ such that $(Pf)'\cap (fMf)^{\omega}$ is diffuse. 
By applying Theorem \ref{afpgamma} to the subalgebra $Pf$ of $fMf$, we get that either $(Pf)'\cap (fMf)^{\omega}\prec_{M^{\omega}}B^{\omega}$, $Pf\prec_{M}M_i,$ for some $i\in\{1,2\}$, or $Pf_0$ is amenable relative to $B$, for some non-zero projection $f_0\in\mathcal Z(P'\cap M)f$. 
Moreover, if $(Pf)'\cap (fMf)^{\omega}\prec_{M^{\omega}}B^{\omega}$  then since $(Pf)'\cap (fMf)^{\omega}$ is diffuse, Lemma \ref{ultrapower} implies that $Pf\prec_{M}B$.

Altogether, since $e_0\leqslant p$, $f\leqslant p$ and $B\subset M_1\cap M_2$, we get that either  $Pp\prec_{M}M_i$, for some $i\in\{1,2\}$, or  $Pg$ is amenable relative to $B$, for some non-zero projection $g\in\mathcal Z(P)p$.
Since $B$ is amenable, the second condition implies that $Pp$ has an amenable  direct summand, which contradicts the maximality of $z$.

In order to finish the proof, assume that $Pp\prec_{M}M_i$, for some $i\in\{1,2\}$. 
Since $P'\cap M\subset P$, it follows that we can find  non-zero projections $p_0\in Pp$, $q\in M_i$, a partial isometry $v\in M$ such that $v^*v=p_0$ and $vv^*\leqslant q$, and a $*$-homomorphism $\phi:p_0Pp_0\rightarrow qM_iq$ such that $\phi(x)v=vx$, for all $x\in p_0Pp_0$. Since $\phi(p_0Pp_0)\subset qM_iq$ is a diffuse subalgebra and the inclusion $M_i\subset M$ is mixing, Lemma \ref{commutant} gives that $\phi(p_0Pp_0)'\cap qMq\subset qM_iq$ and thus $vv^*\in M_i$.

Hence, after replacing $P$ with $uPu^*$, for some unitary $u\in M$, we may assume that $p_0\in M_i$ and $p_0Pp_0\subset p_0M_ip_0$. Next, we can find a non-zero projection $p_1\in p_0Pp_0$ and partial isometries $v_1,v_2,...,v_n\in P$ such that $v_i^*v_i=p_1$, for all $i\in\{1,2,...,n\}$, and  $p'=\sum_{i=1}v_iv_i^*$ is a central projection of $P$. Since $p_1Pp_1\subset p_1M_ip_1$, there exists an embedding $\theta:Pp'\rightarrow\mathbb M_n(p_1M_ip_1)$.  

Since $M_i$ is  strongly solid, \cite[Proposition 5.2]{Ho09} gives that $\mathbb M_n(p_1M_ip_1)$ is also strongly solid. Since the inclusion $Ap'\subset Pp'$ is regular and $Ap'$ is a diffuse amenable von Neumann algebra, we deduce that $Pp'$ is amenable. Since $p'p\not=0$ (as we have $0\not=p_1\leqslant p\wedge p'$) we again get a contradiction with the maximality of $z$. This completes the proof of the theorem.
\hfill$\square$

We end with several consequences of Theorem \ref{strongly}.

\begin{corollary}\label{stro1}
Let $(M_1,\tau_1)$ and $(M_2,\tau_2)$ be strongly solid von Neumann algebras.

Then $M=M_1*M_2$ is strongly solid.
\end{corollary}

\begin{corollary}\label{stro2}
Let $(M_1,\tau_1),(M_2,\tau_2),...,(M_n,\tau_n)$ be tracial amenable von Neumann algebras with a common von Neumann subalgebra $B$ such that ${\tau_1}_{|B}={\tau_2}_{|B}=...={\tau_n}_{|B}$. Assume that the inclusions $B\subset M_1,B\subset M_2,...,B\subset M_n$ are mixing. Denote $M=M_1*_{B}M_2*_{B}...*_{B}M_n$.

Then $M$ is strongly solid.
\end{corollary}
{\it Proof.} Since the inclusions $B\subset M_1,B\subset M_2,...,B\subset M_n$ are mixing, it is easy to see that the inclusion $B\subset M_1*_{B}M_2*_{B}...*_{B}M_i$ is mixing, for all $i\in\{1,2,...,n\}$. The conclusion then follows by using induction and Theorem \ref{strongly}.
\hfill$\square$

Corollary \ref{stro2} provides two new classes of strongly solid von Neumann algebras.

\begin{corollary}\label{strsolid} Let $\Gamma_1,\Gamma_2,...,\Gamma_n$ be countable amenable groups with a common subgroup $\Lambda$. Assume that  $g\Lambda g^{-1}\cap\Lambda$ is finite, for every $g\in(\cup_{i=1}^n\Gamma_i)\setminus\Lambda$.
Denote  $\Gamma=\Gamma_1*_{\Lambda}\Gamma_2*_{\Lambda}...*_{\Lambda}\Gamma_n$.

Then  $L(\Gamma)$ is  strongly solid.
\end{corollary}

{\it Proof}. We claim that the inclusion $L(\Lambda)\subset L(\Gamma_i)$  is mixing, for every $i\in\{1,2,...,n\}$. 

To this end, let $b_n\in (L(\Lambda))_1$ be a sequence  converging weakly to $0$. We aim to show that $\|E_{L(\Lambda)}(xb_ny)\|_2\rightarrow 0$, for every $x,y\in L(\Gamma_i)\ominus L(\Lambda)$. By Kaplansky's density theorem we may assume that $x=u_h$ and $y=u_k$, for some $h,k\in\Gamma_i\setminus\Lambda$.
Then the set $F=\{g\in\Lambda|hgk\in\Lambda\}$ is finite. Since $b_n\rightarrow 0$ weakly we get that $$\|E_{L(\Lambda)}(u_hb_nu_h)\|_2^2=\sum_{g\in F}|\tau(b_nu_g^*)|^2\rightarrow 0.$$

Corollary \ref{stro2} now implies that $L(\Gamma)=L(\Gamma_1)*_{L(\Lambda)}L(\Gamma_2)*_{L(\Lambda)}....*_{L(\Lambda)}L(\Gamma_n)$ is strongly solid.
\hfill$\square$

 Corollary \ref{strsolid} generalizes the main result of \cite{Ho09}, where the same statement is proven under the additional assumption that for every $i\in\{1,2,...,n\}$ we can decompose $\Gamma_i=\Upsilon_i\rtimes\Lambda$, for some abelian group $\Upsilon_i$.

\begin{corollary} Let   $\Gamma$ be a countable amenable group and $(D_1,\tau_1),(D_2,\tau_2),...,(D_n,\tau_n)$ be tracial amenable von Neumann algebras. Let $\Gamma\curvearrowright^{\sigma_1} (D_1,\tau_1), \Gamma\curvearrowright^{\sigma_2} (D_2,\tau_2),...,\Gamma\curvearrowright^{\sigma_n} (D_n,\tau_n)$ be mixing trace preserving actions. Denote $D=D_1*D_2*...*D_n$ and endow $D$ with its natural trace $\tau$. Consider the free product action $\Gamma\curvearrowright^{\sigma} (D,\tau)$ given by $${\sigma(g)}(x_1x_2...x_n)={\sigma_1}(g)(x_1)\sigma_2(g)(x_2)...\sigma_n(g)(x_n),\;\;\;\text{for}\;\;x_1\in D_1,x_2\in D_2,...,x_n\in D_n.$$ 
Then $M=D\rtimes\Gamma$ is strongly solid.

\end{corollary}

\begin{proof} Denote $M_i=D_i\rtimes\Gamma$. Since the action $\Gamma\curvearrowright (D_i,\tau_i)$ is mixing, the inclusion $L(\Gamma)\subset M_i$ is mixing, for all $1\leqslant i\leqslant n$. Since $\Gamma$ as well as $D_1, D_2,...,D_n$ are amenable, we have that $M_1,M_2,...,M_n$ are amenable. Since $M=M_1*_{L(\Gamma)}M_2*...*_{L(\Gamma)}M_n$, the conclusion follows from Corollary \ref{stro2}.
\end{proof}

\vskip 0.7in

\begin{center}
{\boldmath\Large\bf  Appendix: spectral gap for inclusions of von Neumann algebras} 
\vspace{2ex}

{\sc Adrian Ioana and Stefaan Vaes\footnote{KU~Leuven, Department of Mathematics, Leuven (Belgium), stefaan.vaes@wis.kuleuven.be. Supported by Research Programme G.0639.11 of the Research Foundation~-- Flanders (FWO) and KU~Leuven BOF research grant OT/13/079.}}
\end{center}

\vspace{2ex}

Let $(M,\tau)$ be a von Neumann algebra equipped with a faithful normal tracial state. Let $P \subset M$ be a von Neumann subalgebra. In \cite[Section 2]{Po09}, Popa introduced the following two different notions of spectral gap for the inclusion $P \subset M$.
\begin{itemize}
\item[(a)] $P \subset M$ has \emph{spectral gap} if every net of unit vectors $\xi_i \in L^2(M)$ that asymptotically commutes with $P$, meaning that $\lim_i \|x \xi_i - \xi_i x\|_2 = 0$ for all $x \in P$, must lie asymptotically in $L^2(P' \cap M)$, namely $\lim_i \|\xi_i - E_{P' \cap M}(\xi_i)\|_2 = 0$.
\item[(b)] $P \subset M$ has \emph{$w$-spectral gap} if every net $\xi_i \in (M)_1$ in the unit ball of $M$ that asymptotically commutes with $P$, meaning that $\lim_i \|x \xi_i - \xi_i x\|_2 = 0$ for all $x \in P$, must lie asymptotically in $P' \cap M$, namely $\lim_i \|\xi_i - E_{P' \cap M}(\xi_i)\|_2 = 0$.
\end{itemize}
Here, $E_{P' \cap M}$ denotes the conditional expectation of $M$ onto $P' \cap M$, or its extension as the orthogonal projection of $L^2(M)$ onto $L^2(P' \cap M)$.

In \cite[Remark 2.2]{Po09}, the subtle difference between spectral gap and $w$-spectral gap is explained: concrete examples of inclusions without spectral gap, but yet having $w$-spectral gap are given, and the analogy with the difference between strong ergodicity and spectral gap for a probability measure preserving group action $\Gamma \actson (X,\mu)$ is explained, yielding the following example. Let $\Gamma=\F_n$ be a free group, for $n \geq 2$, and let $\Gamma\curvearrowright (X,\mu)$ be a measure preserving action on a standard probability space that is strongly ergodic but does not have spectral gap (see \cite[Example 2.7]{Sc81}). Denote $A=L^{\infty}(X)$, $M=A\rtimes\Gamma$ and $P=L(\Gamma)$. Since $\Gamma$ is not inner amenable and $\Gamma \actson (X,\mu)$ is strongly ergodic, it follows that $P \subset M$ has $w$-spectral gap. On the other hand, $P \subset M$ does not have spectral gap. Indeed,  let $\xi_n\in L^2(A)\ominus \C 1$ be a sequence of unit vectors such that $\|u_g\xi_n-\xi_nu_g\|_2\rightarrow 0$, for all $g\in\Gamma$. Let $x\in P$ and write $x={\sum_{g\in\Gamma}x_gu_g}$, where $x_g\in \C$. Then
$$\|x\xi_n-\xi_n x\|_2^2=\sum_{g\in\Gamma}|x_g|^2\;\|u_g\xi_n-\xi_n u_g\|_2^2 \quad\text{for all}\;\; n \; .$$
Since $\sum_{g\in\Gamma}|x_g|^2=\|x\|_2^2<\infty$, it follows that $\|x\xi_n-\xi_n x\|_2\rightarrow 0$.

Finally note that if $M$ is a II$_1$ factor and $P = M$, then both notions of spectral gap are equivalent by \cite[Theorem 2.1]{Co76}.

In the proof of Theorem \ref{relamen} above, the following technical property is needed.
This property sits, a priori, in between spectral gap and $w$-spectral gap.

\begin{itemize}
\item[(c)] Every net of unit vectors $\xi_i \in L^2(M) \ot \ell^2(\N)$ that asymptotically commutes with $P \ot 1$ and that is asymptotically subtracial, meaning that $\limsup_i \|(a \ot 1)\xi_i\|_2 \leq \|a\|_2$ and $\limsup_i \|\xi_i(a \ot 1)\|_2 \leq \|a\|_2$ for all $a \in M$, must lie asymptotically in $L^2(P' \cap M) \ot \ell^2(\N)$.
\end{itemize}

In the theorem below, we prove that this property (c) is equivalent to $w$-spectral gap. 

The difference between spectral gap and $w$-spectral gap arises when there do exist nontrivial unit vectors $\xi_i \in L^2(M)$ that asymptotically commute with $P$, but when these unit vectors necessarily have their support in a smaller and smaller corner of $M$ with the operator norm of $\xi_i$ becoming larger and larger. If now $\xi_i = \sum_k a_{i,k} \ot \delta_k$ is a net in $L^2(M) \ot \ell^2(\N)$ as in (c), then the subtraciality assumption guarantees that the small supports of the $a_{i,k}$ are evenly spread over $M$. Using a maximality argument, it should be possible to glue the $a_{i,k}$ together into a bounded net in $M$ that asymptotically commutes with $P$, as in \cite[Remark 2.4]{PP84}. We follow a slightly different approach, taking random linear combinations $\sum_k \zeta_k a_{i,k}$ with $\zeta_k \in \T$, very much inspired by \cite[Proof of Lemma 4.3]{Ha84}.

{\bf Theorem.} %\label{thm.main}
 {\it Let $(M,\tau)$ be a von Neumann algebra with a faithful normal tracial state. Let $P \subset M$ be a von Neumann subalgebra. The following two conditions are equivalent.
\begin{enumerate}
\item The inclusion $P \subset M$ does not have $w$-spectral gap: there exists a net $u_i \in (M)_1$ in the unit ball of $M$ satisfying $\lim_i \|x u_i - u_i x\|_2 = 0$ for all $x \in P$ and satisfying \linebreak $\liminf_i \|u_i - E_{P' \cap M}(u_i)\|_2 > 0$.
\item There exist a Hilbert space $H$ and a net of vectors $\xi_i \in L^2(M) \ot H$ satisfying the following properties:
\begin{itemize}
\item $\lim_i \|(x \ot 1) \xi_i - \xi_i(x \ot 1)\|_2 = 0$ for all $x \in P$,
\item $\liminf_i \|\xi_i - p_{L^2(P' \cap M) \ot H}(\xi_i)\|_2 > 0$,
\item $\limsup_i \|(a \ot 1)\xi_i\|_2 \leq \|a\|_2$ and $\limsup_i \|\xi_i(a \ot 1)\|_2 \leq \|a\|_2$ for all $a \in M$.
\end{itemize}
\end{enumerate}}

\begin{proof}
It is obvious that 1 implies 2 by taking $H = \C$ and $\xi_i = u_i$.

Assume that 2 holds. Write $\cP = p_{L^2(P' \cap M) \ot H}$ and $\mu_i = \cP(\xi_i)$. Obviously $(x \ot 1) \mu_i = \mu_i (x \ot 1)$ for all $x \in P$. Also,
$$\|(a \ot 1) \mu_i \|_2 = \|\cP((E_{P' \cap M}(a^* a)^{1/2} \ot 1) \xi_i)\|_2 \quad\text{for all}\;\; a \in M \;\;\text{and all}\;\; i \; .$$
Therefore, also $\limsup_i \|(a \ot 1)\mu_i\|_2 \leq \|a\|_2$ for all $a \in M$, and similarly with $\|\mu_i(a \ot 1)\|_2$. Replacing $\xi_i$ by $(\xi_i - \mu_i)/2$, we may from now on moreover assume that $\cP(\xi_i) = 0$ for all $i$.

Define the normal positive functionals $\om_i, \om'_i \in M_*$ given by $\om_i(a) = \langle (a \ot 1)\xi_i,\xi_i \rangle$ and $\om'_i(a) = \langle \xi_i (a \ot 1), \xi_i \rangle$. After passage to a subnet, we may assume that $\om_i \recht \om$ and $\om'_i \recht \om'$ weakly$^*$, where $\om,\om' \in M^*$ are nonzero positive functionals satisfying $\om,\om' \leq \tau$. Convex combinations of the functionals $\om_i$, resp.\ $\om'_i$, then converge in norm to $\om$, resp.\ $\om'$. Such convex combinations are canonically implemented by vectors in $H \ot \ell^2(\N)$. Therefore, replacing $H$ by $H \ot \ell^2(\N)$, we may assume that $\lim_i \|\om_i - \om\|_1 = \lim_i \|\om'_i - \om'\|_1 = 0$.

Write $\om = \tau(\,\cdot\, T)$ and $\om_i = \tau(\,\cdot\, T_i)$, where $T,T_i$ are positive elements in $L^1(M)$. We have $0 \leq T \leq 1$ and $\lim_i \|T_i -T\|_1 = 0$. Denote by $p_i \in M$ the spectral projection of $T_i$ corresponding to the interval $[0,2]$. We claim that $\lim_i \om_i(1-p_i) = 0$. Write $q_i = 1-p_i$. Then, $q_i T_i q_i \geq 2 q_i$. Also, $q_i T q_i \leq q_i$ because $T \leq 1$. Therefore, $q_i(T_i - T)q_i \geq q_i$. Since $\|q_i (T_i - T) q_i\|_1 \recht 0$, it follows that $\|q_i\|_1 \recht 0$. Then also $\|q_i T q_i\|_1 \recht 0$, so that $\|q_i T_i q_i\|_1 \recht 0$, proving the claim.

By the claim, we have that $\lim_i \|\xi_i - (p_i \ot 1) \xi_i\|_2 = 0$. We similarly define $p'_i$ and get that $\lim_i \|\xi_i - (p_i \ot 1) \xi_i (p'_i \ot 1)\|_2=0$. Replacing $\xi_i$ by $p_i \xi_i p'_i /2$, we now have the following properties.
\begin{itemize}
\item $\lim_i \|(x \ot 1) \xi_i - \xi_i(x \ot 1)\|_2 = 0$ for all $x \in P$,
\item $\lim_i \|\cP(\xi_i)\|_2 = 0$ and $\liminf_i \|\xi_i\|_2 > 0$,
\item $\|(a \ot 1)\xi_i\|_2 \leq \|a\|_2$ and $\|\xi_i(a \ot 1)\|_2 \leq \|a\|_2$ for all $i$ and all $a \in M$.
\end{itemize}

Define $\delta > 0$ such that $\liminf_i \|\xi_i\|_2^2 > 4 \delta$. Fix a finite subset $\cF \subset P$ satisfying $\cF = \cF^*$ and fix $\eps > 0$. We will construct an element $W \in M$ satisfying $\|W\|^2 \leq 8/\delta$, $E_{P' \cap M}(W) = 0$, $\|W\|_2^2 > \delta$ and $\|x W - W x\|_2 \leq \eps$ for all $x \in \cF$. Once we have done this for arbitrary finite $\cF \subset P$ and $\eps > 0$ (with the same fixed $\delta$ from the beginning), the net in 1 indeed exists.

Every vector $\xi \in L^2(M) \ot H$ belongs to $L^2(M) \ot H_0$ for some separable subspace $H_0 \subset H$. We can therefore find a \emph{sequence} of vectors $\xi_n \in L^2(M) \ot \ell^2(\N)$ satisfying
\begin{itemize}
\item $\lim_n \|(x \ot 1) \xi_n - \xi_n(x \ot 1)\|_2 = 0$ for all $x \in \cF$,
\item $\lim_n \|\cP(\xi_n)\|_2 = 0$ and $\liminf_n \|\xi_n\|_2^2 > 4\delta$,
\item $\|(a \ot 1)\xi_n\|_2 \leq \|a\|_2$ and $\|\xi_n(a \ot 1)\|_2 \leq \|a\|_2$ for all $n$ and all $a \in M$.
\end{itemize}
By the last property, we have $\xi_n = \sum_k a_{n,k} \ot \delta_k$ where $a_{n,k} \in M$ satisfies $\sum_k a_{n,k} a_{n,k}^* \leq 1$ and $\sum_k a_{n,k}^* a_{n,k} \leq 1$. Approximating $\xi_n$ by a finite sum, we may assume that for every $n$, there are only finitely many nonzero $a_{n,k}$.

Define $\cK \subset L^2(M) \ot \ell^2(\N)$ as the linear span of all $a \ot \delta_k$. Define the standard probability space $X = \T^\N$ as an infinite product of tori equipped with the Lebesgue measure. Write $\cM = L^\infty(X) \ovt M$ and define the linear map
$$\Theta : \cK \recht \cM : (\Theta(a \ot \delta_k))(\zeta) = \zeta_k a \quad\text{for all}\;\; a \in M, k \in \N, \zeta \in X \; .$$
Write $B = L^\infty(X) \ovt (P' \cap M)$. By a direct computation, using that the functions $\zeta \mapsto \zeta_i$ are orthogonal for distinct $i$, we get that
\begin{itemize}
\item $\Theta((x \ot 1) \xi (y \ot 1)) = (1 \ot x) \Theta(\xi) (1 \ot y)$ for all $x,y \in M$, $\xi \in \cK$,
\item $\|\Theta(\xi)\|_2 = \|\xi\|_2$ for all $\xi \in \cK$,
\item $E_B(\Theta(\xi)) = \Theta(\cP(\xi))$ for all $\xi \in \cK$.
\end{itemize}
Finally we prove that, if $\xi \in \cK$ is given by a finite sum $\xi = \sum_k a_k \ot \delta_k$ satisfying $\sum_k a_k a_k^* \leq 1$ and $\sum_k a_k^* a_k \leq 1$, then
\begin{equation}\label{eq.estim}
\tau\bigl( |\Theta(\xi)|^4 \bigr) \leq 2 \; .
\end{equation}
To prove \eqref{eq.estim}, first note that
$$|\Theta(\xi)|^4 (\zeta) = \sum_{i,j,k,l} \overline{\zeta_i} \, \zeta_j \, \overline{\zeta_k} \, \zeta_l \, a_i^* a_j a_k^* a_l \; .$$
The integral over $\zeta$ is zero, except in two cases: the case where $i=j$ and $k=l$, and the case where $i = l$ and $j=k$. Counting `twice' the case where $i=j=k=l$, we find that
$$E_{1 \ot M}\bigl( |\Theta(\xi)|^4 \bigr) = \Bigl(\sum_i a_i^* a_i \Bigr)^2 + \sum_{i} a_i^* \Bigl(\sum_j a_j a_j^* \Bigr) a_i - \sum_i (a_i^* a_i)^2 \; .$$
Using that $\sum_j a_j a_j^* \leq 1$ and $\sum_i a_i^* a_i \leq 1$, it follows that $E_{1 \ot M}\bigl( |\Theta(\xi)|^4 \bigr) \leq 2$. Applying $\tau$, we find that \eqref{eq.estim} holds.

Define the sequence $U_n \in \cM$ given by $U_n = \Theta(\xi_n)$. Fix a free ultrafilter $\omega$ on $\N$. We claim that $(U_n)$ defines an element in $L^2(\cM^\omega)$. For every $n \in \N$ and $\lambda > 0$, denote by $p_{n,\lambda}$ the spectral projection of $|U_n|$ corresponding to the interval $[0,\lambda]$. Write $q_{n,\lambda} = 1-p_{n,\lambda}$. Using \eqref{eq.estim} in the last inequality, we get that
$$\lambda^2 \, \|U_n \, q_{n,\lambda}\|_2^2 = \lambda^2 \, \tau(|U_n|^2 q_{n,\lambda}) \leq \tau(|U_n|^4 q_{n,\lambda}) \leq \tau(|U_n|^4) \leq 2 \; .$$
It follows that $(U_n p_{n,\lambda})_n$ belongs to $\cM^\omega$ and converges in $\|\,\cdot\,\|_2$ to $U = (U_n) \in L^2(\cM^\omega)$ as $\lambda \recht \infty$. We still have that $\tau(|U|^4) \leq 2$. The other properties of the sequence $(\xi_n)$ now translate to: $U$ commutes with $1 \ot \cF$, $\|U\|_2^2 > 4 \delta$ and $E_{B^\omega}(U) = 0$.

Put $\lambda = \sqrt{2/\delta}$ and denote by $p_\lambda$ the spectral projection of $|U|$ corresponding to the interval $[0,\lambda]$. Write $q_\lambda = 1 - p_\lambda$. Then, $p_\lambda \in \cM^\omega \cap (1 \ot \cF)'$ and, as above,
$$\|U q_\lambda\|_2^2 \leq \frac{2}{\lambda^2} = \delta \; .$$
Define $V = U p_\lambda$. Then, $V \in \cM^\omega \cap (1 \ot \cF)'$ and $\|V\| \leq \lambda$. Also,
\begin{align*}
\|V - E_{B^\omega}(V)\|_2^2 &= \|V\|_2^2 - \|E_{B^\omega}(V)\|_2^2 = \|V\|_2^2 - \|E_{B^\omega}(U) - E_{B^\omega}(U q_\lambda)\|_2^2 \\
&= \|V\|_2^2 - \|E_{B^\omega}(U q_\lambda)\|_2^2 = \|U\|_2^2 - \|U q_\lambda\|_2^2 - \|E_{B^\omega}(U q_\lambda)\|_2^2 \\
& \geq \|U\|_2^2 - 2 \delta > 2 \delta \; .
\end{align*}
Represent $V$ by a sequence $V = (V_n)$ with $V_n \in \cM$ and $\|V_n\| \leq \|V\| \leq \lambda$. Since $V$ commutes with $1 \ot \cF$, we fix $n$ close enough to $\omega$ such that
\begin{align}
& \sum_{x \in \cF} \|(1 \ot x) V_n - V_n (1 \ot x)\|_2^2 < \frac{\eps^2 \delta}{\lambda^2} \quad\text{and} \label{eq.good1}\\
& \|V_n - E_B(V_n)\|_2^2 > 2 \delta \; .\label{eq.good2}
\end{align}
From now on, we view $V_n$ as a measurable function from $X$ to $M$, with $\|V_n(\zeta)\| \leq \lambda$ for all $\zeta \in X$. Define the sets
\begin{align*}
X_0 &= \Bigl\{ \zeta \in X \Bigm| \sum_{x \in \cF} \|x V_n(\zeta) - V_n(\zeta) x\|_2^2 < \eps^2 \Bigr\} \; , \\
X_1 &= \Bigl\{ \zeta \in X \Bigm| \|V_n(\zeta) - E_{P' \cap M}(V_n(\zeta))\|_2^2 > \delta\Bigr\} \; .
\end{align*}
Because of \eqref{eq.good1}, we have that $\mu(X_0) > 1 - \delta / \lambda^2$. We claim that also $\mu(X_1) > \delta / \lambda^2$. Indeed, if $\mu(X_1) \leq \delta / \lambda^2$, using that $\|V_n(\zeta)- E_{P' \cap M}(V_n(\zeta))\|_2 \leq \|V_n(\zeta)\|_2 \leq \|V_n(\zeta)\| \leq \lambda$ for all $\zeta \in X$, it follows that
\begin{align*}
\|V_n & - E_B(V_n)\|_2^2 \\ &= \int_{X_1} \|V_n(\zeta) - E_{P' \cap M}(V_n(\zeta))\|_2^2 \; d\mu(\zeta) + \int_{X \setminus X_1} \|V_n(\zeta) - E_{P' \cap M}(V_n(\zeta))\|_2^2 \; d\mu(\zeta) \\
&\leq \mu(X_1) \lambda^2 + \mu(X\setminus X_1) \delta \leq 2 \delta \; .
\end{align*}
This contradicts \eqref{eq.good2} and the claim follows. But then $\mu(X_0 \cap X_1) > 0$ and we pick $\zeta \in X_0 \cap X_1$. Define $W = V_n(\zeta) - E_{P' \cap M}(V_n(\zeta))$. By construction, we have that $\|W\|^2 \leq (2 \lambda)^2 = 8/\delta$, $E_{P' \cap M}(W) = 0$, $\|W\|_2^2 > \delta$ and $\|x W - W x\|_2 < \eps$ for all $x \in \cF$.
\end{proof}

{\bf Corollary.} 
{\it Let $(M,\tau)$ and $(N,\tau)$ be von Neumann algebras with a faithful normal tracial state. Let $P \subset M$ be a von Neumann subalgebra. If $P \subset M$ has $w$-spectral gap, then also $P \ot 1 \subset M \ovt N$ has $w$-spectral gap.}

\begin{proof}
It suffices to put $H = L^2(N)$ and to view unitary operators in $M \ovt N$ as vectors in $L^2(M) \ot H$.
\end{proof}

\end{document}